\documentclass[reqno,11pt]{amsart}
\usepackage{amssymb,amsmath,amsthm}
\usepackage{a4wide}
\usepackage{mathrsfs}
\usepackage{color}
\usepackage{esint}
\usepackage{stmaryrd}
\usepackage[colorlinks,linkcolor=blue,anchorcolor=blue,citecolor=blue]{hyperref}
\newtheorem{theorem}{Theorem}

\newtheorem{lemma}[theorem]{Lemma}
\newtheorem{proposition}[theorem]{Proposition}
\newtheorem{remark}[theorem]{Remark}

\def\beq{\begin{equation}}
\def\eeq{\end{equation}}
\def\beqs{\begin{equation*}}
\def\eeqs{\end{equation*}}
\def\bal#1\eal{\begin{align}#1\end{align}}
\def\bals#1\eals{\begin{align*}#1\end{align*}}
\def\bsp#1\esp{\begin{split}#1\end{split}}

\def\d{{\mathrm{d}}}

\let\e=\varepsilon
\numberwithin{equation}{section}
\numberwithin{theorem}{section}

\allowdisplaybreaks[3]% 其中1，2，3，4代表执行跨页强度。4 最强。

\begin{document}
\date{}
\title[Diffusive Limit of  the Non-cutoff Vlasov--Maxwell--Boltzmann system]{
\large{Diffusive Limit of the  Vlasov--Maxwell--Boltzmann System without Angular Cutoff
}}
%for the Full Range of Collision Potentials

\author{Yuan Xu, Fujun Zhou, Weihua Gong and Weijun Wu}

\address[Yuan Xu]{School of Mathematics, South China University of Technology, Guangzhou 510640, China}
\email{yuanxu2019@163.com}

\address[Fujun Zhou, Corresponding author]{School of Mathematics, South China University of Technology, Guangzhou 510640, China}
\email{fujunht@scut.edu.cn}

\address[Weihua Gong]{School of Mathematics, Guangdong University of Education, Guangzhou 510303, China}
\email{weihuagong2020@163.com}

\address[Weijun Wu]{School of Network Security, Guangdong Police College, Guangzhou 510230, China}
\email{scutweijunwu@qq.com}

\begin{abstract}
Diffusive limit of the non-cutoff Vlasov--Maxwell--Boltzmann system
in perturbation framework still remains open.
By employing a new weight function and making full use of the anisotropic dissipation property of the non-cutoff linearized Boltzmann operator,
we solve this problem with some novel treatments for non-cutoff potentials $  \gamma > \max\{-3, -\frac{3}{2}-2s\}$, including both strong angular singularity $\frac{1}{2} \leq s <1$ and weak angular singularity $0 < s < \frac{1}{2}$.
Uniform estimate with respect to the Knudsen number $\varepsilon\in (0,1]$ is established globally in time, which eventually leads to the global existence of solutions to the non-cutoff Vlasov--Maxwell--Boltzmann system as well as hydrodynamic limit to the two-fluid incompressible Navier--Stokes--Fourier--Maxwell system with Ohm's law.
The indicators $\gamma > \max\{-3, -\frac{3}{2}-2s\}$ and $0<s<1$ in this paper cover all ranges that can be achieved by the previously established global solutions
to the non-cutoff Vlasov--Maxwell--Boltzmann system in perturbation framework.
 \\[3mm]
{\em Mathematics Subject Classification (2020)}:  35Q20; 35Q83
\\[1mm]
{\em Keywords}: Vlasov--Maxwell--Boltzmann system; non-cutoff potentials; hydrodynamic limit;  global solutions; incompressible Navier--Stokes--Fourier--Maxwell system.
\end{abstract}

%\subjclass{35Q20; 76P05}
%
%\keywords{Vlasov--Poisson--Boltzmann system; hydrodynamic limit; Navier--Stokes--Fourier--Poisson system; optimal time decay rate}
%
\maketitle

\tableofcontents

\section{Introduction}

\subsection{Description of the Problem}
\hspace*{\fill}

As one of the most fundamental model for dynamics of dilute charged particles, the Vlasov--Maxwell--Boltzmann (VMB, for short) system describes particles
interacting with themselves through collisions and with their self-consistent electromagnetic field.
Despite its importance, hydrodynamic limit to the fluid equations has not been completely solved so far.
In this paper, we study diffusive limit of the two-species VMB system without angular cutoff
\begin{equation}\label{GG1}
\left\{
	\begin{array}{ll}
	\displaystyle \varepsilon \partial_t F^\varepsilon_{+}+v\cdot\nabla_xF^\varepsilon_{+}
     + (\varepsilon E^\varepsilon + v \times B^\varepsilon )\cdot\nabla_v F^\varepsilon_{+} =\frac{1}{\varepsilon}Q(F^\varepsilon_{+},F^\varepsilon_{+})+\frac{1}{\varepsilon}Q(F^\varepsilon_{+},F^\varepsilon_{-}),    \\[2mm]
		
	\displaystyle \varepsilon \partial_tF^\varepsilon_{-}+v\cdot\nabla_xF^\varepsilon_{-}
    - (\varepsilon E^\varepsilon + v \times B^\varepsilon )\cdot\nabla_v F^\varepsilon_{-}
    =\frac{1}{\varepsilon}Q(F^\varepsilon_{-},F^\varepsilon_{-})+\frac{1}{\varepsilon}Q(F^\varepsilon_{-},F^\varepsilon_{+}),   \\[2mm]
		
    \displaystyle \partial_t E^\varepsilon - \nabla_x\times B^\varepsilon = -\frac{1}{\varepsilon^2}\int_{\mathbb{R}^3} v (F^{\varepsilon}_{+}-F^{\varepsilon}_{-})\d v, \\ [2.5mm]
		
	\displaystyle \partial_t B^\varepsilon+\nabla_x \times E^\varepsilon=0, \\[2mm]

    \displaystyle \nabla_x \cdot E^\varepsilon= \frac{1}{\varepsilon}\int_{\mathbb{R}^3} (F^\varepsilon_{+}-F^\varepsilon_{-})\d v,
    \quad \nabla_x \cdot B^\varepsilon=0.

\end{array}\right.%\tag{+}
\end{equation}
%The model describes the dynamics of dilute ionized plasmas consisting of two-species particles (e.g., electrons and ions) under the influence of binary collisions and their %self-consistent electromagnetic field. More precisely,
  Here, $F_{\pm}^{\varepsilon}(t,x,v)\geq 0$ are the density distribution functions for the ions $(+)$ and electrons $(-)$, respectively, at time $t \geq 0$, position $x=(x_1, x_2, x_3)\in \mathbb{R}^{3}$ and velocity $v =(v_1, v_2, v_3) \in \mathbb{R}^{3}$.
The electric field $E^\varepsilon(t, x)$ and the magnetic field $B^\varepsilon(t, x)$, which are generated by the motion of the particles in the plasma itself, are governed by the Maxwell equations in (\ref{GG1}).
The positive parameter $\varepsilon\in (0,1]$ is the Knudsen number, which equals to the ratio of the mean free path to the macroscopic length scale.
The initial data of the VMB system (\ref{GG1}) are imposed as
\begin{align}
  F_{\pm}^\varepsilon(0,x,v)= F_{\pm,0}^\varepsilon(x,v), \quad E^\varepsilon(0,x)=E^\varepsilon_0(x), \quad B^\varepsilon(0,x)=B^\varepsilon_0(x)
\end{align}
with the compatibility conditions
\begin{align*}
\nabla_x \cdot E_0^\varepsilon= \frac{1}{\varepsilon}\int_{\mathbb{R}^3} (F^\varepsilon_{+,0}-F^\varepsilon_{-,0})\d v,
    \quad \nabla_x \cdot B_0^\varepsilon=0.
\end{align*}
The Boltzmann collision operator $Q(F, G)$ is a bilinear operator, defined as
\begin{align*}
Q(F, G)(v) :=
\iint_{\mathbb{R}^{3} \times \mathbb{S}^2} B(v-v_*,\sigma)[F(v^{\prime})G(v_{*}^{\prime})-F(v)G(v_*)]\d v_* \d \sigma.
\end{align*}
Here, $v$, $v_*$ and $v^{\prime}$, $v_*^{\prime}$ are the velocities of a pair of particles before and after collision, respectively. They are connected through the formulas
$$
v^{\prime}=\frac{v+v_*}{2}+\frac{|v-v_*|}{2} \sigma, \quad v_*^{\prime}=\frac{v+v_*}{2}-\frac{|v-v_*|}{2} \sigma, \quad \sigma \in \mathbb{S}^2.
$$

The Boltzmann collision kernel $B(v-v_*, \sigma)$ depends only on the relative velocity $|v-v_*|$ and on the deviation angle $\theta$ given by $\cos \theta=\langle\sigma,(v-v_*) /|v-v_*|\rangle$, where $\langle\cdot, \cdot\rangle$ is the usual dot product in $\mathbb{R}^3$. As in \cite{GS2011}, without loss of generality, we suppose that $B(v-v_*, \sigma)$ is supported on $\cos \theta \geq 0$.
Throughout this paper, the collision kernel $B(v-v_*, \sigma) $ is supposed to satisfy the following assumptions:
\begin{itemize}
\item
$B(v-v_*, \sigma)$ takes the product form in its argument, that is,
$$
B(v-v_*, \sigma)=\Phi(|v-v_*|) b(\cos \theta)
$$
with non-negative functions $\Phi$ and $b$.

\item
The angular function $\sigma \rightarrow b(\langle\sigma,(v-v_*) /|v-v_*|\rangle)$ is not integrable on $\mathbb{S}^2$, i.e.
$$
\int_{\mathbb{S}^2} b(\cos \theta) \d \sigma=2 \pi \int_0^{\pi / 2} \sin \theta b(\cos \theta) \d \theta=\infty,
$$
 and there are two positive constants $c_b>0$ and $0<s<1$ such that
$$
\frac{c_b}{\theta^{1+2 s}} \leq \sin \theta b(\cos \theta) \leq \frac{1}{c_b \theta^{1+2 s}}.
$$

  \item
  The kinetic function $z \rightarrow \Phi(|z|)$ satisfies
$$
\Phi(|z|)=C_{\Phi}|z|^\gamma
$$
for some positive constant $C_{\Phi}>0$, where the exponent $\gamma>-3$ is determined by the intermolecular interactive mechanism.
\end{itemize}

It is convenient to call ``soft potentials'' when $-3 < \gamma <-2s$ and ``hard potentials'' when $\gamma \geq -2s$. The current work is restricted to the case of
\begin{align*}
  \gamma > \max\left\{-3, -\frac{3}{2}-2s\right\} \; \text{ and }\; 0<s<1.
\end{align*}
Moreover, $0 < s < \frac{1}{2}$ is usually called weak angular singularity, and $ \frac{1}{2}\leq  s < 1$ is called strong angular singularity \cite{DL2013}.
Recall that when the intermolecular interactive potential takes the inverse power law in the form of $U(|x|)=|x|^{-(\ell-1)}$ with $2 < \ell < \infty$, the collision kernel $B(v-v_*, \sigma)$ in three space dimensions satisfies the above assumptions with $\gamma=\frac{\ell-5}{\ell-1}$ and $s=\frac{1}{\ell-1}$. Note that $\gamma \to -3$ and $s \to 1$ as $\ell \to 2$ in the limiting case, for which the grazing collisions between particles are dominated and the Boltzmann collision term has to be replaced by the classical Landau collision term for the Coulomb potential, cf. \cite{V}.

Much significant progress has been made on the global existence and decay estimates of solutions to the VMB system. Firstly, for the cutoff VMB system, that is, the collision kernel $B(v-v_*, \sigma)$ inside the Boltzmann collision operator $Q(F,G)$ is assumed to be integrable and taken the form
$$
 B(v-v_*, \sigma)=|v-v_*|^{\gamma}b(\cos\theta)\;\;\text{ with } \; 0\leq b(\cos\theta)\leq C|\cos\theta|,
$$
there have been many results. In the framework of renormalized solutions, DiPerna--Lions developed global solutions to the Boltzmann equation \cite{Lion1989Ann}, Vlasov--Poisson--Boltzmann system \cite{Lions1994-1, Lions1994-2} and Vlasov--Maxwell system \cite{Lion1989CPAM} for general large initial data, but without global solutions for the VMB system. Recently, Ars\'{e}nio--Saint-Raymond \cite{AS2019} obtained global renormalized  solutions for general large initial data, including both cutoff and non-cutoff potentials.
In the framework of perturbation around the global Maxwellian, a significant breakthrough was achieved by \cite{Guo2003}, wherein Guo constructed classical solutions to the cutoff VMB system  in $\mathbb{T}^3$  for the hard sphere case $\gamma=1$ by employing the robust energy method.
Subsequently, Strain--Duan \cite{Strain2006CMP, DS2011} established global existence and optimal time decay rate of classical solutions in $\mathbb{R}^3$ for  the hard sphere potential $\gamma=1$. Ultimately, global existence for the full range of soft potentials $-3<\gamma<0$ was solved by \cite{DLYZ2017}. In addition, Li--Yang--Zhong \cite{LYZ2016-1} utilized spectral analysis method to study global solutions and optimal time decay rate for the hard sphere case $\gamma=1$.
Secondly, for the non-cutoff VMB system, some research results have also emerged. Following the breakthrough of the non-cutoff Boltzmann equation by Gressman--Strain \cite{GS2011} and Alexandre--Morimoto--Ukai--Xu--Yang \cite{AMUXY2011AA, AMUXY2012JFA}  independently, Duan--Liu--Yang--Zhao \cite{DLYZ2013} constructed global solutions to
the non-cutoff VMB system for the range of soft potentials $\max\{-3, -\frac{3}{2}-2s\} <\gamma< -2s$ with strong angular singularity $\frac{1}{2} \leq s < 1$.
Later on, the analysis of global classical solutions in this perturbation framework was extended to the soft potential case $\max\{-3, -\frac{3}{2}-2s\} <\gamma< -2s$ with weak angular singularity $0 < s < \frac{1}{2}$ by \cite{FLLZ2018}.

Hydrodynamic limit of the VMB system has also been the focus of extensive research efforts.
In a recent significant advancement, Ars\'{e}nio--Saint-Raymond \cite{AS2019} justified various limits (depending on the scalings) towards incompressible viscous electro--magneto--hydrodynamics in the framework of renormalized solutions. Among these limits, the most singular one is the limit to the two-fluid incompressible Navier--Stokes--Fourier--Maxwell (NSFM, for short) system with Ohm's law.  This vital work opened up a series of subsequent progress on diffusive limit of the VMB system, particularly in perturbation framework near the global Maxwellian. Jiang--Luo \cite{JL2019} justified the incompressible NSFM limit of classical solutions in perturbation framework for the cutoff hard sphere case $\gamma=1$ in $\mathbb{R}^3$, and also considered Hilbert expansion \cite{Ca} for the cutoff hard potential case $0\leq \gamma\leq 1$ \cite{JLZ2023ARMA}
in $\mathbb{T}^3$. More recently, Jiang--Lei \cite{JL2023ARXIV} investigated the incompressible NSFM limit for the cutoff VMB system in the range of soft potentials $-1\leq \gamma<0$, and for the non-cutoff VMB system in the range of soft potentials $\max\{-3, -\frac{3}{2}-2s\}<\gamma<-2s$ with strong angular singularity $\frac{1}{2}\leq s<1$.
Additionally, Yang--Zhong \cite{YZ24} demonstrated the above limit for the cutoff hard sphere case $\gamma=1$ using spectral analysis method, which also provided insights into the convergence rate.
In addition to incompressible NSFM limit of the VMB system mentioned above, Jang \cite{J2009} studied the incompressible Vlasov--Navier--Stokes--Fourier limit for the cutoff hard sphere case $\gamma=1$ in $\mathbb{T}^3$. Furthermore, recent studies have explored the compressible Euler--Maxwell limit of one-species VMB system in hyperbolic regime, as presented by \cite{DYY2023M3AS} for the cutoff hard sphere potential $\gamma=1$ and \cite{LLXZ2023ARXIV} for the non-cutoff soft potentials $\max\{-3, -\frac{3}{2}-2s\}<\gamma<-2s$ and $0<s<1$, respectively.
For more related works on hydrodynamic limit of relevant kinetic equations, we refer to \cite{BGL91, BGL93, BU91, CIP, DEL, GS, GZW-2021, guo2006, guo2010cmp, GJJ, GX2020, JXZ2018, LW2023, LYZ2020, LM, Saint} and  the references cited therein.

As shown by the research progress mentioned above, for diffusive limit of the non-cutoff VMB system in perturbation framework,
the existing result is only within the range of soft potentials $\max\{-3, -\frac{3}{2}-2s\}<\gamma<-2s$ with strong angular singularity
$\frac{1}{2}\leq s<1$, given
by \cite{JL2023ARXIV}. The more challenging case of soft potentials $\max\{-3, -\frac{3}{2}-2s\}<\gamma<-2s$ with weak angular singularity $0<s<\frac{1}{2}$ still remains open. The purpose of this paper is to fill in this gap and investigate diffusive limit of the non-cutoff VMB system \eqref{GG1} uniformly for the potential range
$\gamma>\max\{-3, -\frac{3}{2}-2s\}$ and $0<s<1$. By introducing a new weight function and some novel treatments, we establish the uniform weighted estimates with respect to the Knudsen number $\varepsilon\in (0,1]$ globally in time, which eventually justifies the convergence to the two-fluid incompressible NSFM system with Ohm's law.
Consequently, this paper completely solves diffusive limit of the non-cutoff VMB system in perturbation framework, in the sense that the indicators
$\gamma> \max\{-3, -\frac{3}{2}-2s\}$ and $0<s<1$ cover all ranges that can be achieved by the previously established global solutions
to the non-cutoff VMB system in perturbation framework.

\subsection{Reformulation of the Problem}
\hspace*{\fill}

We investigate diffusive limit of the two-species VMB system (\ref{GG1}) in the framework of perturbation around the global Maxwellian.
It is well-known that $[\mu(v),\mu(v)]$ forms the global equilibrium of the two-species VMB system (\ref{GG1}), where
$$
\mu=\mu(v) :=(2 \pi)^{-3 / 2} \mathrm{e}^{-|v|^{2} / 2}
$$
represents the global Maxwellian.
By writing
$$
F_{\pm}^{\varepsilon}(t,x,v)=\mu+\varepsilon \mu^{1/2}f_{\pm}^{\varepsilon}(t,x,v)
$$
with the fluctuation $f_{\pm}^{\varepsilon}$,
the VMB system
(\ref{GG1}) is transformed into the following equivalent perturbed VMB system
\begin{align}\label{rVMB}
\left\{\begin{array}{l}
\displaystyle \partial_t f^{\varepsilon}+\frac{1}{\varepsilon}v \cdot \nabla_x f^{\varepsilon}+ \frac{1}{\varepsilon} q_0\left(\varepsilon E^{\varepsilon}+v \times B^{\varepsilon}\right) \cdot \nabla_v f^{\varepsilon}-\frac{1}{\varepsilon}\left(E^{\varepsilon} \cdot v\right) \mu^{1/2} q_1 +\frac{1}{\varepsilon^2} L f^{\varepsilon} \\ [2mm]
%----------------------------------
\displaystyle \qquad\qquad\qquad\qquad\qquad\qquad\quad=\frac{q_0}{2}\left(E^{\varepsilon} \cdot v\right) f^{\varepsilon}+\frac{1}{\varepsilon} \Gamma\left(f^{\varepsilon}, f^{\varepsilon}\right), \\ [2mm]
%----------------------------------
\displaystyle \partial_t E^{\varepsilon}-\nabla_x \times B^{\varepsilon}=-\frac{1}{\varepsilon} \int_{\mathbb{R}^3} f^{\varepsilon} \cdot q_1 v \mu^{1/2} \d v, \\ [3mm]
%----------------------------------
\displaystyle \partial_t B^{\varepsilon}+\nabla_x \times E^{\varepsilon}=0, \\ [1mm]
%-----------------------------------
\displaystyle \nabla_x \cdot E^{\varepsilon}=\int_{\mathbb{R}^3} f^{\varepsilon} \cdot q_1 \mu^{1/2} \d v, \quad \nabla_x \cdot B^{\varepsilon}=0,
\end{array}\right.
\end{align}
where $f^{\varepsilon}=[f_{+}^{\varepsilon}, f_{-}^{\varepsilon}]$ represents the vector in $\mathbb{R}^2$ with the components $f_{\pm}^{\varepsilon}$, $q_0=\mathrm{diag}(1, -1)$ represents the $2 \times 2$ diagonal matrix, $q_1=[1, -1]$, the linearized Boltzmann collision operator $Lf^{\varepsilon}$ and the nonlinear collision term $\Gamma(f^{\varepsilon}, f^{\varepsilon})$ are respectively defined by
$$Lf^{\varepsilon}:=[L_{+}f^{\varepsilon}, L_{-}f^{\varepsilon}],\qquad
\Gamma(f^{\varepsilon}, f^{\varepsilon}):=\left[\Gamma_{+}(f^{\varepsilon}, f^{\varepsilon}), \Gamma_{-}(f^{\varepsilon}, f^{\varepsilon})\right].
$$
For any given $f=\left[f_{+}, f_{-}\right]$ and $g=\left[g_{+}, g_{-}\right]$,
define
\begin{align*}
L_{\pm}f&:=-\mu^{-1/2}\left\{2Q\left(\mu^{1/2}f_{\pm}, \mu\right)+Q\left(\mu,\mu^{1/2}(f_{\pm}+f_{\mp})\right)\right\},  \\
\Gamma_{\pm}(f,g)&:=\mu^{-1/2}\left\{Q\left(\mu^{1/2}f_{\pm},\mu^{1/2}g_{\pm}\right)+Q\left(\mu^{1/2}f_{\pm},\mu^{1/2}g_{\mp}\right)\right\}.
\end{align*}

It is well-known that the operator $L$ is non-negative with null space
$$
\mathcal{N}(L):=\mathrm{span}\left\{[1,0]\mu^{1/2},[0,1]\mu^{1/2},[v_{i},v_{i}]\mu^{1/2}(1\leq i\leq 3),\big[|v|^2, |v|^2\big]\mu^{1/2}\right\},
$$
cf. \cite{Guo2003}.
For given vector valued function $f(t,x,v)$, there exists the following macro-micro decomposition
\begin{equation}\label{f decomposition}
f=\mathbf{P}f+\mathbf{P}^{\perp}f.
\end{equation}
Here,  $\mathbf{P}$ denotes the orthogonal projection from $L^2(\mathbb{R}_{v}^3) \times L^2(\mathbb{R}_{v}^3)$ to $\mathcal{N}(L)$, defined by
\begin{equation}\label{Pf define}
\mathbf{P}f:=\left\{a_{+}(t,x)[1,0]+a_{-}(t,x)[0,1]+v \cdot b(t,x)[1,1]+(|v|^2-3)c(t,x)[1,1]\right\}\mu^{1/2},
\end{equation}
or equivalently $\mathbf{P}=[\mathbf{P}_{+},\mathbf{P}_{-}]$ with
$$
\mathbf{P}_{\pm}f:=\left\{a_{\pm}(t,x)+v \cdot b(t,x)+ (|v|^2-3)c(t,x)\right\}\mu^{1/2},
$$
where
\begin{align}
a_{\pm}(t,x):=\;&\langle \mu^{1/2}, f_{\pm} \rangle =\langle \mu^{1/2}, \mathbf{P}_{\pm}f \rangle, \nonumber \\
b(t,x):=\;&\frac{1}{2}\langle v\mu^{1/2},f_{+}+f_{-}\rangle =\langle v\mu^{1/2}, \mathbf{P}_{\pm}f\rangle, \nonumber \\
c(t,x):=\;&\frac{1}{12} \big \langle (|v|^2-3)\mu^{1/2},f_{+}+f_{-} \big \rangle= \frac{1}{6} \big\langle(|v|^2-3)\mu^{1/2},\mathbf{P}_{\pm}f \big\rangle. \nonumber
\end{align}
Moreover, the symbol $\mathbf{I}$ denotes the identity operator defined by $\mathbf{I}:=[\mathbf{I}_{+}, \mathbf{I}_{-}]$ with $\mathbf{I}_{\pm}f=f_{\pm}$,  while the micro-projection operator $\mathbf{P}^{\perp}=[\mathbf{P}^{\perp}_{+}, \mathbf{P}^{\perp}_{-}]:=\mathbf{I}-\mathbf{P}=[\mathbf{I}_{+}-\mathbf{P}_{+}, \mathbf{I}_{-}-\mathbf{P}_{-}]$.

\subsection{Notations}
\hspace*{\fill}

Throughout the paper, $C$ denotes a generic positive constant independent of $\varepsilon$.
We use $X \lesssim Y$ to denote $X \leq CY$, where $C$ is a constant independent of $X$, $Y$. We also use the notation $X \approx Y$ to represent $X\lesssim Y$ and $Y\lesssim X$. The notation $X \ll 1$ means that $X$ is a positive constant small enough.

The multi-indexs $\alpha = [\alpha_1, \alpha_2, \alpha_3]\in\mathbb{N}^3$ and $\beta = [\beta_1, \beta_2, \beta_3]\in\mathbb{N}^3$ will be used to record space and velocity derivatives, respectively. We  denote the $\alpha$-th order space partial derivatives by $\partial_x^\alpha=\partial_{x_1}^{\alpha_1}\partial_{x_2}^{\alpha_2}\partial_{x_3}^{\alpha_3}$,
and the $\beta$-th order velocity partial derivatives by $\partial_v^\beta=\partial_{v_1}^{\beta_1}\partial_{v_2}^{\beta_2}\partial_{v_3}^{\beta_3}$.
In addition, $\partial^\alpha_\beta=\partial^\alpha_x\partial^\beta_v
=\partial^{\alpha_1}_{x_1}\partial^{\alpha_2}_{x_2}\partial^{\alpha_3}_{x_3}
\partial^{\beta_1}_{v_1}\partial^{\beta_2}_{v_2}\partial^{\beta_3}_{v_3}$ stands for the mixed space-velocity derivative.
The length of $\alpha$ is denoted by $|\alpha|=\alpha_1+\alpha_2+\alpha_3$.
If each component of $\theta$ is not greater than that of $\bar{\theta}$, we denote by $\theta \leq \bar{\theta}$. $\theta < \bar{\theta}$ means $\theta \leq \bar{\theta}$ and $|\theta| < |\bar{\theta}|$. %We shall also denote $D_v=\frac{1}{i}\partial_v$.

We use $| \cdot |_{L^p}$ to denote the $L^p$ norm in $\mathbb{R}^3_v$, and use $\|\cdot\|_{L^p}$ to denote the $L^p$ norm in $\mathbb{R}^3_x \times \mathbb{R}^3_v$ or in $\mathbb{R}^3_x$. In particular, if $p=2$, then we use $\|\cdot\|$ to denote $L^2$ norm in $\mathbb{R}^3_x \times \mathbb{R}^3_v$ or in $\mathbb{R}^3_x$.
Besides, $\| \cdot \|_{L^p_x L^q_v}$ denotes $\| | \cdot |_{L^q_v} \|_{L^p_x}$.
$\langle \cdot, \cdot\rangle$ denotes the $L^2$ inner product in $\mathbb{R}_v^3$ and $( \cdot, \cdot ) $ denotes the $L^2$ inner product in $\mathbb{R}_x^3 \times \mathbb{R}_v^3$ or in $\mathbb{R}_x^3$.
The norm of a vector means the sum of the norms for all components
of this vector. Also, the norm of $\nabla_x^k f$ means the sum of the norms for all functions $\partial^\alpha f$ with $|\alpha| = k$.

We use $\Lambda^{-\varrho}f(x)$ to denote
\begin{align*}
\Lambda^{-\varrho} f(x):= {(2\pi)^{-3/2}} \int_{\mathbb{R}^3} |y|^{-\varrho} \widehat{f}(y) e^{ix \cdot y} \d y,
\end{align*}
where $\widehat{f}(y) :=  {(2\pi)^{-3/2}} \int_{\mathbb{R}^3} f(x) e^{-ix \cdot y}\d x$
represents the Fourier transform of $f(x)$.

As described in \cite{AMUXY2012JFA, DLYZ2013}, we introduce the following norms for the non-cutoff VMB system
\begin{align*}
|f|_{L^2_D}^2=|f|_{D}^2:= & \int_{\mathbb{R}^6 \times \mathbb{S}^2} B\left(v-v_*, \sigma\right) \mu(v_*)\Big(f(v^\prime)-f(v)\Big)^2 \d v_* \d v \d \sigma \\ & +\int_{\mathbb{R}^6 \times \mathbb{S}^2} B\left(v-v_*, \sigma\right) f(v_*)^2\left(\mu(v_*^{\prime})^{1/2}- \mu(v_*)^{1/2}\right)^2 \d v_* \d v \d \sigma.
\end{align*}
For $\ell \in \mathbb{R}$, $\langle v \rangle := \sqrt{1+|v|^2}$, $L^2_\ell$ denotes the weighted space with the norm
\begin{align*}
|f|_{L_{\ell}^2}^2:=\int_{\mathbb{R}^3}\langle v\rangle^{2\ell}|f(v)|^2 \d v.
\end{align*}
And the weighted fractional Sobolev norm $\left|f(v)\right|_{H^s_\ell}^2=\left|\langle v \rangle^\ell f(v)\right|_{H^s}^2$ is given by
\begin{align*}
|f|_{H_{\ell}^s}^2:=|f|_{L_{\ell}^2}^2+\iint_{\mathbb{R}^3 \times \mathbb{R}^3} \frac{\big[\langle v\rangle^{\ell}  f(v)-\langle v^{\prime}\rangle^{\ell}  f(v^{\prime})\big]^2}{|v-v^{\prime}|^{3+2 s}} \chi_{|v-v^{\prime}| \leq 1} \d v \d v^{\prime},
\end{align*}
which turns out to be equivalent with
\begin{align*}
| f|_{H_{\ell}^s}^2:=\int_{\mathbb{R}^3}\langle v\rangle^{2\ell} \big|(1-\Delta_v)^{\frac{s}{2}} f(v)\big|^2 \d v.
\end{align*}
Notice that
\begin{align}\label{norm inequality}
| f |_{L_{\gamma/2 + s}^2}^2+| f |_{H_{\gamma/2}^s}^2 \lesssim | f |_{D}^2
\lesssim | f |_{H_{\gamma/2+  s}^s}^2,
\end{align}
cf. \cite{AMUXY2012JFA}.
Furthermore, in $\mathbb{R}_x^3 \times \mathbb{R}_v^3$, we use the notations
$$
\|f\|_{L^2_\ell}=\big\| |f|_{L^2_\ell}\big\|_{L_x^2},
\quad
\|f\|_{L^p_x L^2_\ell}=\big\| |f|_{L^2_\ell}\big\|_{L_x^p},
\quad
 \| f \|_D=\big\| |f |_D \big\|_{L_x^2},
 \quad
 \|  f  \|_{L^p_x L^2_D}=\big\| | f |_D\big\|_{L_x^p}.
$$
For an integer $N \geq 0$, we define the Sobolev space
\begin{align*}
&|f|_{H^N_\ell}:=\sum_{|\beta| \leq N} |\partial_\beta f |_{L^2_\ell}, \;\;\;\;
\|f\|_{H^N_{x}}:=\sum_{|\alpha| \leq N} \|\partial^\alpha f \|_{L^2_x}, \;\;\;\;
\|f\|_{L^2_{x} H^N_v}:=\sum_{|\beta| \leq N} \|\partial_\beta f \|_{L^2_{x,v}}, \\
&\|f\|_{H^N_{x} L^2_v}:=\sum_{|\alpha| \leq N} \|\partial^\alpha f \|_{L^2_{x,v}}, \;\;\;\;
\|f\|_{H^N_{x} L^2_D}:=\sum_{|\alpha| \leq N} \|\partial^\alpha f \|_{D}.
\end{align*}

Finally, we define $B_C \subset \mathbb{R}^3$ to be the ball with center origin and radius $C$, and use $L^2 (B_C )$ to denote the space $L^2$ over $B_C$ and likewise for other spaces.

\subsection{Main Results}
\hspace*{\fill}

To state the main results of this paper, we introduce the following fundamental instant energy functional and the corresponding dissipation rate functional
\begin{align}
\label{without weight energy functional}
{\mathcal{E}}_{N} (t) \sim\;
&\sum_{|\alpha|\leq N} \left\|\partial^{\alpha}f^{\varepsilon}\right\|^2
+\sum_{|\alpha|\leq N} \left\|\partial^{\alpha}E^{\varepsilon}\right\|^2
+\sum_{|\alpha|\leq N} \left\|\partial^{\alpha}B^{\varepsilon}\right\|^2,  \\
\label{without weight dissipation functional}
%-----------------------------------
{\mathcal{D}}_{N} (t)\sim \;
&\sum_{1 \leq |\alpha| \leq N} \left \| \partial^{\alpha} \mathbf{P}f^\varepsilon \right \|^2
+ \frac{1}{\varepsilon^2} \sum_{|\alpha| \leq N}\left\|\partial^{\alpha}
\mathbf{P}^{\perp}f^\varepsilon\right\|^2_D
+ \left\| a^\varepsilon_{+}-a^\varepsilon_{-}\right\|^2 \nonumber\\
\;&+\sum_{|\alpha| \leq N-1} \left\|\partial^{\alpha}E^{\varepsilon}\right\|^2
+\sum_{1 \leq |\alpha| \leq N-1} \left\|\partial^{\alpha}B^{\varepsilon}\right\|^2,
\end{align}
respectively, where the integer $N\in\mathbb{Z}$ will be determined later.

Due to the weaker dissipation of the linearized Boltzmann operator $L$ for the non-hard sphere case rather than the hard sphere case, in order to deal with the external force term brought by the self-consistent electromagnetic field, we need to introduce the following time-velocity weight function
\begin{align}\label{weight function}
w_{l}(\alpha,\beta):= \overline{w}_{l}(\alpha,\beta)
e^\frac{q \langle v \rangle}{(1+t)^\vartheta},
\end{align}
where $0 < q \ll 1$, the constants $l$ and $\vartheta$ will be determined later. Additionally, the velocity weight function $\overline{w}_{l}(\alpha,\beta)$  is defined as follows
\begin{align}\label{weight function2}
\overline{w}_{l}(\alpha,\beta):=
\begin{cases}
\langle v \rangle^{l-\max\left\{\frac{\gamma+2s}{2},\frac{1}{s}\right\}|\alpha|-\max\left\{\frac{\gamma+2s}{2},\frac{1}{s}\right\}|\beta|},
&  \text{ if } \gamma +2s \geq 0, \\[1mm]
\langle v \rangle^{l-(\frac{-2\gamma+1}{s}+\gamma)|\alpha|-(\frac{-2\gamma+1}{s})|\beta|},
&  \text{ if } \max\{-3,-\frac{3}{2}-2s\} < \gamma < -2s.
\end{cases}
\end{align}
In order to overcome the regularization loss of the electromagnetic field, we use the following notation
\begin{align}\label{weight function1}
\widetilde{w}_{l}(\alpha,\beta):=
\begin{cases}
w_{l}(\alpha,\beta), &\quad \text{ if } |\alpha|+|\beta| \leq N-1, \\[1mm]
(1+t)^{-\frac{1+\varepsilon_0}{4}}w_{l}(\alpha,\beta), &\quad\text{ if } |\alpha|+|\beta| =N,
\end{cases}
\end{align}
where $\varepsilon_0> 0$ is small enough.
It is worth noting that this notation $\widetilde{w}_{l}(\alpha,\beta)$ is only for writing convenience. Actually, we can first make the energy estimate with respect to the weight function $w_{l}(\alpha, \beta)$ $(|\alpha|+|\beta|=N)$ and then apply the time factor $(1+t)^{-\frac{1+\varepsilon_0}{2}}$.
Moreover, sometimes we write $\widetilde{w}_{l}(|\alpha|,|\beta|)$ to denote $\widetilde{w}_{l}(\alpha, \beta)$ for simplicity.
Then the weighted instant energy functional $\widetilde{\mathcal{E}}_{N,l} (t)$ and the corresponding dissipation rate functional $\widetilde{\mathcal{D}} _{N,l} (t)$ are defined as
\begin{align}\label{energy functional}
\widetilde{\mathcal{E}}_{N,l} (t) \sim\;
& {\mathcal{E}}_{N} (t)  + \sum_{\substack{|\alpha|+|\beta| \leq N \\ |\alpha| \leq N-1}}\left\|\widetilde{w}_l (\alpha, \beta) \partial_{\beta}^{\alpha} \mathbf{P}^{\perp}f^{\varepsilon}\right\|^2
+\varepsilon\sum_{|\alpha|=N}\left\|\widetilde{w}_l (\alpha, 0) \partial^{\alpha} f^{\varepsilon}\right\|^2 , \\
%------------------------------------
\label{dissipation functional}
\widetilde{\mathcal{D}}_{N,l} (t)\sim \;& {\mathcal{D}}_{N} (t)
+\frac{1}{\varepsilon^{2}} \sum_{\substack{|\alpha|+|\beta| \leq N \\
|\alpha| \leq N-1}}\left\|\widetilde{w}_l (\alpha, \beta) \partial_{\beta}^{\alpha}
\mathbf{P}^{\perp} f^{\varepsilon}\right\|^2_D
+\frac{1}{\varepsilon}\sum_{|\alpha|=N} \left\|\widetilde{w}_l (\alpha, 0) \partial^{\alpha}
\mathbf{P}^{\perp} f^{\varepsilon}\right\|^2_D \nonumber\\
&+ \frac{1}{(1+t)^{1+\vartheta}} \left\{ \sum_{\substack{|\alpha|+|\beta| \leq N \\
|\alpha| \leq N-1}}\left\|\langle v \rangle^\frac{1}{2} \widetilde{w}_l (\alpha, \beta) \partial_{\beta}^{\alpha}
\mathbf{P}^{\perp} f^{\varepsilon}\right\|^2 +  \varepsilon \sum_{|\alpha|= N} \left\|\langle v \rangle^\frac{1}{2} \widetilde{w}_l (\alpha, 0) \partial^{\alpha} f^{\varepsilon}\right\|^2 \right\} \nonumber \\
&+ \frac{1}{1+t} \left\{ \sum_{\substack{|\alpha|+|\beta| = N \\
|\alpha| \leq N-1}}\left\| \widetilde{w}_l (\alpha, \beta) \partial_{\beta}^{\alpha}
\mathbf{P}^{\perp} f^{\varepsilon}\right\|^2 +  \varepsilon \sum_{|\alpha|= N} \left\| \widetilde{w}_l (\alpha, 0) \partial^{\alpha} f^{\varepsilon}\right\|^2 \right\},
\end{align}
respectively.

Besides, to seek for the desired time decay rate to close the energy estimate, we introduce the following energy functional with the lowest $k$-order space derivative $\mathcal{E}^k_{N_0}(t)$ and the corresponding dissipation rate functional with the lowest $k$-order space derivative ${\mathcal{D}}^k_{N_0} (t)$
\begin{align}
\label{low k energy}
\mathcal{E}^k_{N_0}(t)\sim\;
& \sum_{k \leq |\alpha| \leq N_0} \left\|\partial^{\alpha}f^{\varepsilon}\right\|^2
+\sum_{k \leq |\alpha| \leq N_0} \left\|\partial^{\alpha}E^{\varepsilon}\right\|^2
+\sum_{k \leq |\alpha| \leq N_0} \left\|\partial^{\alpha}B^{\varepsilon}\right\|^2,  \\
\label{low k dissipation}
{\mathcal{D}}^k_{N_0} (t)\sim \;
& \sum_{k+1 \leq |\alpha| \leq N_0} \left \| \partial^{\alpha} \mathbf{P}f^\varepsilon \right \|^2
+ \frac{1}{\varepsilon^2} \sum_{k \leq |\alpha| \leq N_0} \left\|\partial^{\alpha}
\mathbf{P}^{\perp}f^\varepsilon\right\|^2_D
+\sum_{ |\alpha| = k } \left \| \partial^{\alpha} (a_{+}^\varepsilon-a_{-}^\varepsilon) \right \|^2 \nonumber\\
\;& +\sum_{k \leq |\alpha| \leq N_0-1} \left\|\partial^{\alpha} E^{\varepsilon}\right\|^2
+\sum_{k+1 \leq |\alpha| \leq N_0-1} \left\|\partial^{\alpha} B^{\varepsilon}\right\|^2,
\end{align}
respectively, where the integer $N_0\in\mathbb{Z}$ will be determined later.

Moreover, due to the singularity brought by the Knudsen number $\varepsilon$, the time decay estimate of the nonlinear problem can not be obtained by the Duhamel principle and semigroup estimate of the linear problem, so that we have to resort the Sobolev space with negative index. For this, we define the following instant energy functional $\overline{\mathcal{E}}_{N,l}(t)$ and the corresponding dissipation rate functional $\overline{\mathcal{D}}_{N,l} (t)$
\begin{align}
\label{negative sobolev energy}
\overline{\mathcal{E}}_{N,l} (t) \sim\;
& \widetilde{\mathcal{E}}_{N,l} (t) +\left\| \Lambda^{-\varrho}f^\varepsilon\right\|^2
+\left\| \Lambda^{-\varrho}E^\varepsilon\right\|^2
+\left\| \Lambda^{-\varrho}B^\varepsilon\right\|^2,\\
\label{negative sobolev dissipation}
\overline{\mathcal{D}}_{N,l} (t) \sim\;&\widetilde{\mathcal{D}}_{N,l}(t)
+\left\| \Lambda^{1-\varrho} \mathbf{P}f^\varepsilon\right\|^2
+\frac{1}{\varepsilon^2}\left\| \Lambda^{-\varrho} \mathbf{P}^{\perp}f^\varepsilon\right\|^2_D \nonumber\\
\;& + \left\| \Lambda^{-\varrho} (a^\varepsilon_{+}-a^\varepsilon_{-})\right\|^2
+ \left\| \Lambda^{-\varrho}E^\varepsilon\right\|_{H^1}^2
+ \left\| \Lambda^{1-\varrho}B^\varepsilon\right\|^2,
\end{align}
respectively, where the positive constant $\varrho>0$ will be determined later.
\medskip

Our first main result gives the uniform estimate with respect to $\varepsilon\in (0,1]$ of global solutions to the VMB system (\ref{rVMB}) for the potential case  $ \gamma > \max\{-3, -\frac{3}{2}-2s\} $ and $0<s<1$.

\begin{theorem}\label{mainth1}
Let $\gamma > \max\left\{-3, -\frac{3}{2}-2s\right\}$, $0 < s <1$ and $0<\varepsilon \leq 1$. Introduce the following constants in sequence
\begin{align}
\begin{split}\label{mainth1 assumption}
&N_0 \geq 3\; (N_0 \in \mathbb{Z}), \quad  N=N_0+3, \quad 1 < \varrho < \frac{3}{2}, \quad 0 < \vartheta \leq \frac{\varrho-1}{2}, \quad 0 < q \ll 1, \\
&\text{~and~}\; l = \left\{\begin{array}{l}
l_1 \geq \max\left\{\frac{\gamma+2s}{2},\frac{1}{s}\right\}N+1, \qquad\;\text{~if~}\; \gamma+2s \geq 0, \\[2mm]
l_2 \geq \frac{-2\gamma+1}{s} N+ 1, \qquad\qquad\quad\;\;\;\; \text{~~\!if~}\; \max\left\{-3,-\frac{3}{2}-2s\right\} < \gamma < -2s.
\end{array}\right.
\end{split}
\end{align}
If there exists a small constant $\delta_1>0$ independent of $\varepsilon$ such that the initial data
\begin{equation}\nonumber
\overline{\mathcal{E}}_{N, l}(0) \leq \delta_1^2,
\end{equation}
  then the VMB system \eqref{rVMB} admits a unique global solution
$(f^{\varepsilon}, E^\varepsilon, B^\varepsilon)$  satisfying
\begin{align}
&\overline{\mathcal{E}}_{N,l}(t) + \lambda  \int_{0}^{t} \overline{\mathcal{D}}_{N,l}(\tau) \d \tau \leq \overline{\mathcal{E}}_{N,l}(0),
\label{thm1 estimate 1} \\
&\mathcal{E}^k_{N_0}(t) \leq C\overline{\mathcal{E}}_{N,l}(0) (1+t)^{-(k+\varrho)}, \;\;\;\; k=0,1,
\label{thm1 estimate 2}
\end{align}
for any $t \geq 0$ and some positive constant $C>0$ independent of $\varepsilon$.
\end{theorem}
\medskip

The second main result is on hydrodynamic limit from the VMB system \eqref{rVMB}
to the two-fluid incompressible NSFM system with Ohm's law.
\begin{theorem}\label{mainth2}
Let $(f^\varepsilon, E^\varepsilon, B^\varepsilon)$  be the
global solution to the VMB system \eqref{rVMB} constructed in Theorem
\ref{mainth1} with initial data $f_0^\varepsilon={f}_0^\varepsilon(x,v)$,  $E_0^\varepsilon=E_0^\varepsilon(x,v)$ and $B_0^\varepsilon=B_0^\varepsilon(x,v)$. Suppose that there exist scalar functions $(\rho_0, n_0, \theta_0, \omega_0)=(\rho_0(x), n_0(x), \theta_0(x), \omega_0(x))$ and vector-valued functions $(u_0, E_0, B_0, j_0)=(u_0(x), E_0(x), B_0(x), j_0(x))$ such that
\begin{align*}%\label{theorem1.3 1}
&{f}_0^\varepsilon \to {f}_0\; \text{ strongly in } H^{N}_x L^2_v,\\
&
E_0^\varepsilon \to E_0,\;\;  B_0^\varepsilon \to B_0\;  \text{ strongly in } H^{N}_x,\\
&\frac{1}{\varepsilon}\big\langle f_0^{\varepsilon}, q_1v\mu^{1/2} \big\rangle \to j_0, \;\;
\frac{1}{\varepsilon}\big\langle f_0^{\varepsilon}, q_1\big(\frac{|v|^2}{3}-1\big)\mu^{1/2} \big\rangle \to \omega_0\; \text{ strongly in } H^{N}_x
\end{align*}
as $\varepsilon \to 0$ and ${f}_0={f}_0(x, v)$ is of the form
\begin{align*}%\label{theorem1.3 2}
f_0=\big(\rho_0+\frac{1}{2}n_0\big)\frac{q_1+q_2}{2}\mu^{1/2} + \big(\rho_0-\frac{1}{2}n_0\big)\frac{q_2-q_1}{2}\mu^{1/2}+ u_0 \cdot v q_2 \mu^{1/2} + \theta_0 \big(\frac{|v|^2}{2}-\frac{3}{2}\big)q_2 \mu^{1/2},
\end{align*}
where $q_{1}=[1,-1], q_{2}=[1,1]$.

Then there hold
\begin{equation*}%\label{theorem1.3 4}
\begin{split}
&{f}^\varepsilon \to {f}
\; \text{ weakly}\!-\!* \text{ in }  L^\infty(\mathbb{R}^+; H^N_x L^2_v)
 \text{ and strongly in } C(\mathbb{R}^+; H^{N-1}_x L^2_v), \\
&E^\varepsilon \to E \;
\text{ weakly}\!-\!*  \text{ in }   L^\infty(\mathbb{R}^+;H^{N}_{x})
\text{ and strongly in } C(\mathbb{R}^+; H^{N-1}_x),\\
&B^\varepsilon \to B \;
\text{ weakly}\!-\!*  \text{ in }   L^\infty(\mathbb{R}^+;H^{N}_{x})
\text{ and strongly in } C(\mathbb{R}^+; H^{N-1}_x),\\
%---------------------------------
&\frac{1}{\varepsilon}\big\langle f^{\varepsilon}, q_1v\mu^{1/2} \big\rangle \to j\;
\text{ weakly}\!-\!* \text{ in } L^\infty(\mathbb{R}^+; H^N_{x}) \text{ and strongly in } C(\mathbb{R}^+; H^{N-1}_x), \\
&\frac{1}{\varepsilon}\big\langle f^{\varepsilon}, q_1\big(\frac{|v|^2}{3}-1\big)\mu^{1/2} \big\rangle \to \omega \;
\text{ weakly}\!-\!*  \text{ in } L^\infty(\mathbb{R}^+; H^N_{x}) \text{ and strongly in } C(\mathbb{R}^+; H^{N-1}_x)
\end{split}
\end{equation*}
as $\varepsilon \to 0$, where $f={f}(t, x, v)$ has the form
\begin{align*}
\begin{split}%\label{theorem1.3 5}
f=\;&\big(\rho+\frac{1}{2}n\big)\frac{q_1+q_2}{2}\mu^{1/2} + \big(\rho-\frac{1}{2}n\big)\frac{q_2-q_1}{2}\mu^{1/2} + u \cdot v q_2 \mu^{1/2} + \theta\big(\frac{|v|^2}{2}-\frac{3}{2}\big)q_2 \mu^{1/2}.
\end{split}
\end{align*}
Moreover, the above mass density $\rho=\rho(t,x)$, the bulk velocity $u=u(t,x)$, the temperature $\theta=\theta(t,x)$, the electric charge $n=n(t,x)$, the internal electric energy $\omega=\omega(t,x)$, the electric field $E=E(t,x)$, the magnetic field $B=B(t,x)$ and the electric current $j=j(t,x)$ satisfy
 \begin{align*}
\begin{split}%\label{jixian solution space}
&\left(\rho, u, \theta, n, \omega, E, B\right) \in  C(\mathbb{R}^+; H^{N-1}_{x}) \cap L^\infty(\mathbb{R}^+; H^{N}_{x}),\; j \in C(\mathbb{R}^+; H^{N-2}_{x}) \cap L^\infty(\mathbb{R}^+; H^{N-1}_{x})
\end{split}
\end{align*}
 and the following two-fluid incompressible NSFM system with Ohm's law
\begin{equation}\label{INSFM limit}
\left\{\begin{array}{lr}
\displaystyle \partial_t u+u \cdot \nabla_x u-\nu \Delta_x u+\nabla_x p=\frac{1}{2}(n E+j \times B), & \nabla_x \cdot u=0, \\[1mm]
\displaystyle  \partial_t \theta+u \cdot \nabla_x \theta-\kappa \Delta_x \theta=0, & \rho+\theta=0, \\[1mm]
\displaystyle \partial_t E-\nabla_x \times B=-j, & \nabla_x \cdot E=n, \\[1mm]
\displaystyle \partial_t B+\nabla_x \times E=0, & \nabla_x \cdot B=0, \\
\displaystyle j-n u=\sigma\big(-\frac{1}{2} \nabla_x n+E+u \times B\big), & \omega= n \theta,
\end{array}\right.
\end{equation}
with initial data
$$
u(0, x)=\mathcal{P} u_0(x),\; \theta(0, x)=\frac{3}{5} \theta_0(x)-\frac{2}{5} \rho_0(x),\; E(0, x)=E_0(x),\; B(0, x)=B_0(x).
$$
Here, $\mathcal{P}$ is the Leray projection, while the viscosity coefficient $\nu$, the heat conductivity coefficient $\kappa$ and the electrical conductivity coefficient $\sigma$ are defined in \cite{AS2019}.
\end{theorem}
\medskip

\begin{remark}\label{remark-0}
The results given in Theorems \ref{mainth1} and \ref{mainth2} indicate that
\begin{equation*}
  \begin{split}
    &\sup_{0\leq t\leq \infty}  \left\|  \left[ F_{+}^{\varepsilon},  F_{-}^{\varepsilon} \right](t,x,v)- \left[ \mu(v),  \mu(v) \right]-\varepsilon \mu^{1/2}(v)f(t,x,v) \right\|_{H^{N-1}_xL^2_v} =o(\varepsilon),  \\
    &\sup_{0\leq t\leq \infty}  \left\{\left\| E^\varepsilon(t,x)- E(t,x) \right\|_{H^{N-1}_x}
    + \left\| B^\varepsilon(t,x)- B(t,x) \right\|_{H^{N-1}_x} \right\} =o(1),
  \end{split}
\end{equation*}
that is, the two-fluid incompressible NSFM system with Ohm's law (\ref{INSFM limit}) is the first-order approximation of
the original VMB system (\ref{GG1}).
\end{remark}
\medskip

\begin{remark}\label{remark-1.4}
This paper solves the open problem of diffusive limit of the non-cutoff VMB system in the framework of perturbation around the global Maxwellian for collision potentials $\gamma> \max\{-3, -\frac{3}{2}-2s\}$ and $0<s<1$, by introducing a new weight function and making
full use of the anisotropic dissipation property of the non-cutoff linearized Boltzmann operator.

This paper extends the previous diffusive limit results of the non-cutoff VMB system for soft potentials $\max\{-3,-\frac{3}{2}-2s\} < \gamma <-2s$ with strong angular singularity $\frac{1}{2} \leq s < 1$ \cite{JL2023ARXIV} to the potential case $\gamma > \max\{-3,-\frac{3}{2}-2s\}$ and $0 < s < 1$.

Moreover, the uniform estimate with respect to $\varepsilon\in (0,1]$ and the subsequent global existence and diffusive limit result constructed in Theorem
\ref{mainth1} cover all potential ranges that can be achieved by the previously global existence result of the non-cutoff VMB system in perturbation framework. Besides, weaker initial conditions are required in this work. In fact, for the case $\gamma > \max\{-3, -\frac{3}{2}-2s\}$, only $N \geq 6$ is required in Theorem \ref{mainth1}, rather than the requirement $N \geq 14$ in \cite{DLYZ2013, H2016}.

\end{remark}
\medskip

\begin{remark}\label{remark-1.5}
In fact, due to the singularity $\frac{1}{\varepsilon}$ present in nonlinear terms, obtaining the time decay estimate independent of $\varepsilon$ proves challenging when using the semigroup estimate of the linearized VMB system in conjunction with the Duhamel principle. Consequently, we resort to negative Sobolev space \cite{GW2012CPDE} to derive the time decay estimate, which is crucial for closing the weighted energy estimate. It is noteworthy that the time decay rate of $\mathcal{E}^k_{N_0}(t)$ can only be obtained for $k=0,1$. Specifically, to address the singularity $\frac{1}{\varepsilon}$ in the term $ \left(\frac{1}{\varepsilon} q_0 (v \times \nabla_x^k B^\varepsilon) \cdot \nabla_v \mathbf{P}^{\perp}f^\varepsilon, \nabla_x^k \mathbf{P}f^\varepsilon \right)$ for $k \leq N_0$, we only treat $\mathbf{P}^{\perp}f^\varepsilon$ as a dissipative term and cannot use the interpolation inequality to estimate it.
Therefore, to calculate it, we resort to integration by parts and $L^2$--$L^\infty$--$L^2$ H\"{o}lder inequality. Consequently, to close the energy estimate, first-order derivative dissipation of the microscopic part $\frac{1}{\varepsilon^2}\left\| \nabla_x \mathbf{P}^{\perp}f^\varepsilon \right\|^2_{D}$, is essential, as discussed in Section \ref{Global Existence}. That's why we avoid using the method mentioned in \cite{DLYZ2017}.
\end{remark}
\medskip

\subsection{Difficulties  and Innovations}
\hspace*{\fill}

In this subsection, we outline the difficulties and innovations proposed in this paper.

Our analysis is based on the uniform weighted energy estimate with respect to $\varepsilon\in (0,1]$ for the VMB system \eqref{rVMB} globally in time, which eventually leads to the incompressible NSFM limit.
To deal with the potential range $ \gamma > \max\{-3,-\frac{3}{2}-2s\} $ and $0<s<1$,  we are faced with considerable difficulties induced by the weak dissipation of the
non-hard sphere linearized Boltzmann operator and the singularity $\frac{1}{\varepsilon}$.
In what follows, we point out the critical technical points in our treatment.

\subsubsection {Control of the velocity growth in the transport $\frac{1}{\varepsilon} v \cdot \nabla_x f^\varepsilon$}
\hspace*{\fill}

 As is well known, the weighted energy estimate is necessary for the Boltzmann equation with soft potentials due to the weak dispassion of the linearized Boltzmann operator.
 To handle the velocity growth in the transport term $\frac{1}{\varepsilon} v \cdot \nabla_x f^\varepsilon$ as well as the singular factor $\frac{1}{\varepsilon}$, we design a time-velocity weight function for the soft potential case $ \max\{-3,-\frac{3}{2}-2s\} < \gamma < -2s$
\begin{align}\label{original weight1}
w_l(\alpha,\beta)= \langle v \rangle^{l-(m+\gamma)|\alpha|-m|\beta|} e^\frac{q \langle v \rangle}{(1+t)^\vartheta} ,
\end{align}
where the constants $l$ and $m$ will be determined later.
With the help of this weight function $w_l(\alpha,\beta)$, we estimate the transport term
$\frac{1}{\varepsilon} \partial^\alpha_{\beta}( v \cdot \nabla_x \mathbf{P}^{\perp} f^\varepsilon)$
for $|\alpha|+|\beta| \leq N$ and $|\beta|\geq 1$ as
\begin{align*}
&\frac{1}{\varepsilon} \left(\partial^{\alpha+e_i}_{\beta-e_i} \mathbf{P}^{\perp}f^\varepsilon, w_l^2(\alpha,\beta)\partial^\alpha_{\beta} \mathbf{P}^{\perp}f^\varepsilon\right)\\
%-------------------------------------
\lesssim\;&\frac{\eta}{\varepsilon^2} \left\| w_l(\alpha,\beta)   \partial^{\alpha}_{\beta} \mathbf{P}^{\perp} f^\varepsilon \right\|_D^2
+\left\| w_l(\alpha+e_i,\beta-e_i)  \partial^{\alpha+e_i}_{\beta-e_i} \mathbf{P}^{\perp} f^\varepsilon \right\|_D^2.
\end{align*}
where we have used the notation $e_i$ to denote the multi-index with the $i$-th element unit and the rest zero, as well as the fact
\begin{align*}
w_l(\alpha,\beta)=w_l(\alpha+e_i,\beta-e_i)\langle v \rangle^\gamma.
\end{align*}
Therefore, by proper linear combination of the weighted energy estimates for each order, the transport term can be treated with the help of this weight.
Note that this design of weight function $w_l(\alpha,\beta)$ can eliminate the limitation of strong angular singularity, i.e. $\frac{1}{2} \leq s < 1$ required in \cite{DLYZ2013}.

\subsubsection{Control of the velocity growth in the nonlinear terms}
\hspace*{\fill}

The first term with the velocity growth is referred to the electric field term $E^\varepsilon \cdot v f^\varepsilon$. To estimate this term, inspired by \cite{DYZ2002},
we use the time-velocity weight function \eqref{original weight1} that will generate an extra dissipation term. This requires a certain time decay rate of $E^\varepsilon$. For example, we have
\begin{align*}
 \left( E^\varepsilon \cdot v \partial^\alpha f^\varepsilon, w_l^2(\alpha, 0) \partial^\alpha f^\varepsilon \right)
 \lesssim \left\|E^\varepsilon\right\|_{L^\infty} \left\| \langle v \rangle ^{\frac{1}{2}}w_l(\alpha, 0) \partial^\alpha f^\varepsilon \right\|^2.
\end{align*}
This can be controlled by the extra dissipation if the electric field term has a certain time decay rate $\frac{1}{(1+t)^{1+\vartheta}}$.
However, this method fails to estimate the magnetic field term $\frac{1}{\varepsilon}  q_0 (v \times B^\varepsilon) \cdot \nabla_v f^\varepsilon$ because of  the extra singularity $\frac{1}{\varepsilon}$.

The second term with the velocity growth is the collision term $\frac{1}{\varepsilon}\Gamma(f^\varepsilon, f^\varepsilon)$ for the hard potential case $\gamma+2s \geq 0$, which has an additional velocity growth $\langle v \rangle ^{\frac{\gamma+2s}{2}}$, cf. Lemma \ref{Gamma1}.  To overcome this difficulty, we design the following time-velocity weight function  for the hard potential case $\gamma+2s \geq 0$,
$$
 w_l(\alpha,\beta)= \langle v \rangle^{l-m|\alpha|-m|\beta|} e^\frac{q \langle v \rangle}{(1+t)^\vartheta}  \;\text{ with }\; m\geq \frac{\gamma+2s}{2}.
$$
To briefly illustrate this, we use the following term as an example,
$$
\frac{1}{\varepsilon}\sum_{\substack{{\alpha_1+\alpha_2=\alpha}\\{\beta_1+\beta_2  \leq \beta}}} \int_{\mathbb{R}^3}
 \left| w_l(\alpha,\beta) \partial_{\beta_1}^{\alpha_1} \mathbf{P}^{\perp}f^\varepsilon\right|_{L_{\gamma/2+s}^2}\left|\partial_{\beta_2}^{\alpha_2} \mathbf{P}^{\perp}f^\varepsilon\right|_D
\left|w_l(\alpha,\beta) \partial_\beta^\alpha \mathbf{P}^{\perp}f^\varepsilon\right|_D \d x
$$
for $|\alpha|+|\beta|= N-1$.
For the case $|\alpha_2|+|\beta_2| \leq 1$, we take $L^2$--$L^\infty$--$L^2$ H\"{o}lder inequality, together with \eqref{norm inequality} and the fact $|\cdot|_{D} \lesssim |\cdot|_{H^1_{\gamma/2+s}}$, to estimate the above term. Moreover, the inequality that $w_{l}(\alpha, \beta) \leq w_{l}(\bar{\alpha}, \bar{\beta})$ for any $|\bar{\alpha}|+|\bar{\beta}| \leq |\alpha|+|\beta|$ and $| \bar{\beta} | \leq |\beta|$ will be frequently used. For the case $2 \leq |\alpha_2|+|\beta_2| \leq  N-2$, we use $L^6$--$L^3$--$L^2$ H\"{o}lder inequality to calculate it.
In this case, noting that $|\alpha_1+\alpha^\prime|+|\beta_1| \leq N-2$ for $|\alpha^\prime|=1$, we have $\langle v \rangle^{\frac{\gamma+2s}{2}}w_l(\alpha,\beta)\leq w_l(\alpha_1+\alpha^\prime,\beta_1)$. Thus, we bound the first term by the energy functional.  While for the case $|\alpha_2|+|\beta_2| = N-1$, we use $L^\infty$--$L^2$--$L^2$ H\"{o}lder inequality. Through this design of weight function, we can successfully solve the above-mentioned problem on velocity growth. For more details, please refer to the proof of Lemma \ref{hard Gamma}.

\subsubsection{Control of the extra velocity derivative in the nonlinear terms}
\hspace*{\fill}

In fact, the extra dissipation generated by the time-velocity weight function
like (\ref{weight function}) cannot directly control the magnetic field term $\frac{1}{\varepsilon}  q_0 (v \times B^\varepsilon) \cdot \nabla_v f^\varepsilon$, due to the fact that higher order derivatives are associated with weaker velocity weights. To overcome this difficulty, we make use of the dissipation property of the non-cutoff linearized Boltzmann operator and employ an interpolation inequality to balance the mismatch between derivatives and weight functions.

To illustrate our novel processing trick, let us mainly demonstrate the microscopic part $\frac{1}{\varepsilon}  q_0 (v \times B^\varepsilon) \cdot \nabla_v \mathbf{P}^{\perp}f^\varepsilon$ in the soft potential case $\max\{-3, -\frac{3}{2}-2s\} < \gamma <-2s$.  The term
$E^\varepsilon \cdot \nabla_v f^\varepsilon$ can be treated in a similar way.
On one hand, if all derivatives act on $\mathbf{P}^{\perp} f^\varepsilon$, then we use integration by parts with respect to velocity to make it disappear, namely,
\begin{align*}
\left( \frac{1}{\varepsilon} q_0 (v \times B^\varepsilon) \cdot \nabla_v \partial^\alpha \mathbf{P}^{\perp}f^\varepsilon,
 w^2_l(\alpha,0)\partial^\alpha \mathbf{P}^{\perp}f^\varepsilon\right)
=0.
\end{align*}
On the other hand, if not all derivatives act on $\mathbf{P}^{\perp} f^\varepsilon$, then we have that for $1\leq |\alpha_1| \leq |\alpha|$,
\begin{align*}
&\left( \frac{1}{\varepsilon} q_0 (v \times \partial^{\alpha_1}B^\varepsilon) \cdot \nabla_v \partial^{\alpha-\alpha_1} \mathbf{P}^{\perp}f^\varepsilon,
 w^2_l(\alpha,0)\partial^\alpha \mathbf{P}^{\perp}f^\varepsilon\right)\\
%-----------------------
\lesssim
&\;\frac{1}{\varepsilon}\int_{\mathbb{R}^3}|\partial^{\alpha_1} B^\varepsilon | \underbrace{\left|\langle v\rangle^{1-\frac{\gamma}{2}} w_l(\alpha,0)\partial^{\alpha-\alpha_1}_{e_i} \mathbf{P}^{\perp}f^{\varepsilon}\right|_{L^2} } \left|\langle v\rangle^{\frac{\gamma}{2}} w_l(\alpha,0)\partial^{\alpha} \mathbf{P}^{\perp}f^{\varepsilon}\right|_{L^2} \d x,
\end{align*}
Here, the weight function $\langle v\rangle^{1-\frac{\gamma}{2}} w_l(\alpha,0)$ before the term $\partial^{\alpha-\alpha_1}_{e_i} \mathbf{P}^{\perp} f^{\varepsilon}$
mismatches with the desired weight function $w_l(\alpha-\alpha_1,e_i)$, since the inequality $\langle v \rangle ^{1-\frac{\gamma}{2}}w_l(\alpha,0) \leq \langle v \rangle ^{\frac{\gamma}{2}}w_l(\alpha-\alpha_1,e_i)$ for $|\alpha_1| =1$ does not hold.
Thereby, to control the underbraced term, inspired by \cite{CDL2024SIAM}, we employ the following interpolation inequality
\begin{align*}
\left| f \right|_{H^1} \lesssim \left| \langle v \rangle ^\ell f\right|_{H^s} + \left| \langle v \rangle ^{-\frac{\ell s}{1-s}} f\right|_{H^{1+s}}.
\end{align*}
By defining a velocity weight function for soft potentials $ \max\{-3,-\frac{3}{2}-2s\} < \gamma < -2s$,
\begin{align}\label{original weight2}
\overline{w}_l(\alpha,\beta)= \langle v \rangle^{l-(m+\gamma)|\alpha|-m|\beta|},
\end{align}
and setting
\begin{align*}
\langle v \rangle ^\ell =\{\overline{w}_l(|\alpha|-1,0)\}^{1-s}\{\overline{w}_l(|\alpha|-1,1)\}^{-(1-s)},
\end{align*}
we find that
\begin{align*}
\langle v \rangle ^{-\frac{\ell s}{1-s}} =\{\overline{w}_l(|\alpha|-1,0)\}^{-s}\{\overline{w}_l(|\alpha|-1,1)\}^s.
\end{align*}
Substituting these two terms into the above inequality, the underbraced term
can be controlled as
\begin{align*}
&\left|\langle v\rangle^{1-\frac{\gamma}{2}} w_l(\alpha, 0) \partial^{\alpha-\alpha_1}_{e_i}\mathbf{P}^{\perp} f^{\varepsilon}\right|_{L^2} \\
%----------------------------
\lesssim\;&\left|\langle v\rangle^{1-\frac{\gamma}{2}} w_l(\alpha, 0) \partial^{\alpha-\alpha_1}\mathbf{P}^{\perp} f^{\varepsilon}\right|_{H^1}\\
%----------------------------
\lesssim\;&\left|\langle v \rangle ^\ell\left(\langle v\rangle^{1-\frac{\gamma}{2}} w_l(\alpha, 0) \partial^{\alpha-\alpha_1}\mathbf{P}^{\perp} f^{\varepsilon}\right)\right|_{H^{s}}
%-------------------------------------
+\left|\langle v \rangle ^{-\frac{\ell s}{1-s}} \left(\langle v\rangle^{1-\frac{\gamma}{2}} w_l(\alpha, 0) \partial^{\alpha-\alpha_1}\mathbf{P}^{\perp} f^{\varepsilon}\right)\right|_{H^{1+s}}
 \\
%----------------------------
\lesssim \;& \left| \langle v \rangle ^{\frac{\gamma}{2}} w_l(|\alpha|-1,0) \partial^{\alpha-\alpha_1}\mathbf{P}^{\perp} f^\varepsilon \right|_{H^s}
+\left| \langle v \rangle ^{\frac{\gamma}{2}} w_l(|\alpha|-1,1) \partial^{\alpha-\alpha_1}\mathbf{P}^{\perp} f^{\varepsilon} \right|_{H^{1+s}}\\
%----------------------------
\lesssim \;& \left| w_l(|\alpha|-1,0) \partial^{\alpha-\alpha_1}\mathbf{P}^{\perp} f^\varepsilon \right|_{H^s_{\gamma/2}}
+\left| w_l(|\alpha|-1,1) \partial^{\alpha-\alpha_1}_{e_i}\mathbf{P}^{\perp} f^{\varepsilon} \right|_{H^{s}_{\gamma/2}},
\end{align*}
where we used the facts
$$\langle v \rangle ^{1-\frac{\gamma}{2}}\overline{w}_l(\alpha, 0)\leq \langle v \rangle^{\frac{\gamma}{2}} \{\overline{w}_l(|\alpha|-1, 0)\}^s\{\overline{w}_l(|\alpha|-1, 1)\}^{1-s}
$$
and
$$\left|\nabla_v \left(\langle v \rangle^{\frac{\gamma}{2}}w_{l}(|\alpha|-1, 1) \right)\right|
\lesssim \langle v \rangle^{\frac{\gamma}{2}}w_{l}(|\alpha|-1, 1)
\lesssim \langle v \rangle^{\frac{\gamma}{2}}w_{l}(|\alpha|-1, 0)
$$
 by further setting $m \geq \frac{-2\gamma+1}{s}$ in the definition of $\overline{w}_l(\alpha, \beta)$ \eqref{original weight2}. Then derivatives and weight functions can be matched well with above treatment, which eventually leads to the control of the nonlinear term $\frac{1}{\varepsilon} q_0 (v \times B^\varepsilon) \cdot \nabla_v \mathbf{P}^{\perp}f^\varepsilon$, cf. the proof of Lemma \ref{soft B 1} for more details.

For the hard potential case $\gamma+2s \geq 0$, by setting the weight function
\begin{align*}
  w_l(\alpha,\beta)= \langle v \rangle^{l-m|\alpha|-m|\beta|} e^\frac{q \langle v \rangle}{(1+t)^\vartheta}  \;\text{ with }\; m\geq \frac{1}{s},
\end{align*}
we can guarantee that  the inequality
$$\langle v \rangle \overline{w}_{l}(\alpha,\beta) \leq \{\overline{w}_{l}(|\alpha|-1,|\beta|)\}^{s}\{\overline{w}_{l}(|\alpha|-1,|\beta|+1)\}^{1-s}$$ holds. Then using the similar way as the soft potential case $\max\{-3, -\frac{3}{2}-2s\} < \gamma < -2s$ mentioned above, we can control the extra velocity derivative in the hard potential case $\gamma+2s \geq 0$.

Combining the above statement, we finally design the weight functions (\ref{weight function}) with restriction (\ref{weight function2}),
 respectively for the hard and soft potential cases.

\subsubsection{Treatment of the singularities induced by $\frac{1}{\varepsilon}$}
\hspace*{\fill}

The usage of the weight function $w_l(\alpha, \beta)$ also generates a severe singularity when we estimate the $N$-th order space derivative of the linearized Boltzmann operator $L$, that is
\begin{align*}
\frac{1}{\varepsilon^2} \left(L \partial^\alpha f^\varepsilon,
w^2_l(\alpha,0)\partial^\alpha f^\varepsilon\right)
\gtrsim \frac{\lambda}{\varepsilon^2} \left\| w_l(\alpha,0) \partial^\alpha f^\varepsilon\right\|^2_{D}
-\frac{1}{\varepsilon^2}\left\|\partial^\alpha f^\varepsilon\right\|^2_{L^2(B_C)} \;\; \text{ for } |\alpha|=N,
\end{align*}
where the last term includes the singular macroscopic quantity $\frac{1}{\varepsilon^2}\left\|\partial^\alpha \mathbf{P} f^{\varepsilon}\right\|^2$ and is out of control.
We overcome this difficulty by multiplying an extra $\varepsilon$ in the weighted estimate, namely,
\begin{align*}
& \frac{1}{\varepsilon^2}\left( L \partial^\alpha f^{\varepsilon}, \varepsilon w_l^2(\alpha,0) \partial^\alpha f^{\varepsilon}\right)\\
%-----------------------------------
=\;&\frac{1}{\varepsilon}\left( L \partial^\alpha \mathbf{P}^{\perp}f^{\varepsilon}, w_l^2(\alpha,0) \partial^\alpha \mathbf{P}^{\perp}f^{\varepsilon}\right)
+\frac{1}{\varepsilon}\left( L \partial^\alpha \mathbf{P}^{\perp}f^{\varepsilon}, w_l^2(\alpha,0) \partial^\alpha \mathbf{P}f^{\varepsilon}\right)\\
%-----------------------------------
\gtrsim\;& \frac{\lambda}{\varepsilon} \left\| w_l(\alpha,0) \partial^\alpha \mathbf{P}^{\perp}f^\varepsilon\right\|^2_{D}
-\frac{1}{\varepsilon}\left\|\partial^\alpha \mathbf{P}^{\perp}f^\varepsilon\right\|^2_{L^2(B_C)}
-\frac{1}{\varepsilon^2} \left\| \partial^\alpha \mathbf{P}^{\perp}f^\varepsilon\right\|^2_{D}
-\left\|\partial^\alpha \mathbf{P} f^{\varepsilon}\right\|^2,
\end{align*}
so that the last three terms on the right-hand side can be controlled by the dissipation.

Meanwhile, due to the loss of $\left\| \partial^\alpha E^\varepsilon\right\|^2 (|\alpha|=N)$ in the dissipation, observing that
$$
\frac{1}{\varepsilon}\left( \partial^\alpha E^\varepsilon \cdot v\mu^{1/2}q_1, \varepsilon w^2_l(\alpha, 0) \partial^\alpha f^\varepsilon \right)
\lesssim \left\| \partial^\alpha E^\varepsilon \right\|^2 + \mathcal{D}_N(t),
$$
we use a time factor $(1+t)^{-\frac{1+\varepsilon_0}{2}}$ to calculate it. In fact, for the convenience of writing, we use the notation $\widetilde{w}_{l}(\alpha,\beta)$ defined in \eqref{weight function1} and then have
$$
\frac{1}{\varepsilon}\left( \partial^\alpha E^\varepsilon \cdot v\mu^{1/2}q_1, \varepsilon \widetilde{w}^2_l(\alpha, 0) \partial^\alpha f^\varepsilon \right)
\lesssim \frac{1}{(1+t)^{1+\varepsilon_0}}\left\| \partial^\alpha E^\varepsilon \right\|^2 + \mathcal{D}_N(t) \;\; \text{ for }  |\alpha|=N.
$$
Thus, the first term can be absorbed by the energy functional after certain calculations, cf. the proof of Proposition \ref{energy estimate 1} for more details.

 Finally, in order to close the a priori estimates, we need sufficient time decay of $E^\varepsilon$ and $B^\varepsilon$, as the cost of the dissipative term brought by the weight function.
 Due to the singularity $\frac{1}{\varepsilon}$ in front of the nonlinear terms, the method of semigroup estimate for linear problem combined with the Duhamel principle would result in $O(\frac{1}{\varepsilon})$ singularity in the final time decay rate of the nonlinear problem. To overcome this difficulty, inspired by \cite{GW2012CPDE}, we employ the interpolation inequality and energy estimate in negative Sobolev space $\|\Lambda^{-\varrho} \big(f^\varepsilon, E^\varepsilon, B^\varepsilon\big) \|$ to obtain the time decay estimate
\begin{align*}
\mathcal{E}^k_{N_0}(t) \lesssim   (1+t)^{-(k+\varrho)} \sup_{0 \leq \tau \leq t} \overline{\mathcal{E}}_{N,l}(\tau) \;\; \text{ for } k=0,1.
\end{align*}
  By combining the above strategies, we eventually close the global a priori estimates successfully.
\medskip

The rest of this paper is organized as follows. In Section \ref{Nonlinear Estimates}, we list basic lemmas concerning the properties of $L$ and $\Gamma$ in the framework of \cite{AMUXY2012JFA, DLYZ2013} and present the $\widetilde{w}_l(\alpha,\beta)$-weighted estimates for all the nonlinear terms. In Section \ref{The a Priori Estimate}, we establish a series of a priori estimates by the weighted energy method. In Section \ref{Global Existence}, we first obtain the time decay rate and close the a priori estimates, and then we give the proof of Theorem \ref{mainth1}. In section \ref{Limit section}, based on the uniform energy estimate with respect to $\e \in (0,1]$ globally in time, we justify
the limit of the VMB system \eqref{rVMB} to the two-fluid incompressible NSFM system with Ohm's law \eqref{INSFM limit}, that is, give the proof of Theorem \ref{mainth2}.
\medskip

\section{Nonlinear Estimates}\label{Nonlinear Estimates}
\hspace*{\fill}

In this section, we list some basic lemmas concerning the properties of the linearized Boltzmann operator $L$ and the nonlinear collision operator $\Gamma$ in the functional framework of \cite{AMUXY2012JFA, DLYZ2013}, and also present the $\widetilde{w}_l(\alpha,\beta)$-weighted estimates for the nonlinear terms  in \eqref{rVMB}.

\subsection{Preliminary Lemmas}
\hspace*{\fill}

In this subsection, we list some basic results to be used in the weighted energy estimates of the nonlinear terms in \eqref{rVMB}. Similar to \eqref{weight function}, we introduce the following time-velocity weight function
\begin{align*}
  w_\ell=w_\ell(t,v):=\langle v \rangle^\ell e^{\frac{q\langle v \rangle}{(1+t)^\vartheta}}.
\end{align*}

The first two lemmas concern the estimates on the linearized operator $L$ and nonlinear collision operator $\Gamma$.
\begin{lemma}
Let $\gamma > \max\left\{-3,-\frac{3}{2}-2s\right\}$, $0 < s < 1$ and $\ell \geq 0$.
\begin{itemize}
\setlength{\leftskip}{-5mm}
\item[(1)]  There holds
\begin{align}\label{L coercive1}
\langle Lf, f \rangle\gtrsim \left| \mathbf{P}^{\perp}f\right|^2_D.
\end{align}
\item[(2)] There holds
\begin{equation}\label{L coercive2}
\langle w^{2}_{\ell}Lf,f \rangle \gtrsim \left| w_{\ell}f\right|^2_D-C\left| f \right|^2_{L^2(B_C)}.
\end{equation}
\item[(3)]  For $|\beta| \geq 1$, there holds
\begin{equation}\label{L coercive3}
\langle w^{2}_{\ell}\partial_{\beta}Lf, \partial_{\beta}f \rangle \gtrsim \left| w_{\ell}\partial_{\beta}f \right|^2_D
-C\sum_{|\beta^{\prime}| < |\beta|} \left| w_{\ell}\partial_{\beta^{\prime}}f\right|^2_D-C \left| f \right|^2_{L^2(B_C)}.
\end{equation}
\end{itemize}
\end{lemma}

\begin{proof}
The coercive estimate \eqref{L coercive1} has been shown in \cite{AMUXY2012JFA}. The relevant coercive estimates \eqref{L coercive2} and \eqref{L coercive3} with the case of $\max\left\{-3,-\frac{3}{2}-2s\right\} < \gamma <-2s$ can be found in \cite{DLYZ2013}. For the case $\gamma+2s \geq 0$, we can employ the proof strategy outlined in \cite{DLYZ2013} to prove its validity. For simplicity, we omit the detailed proof.
\end{proof}

\begin{lemma}\label{Gamma1}
Let $0 < s < 1$ and $\ell \geq 0$.
\begin{itemize}
\setlength{\leftskip}{-6mm}
\item[(1)]
For $\gamma > \max\left\{-3, -\frac{3}{2}-2s\right\}$, there holds
\begin{align}
\begin{split}\label{gamma1}
\left|\left\langle \Gamma_{\pm}(f, g),  h_{\pm}\right\rangle\right|
%----------------------------------------------------------------
\lesssim\;& \left\{\left| f \right|_{L_{\gamma/2+s}^2}\left| g\right|_D+\left| g\right|_{L_{\gamma/2+s}^2}\left| f\right|_D \right\}\left| h\right|_{D} \\
& + \min \left\{\left| f\right|_{L^2}\left| g\right|_{L_{\gamma/2+s}^2},\left| g\right|_{L^2}\left| f\right|_{L_{\gamma/2+s}^2}\right\}\left| h\right|_{D}.
\end{split}
\end{align}
Moreover, one has
\begin{align}
\begin{split}\label{gamma2}
& \left|\left\langle \partial_\beta^\alpha \Gamma_{\pm}(f, g), w^{2}_{\ell} \partial_\beta^\alpha h_{\pm}\right\rangle\right| \\
%----------------------------------------------------------------
\lesssim\; & \sum\left\{\left| w_\ell \partial_{\beta_1}^{\alpha_1} f\right|_{L_{\gamma/2+s}^2}\left|\partial_{\beta_2}^{\alpha_2} g\right|_D+\left|\partial_{\beta_2}^{\alpha_2} g\right|_{L_{\gamma/2+s}^2}\left|w_\ell \partial_{\beta_1}^{\alpha_1} f\right|_D \right\}\left|w_\ell \partial_\beta^\alpha h\right|_D \\
& +\sum \min \left\{\left|w_\ell \partial_{\beta_1}^{\alpha_1} f\right|_{L^2}\left|\partial_{\beta_2}^{\alpha_2} g\right|_{L_{\gamma/2+s}^2},\left|\partial_{\beta_2}^{\alpha_2} g\right|_{L^2}\left|w_\ell \partial_{\beta_1}^{\alpha_1} f\right|_{L_{\gamma/2+s}^2}\right\}\left|w_\ell \partial_\beta^\alpha h\right|_{D} \\
& +\sum\left|e^{\frac{q\langle v\rangle}{(1+t)^{\vartheta}}} \partial_{\beta_2}^{\alpha_2} g\right|_{L^2}\left|w_\ell \partial_{\beta_1}^{\alpha_1} f\right|_{L_{\gamma/2+s}^2}\left|w_\ell \partial_\beta^\alpha h\right|_{D},
\end{split}
\end{align}
where the summation $\sum$ is taken over $\alpha_1+\alpha_2=\alpha$ and $\beta_1+\beta_2 \leq \beta$.
\item[(2)]
For $\gamma+2s \geq 0$, there holds
\begin{equation}\label{hard gamma L2}
\left| \langle v \rangle ^{-\ell}  \Gamma (f, g) \right|_{L^2} \lesssim \left| \langle v \rangle ^{-\ell} f\right|_{L^2_{\gamma/2+s}}
\left| g\right|_{H^i_{\gamma+2s}}.
\end{equation}
For $-3 < \gamma < -2s$, there holds
\begin{equation}\label{soft gamma L2}
\left| \langle v \rangle ^{\ell}  \Gamma (f, g) \right|_{L^2} \lesssim \min\left\{ \left| \langle v \rangle ^{\ell}f \right|_{H^2_{\gamma/2+s}}
\left| \langle v \rangle ^{\ell} g \right|_{H^i_{\gamma/2+s}}, \left| \langle v \rangle ^{\ell}f \right|_{L^2_{\gamma/2+s}}
\left| \langle v \rangle ^{\ell} g \right|_{H^{i+2}_{\gamma/2+s}} \right\}.
\end{equation}
Here, $i=1$ if $s \in (0, 1/2)$, and $i=2$ if $s \in [1/2, 1)$.
\item[(3)] Let $\zeta(v)$ be a smooth function that decays in $v$ exponentially. For any $m \geq 0$, there holds
\begin{equation}\label{Gamma zeta}
\left| \langle   \Gamma (f, g), \zeta \rangle  \right|
\lesssim |f|_{L^2_{-m}}|g|_{L^2_{-m}}.
\end{equation}
\end{itemize}
\end{lemma}

\begin{proof}
The estimate \eqref{gamma1} is referenced in \cite{AMUXY2012JFA}. The estimate \eqref{gamma2} with the hard potential case $\gamma+2s \geq 0$ has been illustrated in \cite{H2016}. For the case of weak angular singularity $0 < s < \frac{1}{2}$, the details of the estimate \eqref{gamma2} with soft potentials $\max\left\{-3,-\frac{3}{2}-2s\right\} < \gamma <-2s$ are provided in \cite{FLLZ2018}. For the case of strong angular singularity $\frac{1}{2} < s < 1$, we can follow  \cite{FLLZ2018} and \cite{DLYZ2013} with minor modifications to prove its validity. Actually, the estimate \eqref{gamma2} for the case of $\max\left\{-3,-\frac{3}{2}-2s\right\} < \gamma <-2s$ and $0 < s <1$ also can be established using methodology outlined in \cite{H2016}. The details of proof are omitted for brevity. The estimates \eqref{hard gamma L2} and \eqref{soft gamma L2} have been proven in \cite{Strain2012}. The estimate \eqref{Gamma zeta} can be found in \cite{GS2011}.
\end{proof}

In what follows, we collect some basic inequalities to be used throughout this paper. Firstly, the following Sobolev interpolation inequalities have been shown in \cite{GW2012CPDE}.

\begin{lemma}\label{sobolev interpolation}
 Let $2 \leq p<\infty$ and $k, \ell, m \in \mathbb{R}$. Then for any $f \in C_0^\infty(\mathbb{R}^3)$, we have
\begin{align}\label{sobolev interpolation1}
\left\|\nabla^k f\right\|_{L^p} \lesssim\left\|\nabla^{\ell} f\right\|^\theta\left\|\nabla^m f\right\|^{1-\theta},
\end{align}
where $0 \leq \theta \leq 1$ and $\ell$ satisfies
\begin{align}\nonumber
\frac{1}{p}-\frac{k}{3}=\left(\frac{1}{2}-\frac{\ell}{3}\right) \theta+\left(\frac{1}{2}-\frac{m}{3}\right)(1-\theta).
\end{align}
For the case $p=+\infty$, we have
\begin{align}\label{sobolev interpolation2}
\left\|\nabla^k f\right\|_{L^{\infty}} \lesssim\left\|\nabla^{\ell} f\right\|^\theta\left\|\nabla^m f\right\|^{1-\theta},
\end{align}
where $\ell \leq k+1$, $m \geq k+2$, $0 \leq \theta \leq 1$ and $\ell$ satisfies
\begin{align}\nonumber
-\frac{k}{3}=\left(\frac{1}{2}-\frac{\ell}{3}\right) \theta+\left(\frac{1}{2}-\frac{m}{3}\right)(1-\theta).
\end{align}
\end{lemma}

After considering the above lemma, we can frequently use the following lemma in later analysis.
\begin{lemma}\label{sobolev embedding 1}
For any $f \in C_0^\infty(\mathbb{R}^3)$, we have
\begin{align}\label{Sobolev-ineq}
\left\|f\right\|_{L^3}\lesssim \left\|f\right\|_{H^1},
\quad \left\|f\right\|_{L^6}\lesssim \left\| \nabla f \right\|_{L^2},
\quad \left\|f\right\|_{L^\infty}\lesssim \left\| \nabla f\right\|_{H^1}.
\end{align}
\end{lemma}

If $\varrho \in (0,3)$, then $\Lambda^{-\varrho}$ is the Riesz potential operator. The
Hardy--Littlewood--Sobolev theorem implies the following inequality for the Riesz potential operator $\Lambda^{-\varrho}$.

\begin{lemma}\label{negative embedding theorem}
Let $0<\varrho<3$, $1< p < q < \infty$, $1/q + \varrho/3=1/p$, there holds
\begin{align}\label{negative embed 1}
\left\|\Lambda^{-\varrho} f\right\|_{L^q} \lesssim\|f\|_{L^p}.
\end{align}
\end{lemma}

Based on Lemma \ref{sobolev interpolation}, we have the following lemma to be used frequently in later analysis.

\begin{lemma}\label{negative embedding theorem2}
Let $0<\varrho<3/2$. Then we have
\begin{align}\label{negative embed 2}
\|f\|_{L^{\frac{12}{3+2 \varrho}}} \lesssim\left\|\Lambda^{\frac{3}{4}-\frac{\varrho}{2}} f\right\|,
\quad\|f\|_{L^{\frac{3}{\varrho}}} \lesssim \left\|\Lambda^{\frac{3}{2}-\varrho} f\right\|.
\end{align}
\end{lemma}

In many places, we will use the following Minkowski inequality to interchange the order of a multiple integral, cf. \cite{GW2012CPDE}.
\begin{lemma}\label{minkowski theorem}
For $1 \leq p \leq q \leq \infty$, there holds
\begin{align}\label{minkowski}
\|f\|_{L_x^q L_v^p} \leq\|f\|_{L_v^p L_x^q}.
\end{align}
\end{lemma}

The following lemma concerns the equivalence of velocity weight function and differential operator up to commutation, whose proof can be found in \cite{AMSY2023} and \cite{HMUY2008}.
\begin{lemma}\label{equivalent norm}
Let $1 \leq p \leq \infty$ and $\ell, \theta \in \mathbb{R}$. Then there exists a generic constant $C$ independent of $f$ such that
\begin{align}\label{equivalent norm1}
\frac{1}{C}\left|\langle v\rangle^{\ell}\left\langle D_v\right\rangle^\theta f\right|_{L^p} \leq\left|\left\langle D_v\right\rangle^\theta\langle v\rangle^{\ell} f\right|_{L^p} \leq C\left|\langle v\rangle^{\ell}\left\langle D_v\right\rangle^\theta f\right|_{L^p},
\end{align}
that is, these two norms are equivalent. Here, $\left\langle D_v\right\rangle$ denotes the Fourier multiplier with symbol $\langle\xi\rangle$.
\end{lemma}

Lastly, we present an interpolation formula as described in \cite{CDL2024SIAM}.
\begin{lemma}\label{interpolation1}
Let $0 < s < 1$. For any $\ell \in \mathbb{R}$, there holds
\begin{equation}\label{interpolation}
\left| f \right|_{H^1_v} \lesssim \left| \langle v \rangle ^\ell f\right|_{H^s_v} + \left| \langle v \rangle ^{-\frac{\ell s}{1-s}} f \right|_{H^{1+s}_v}.
\end{equation}
\end{lemma}
\medskip

\subsection{Estimates on the Nonlinear Terms}
\hspace*{\fill}

The goal of this subsection is to make the weighted energy estimates for all the nonlinear terms in \eqref{rVMB}.

In this subsection, we always assume \eqref{mainth1 assumption} holds.
Additionally, we need to provide an assumption about the time decay rate of $E^\varepsilon$ and $B^\varepsilon$:
\begin{align}\label{assumption1}
\sup\limits_{0 \leq t \leq T}\Big\{ (1+t)^{\varrho+1} \sum_{1 \leq |\alpha| \leq N_0}\left\| \partial^\alpha (E^\varepsilon, B^\varepsilon)(t)\right\|^2\Big\} \leq \delta^2_0
\end{align}
for a sufficiently small positive constant $\delta_0$.

Recall the definitions of $w_l(\alpha,\beta)$, $\overline{w}_{l}(\alpha,\beta)$, $\widetilde{w}_l(\alpha,\beta)$, $\widetilde{\mathcal{E}}_{N,l}(t)$ and $\widetilde{\mathcal{D}}_{N,l}(t)$ given in \eqref{weight function}, \eqref{weight function2}, \eqref{weight function1}, \eqref{energy functional} and \eqref{dissipation functional}, respectively. Besides, the decomposition
\begin{equation}\label{gamma decomposition}
\Gamma(f^\varepsilon, f^\varepsilon)=  \Gamma(\mathbf{P} f^\varepsilon, \mathbf{P} f^\varepsilon)+\Gamma(\mathbf{P} f^\varepsilon,\mathbf{P}^{\perp} f^\varepsilon)+\Gamma(\mathbf{P}^{\perp} f^\varepsilon, \mathbf{P} f^\varepsilon)
+\Gamma(\mathbf{P}^{\perp} f^\varepsilon,\mathbf{P}^{\perp} f^\varepsilon)
\end{equation}
will be frequently used in later context.

The following two lemmas are about the estimates on the nonlinear collision term $\Gamma(f^\varepsilon,f^\varepsilon)$.

\begin{lemma}\label{soft Gamma}
Let $\max\left\{-3,-\frac{3}{2}-2s\right\} < \gamma <-2s $, $0 < s < 1$ and $0 < \varepsilon \leq 1$.
\begin{itemize}
\setlength{\leftskip}{-6mm}
\item[(1)] For $| \alpha | \leq N$, there holds
\begin{align}\label{soft gamma3}
\bigg|\frac{1}{\varepsilon}\left(\partial^\alpha \Gamma( f^{\varepsilon}, f^{\varepsilon} ), \partial^\alpha f^{\varepsilon}\right)\bigg|
\lesssim \sqrt{\mathcal{E}_{N}(t)} \mathcal{D}_{N}(t).
\end{align}
\item[(2)] For $| \alpha |+ |\beta| \leq N-1$, there holds
\begin{align}\label{soft gamma4}
\bigg|\frac{1}{\varepsilon} \left( \partial^\alpha_\beta \Gamma( f^{\varepsilon}, f^{\varepsilon} ),
\widetilde{w}^2_{l}(\alpha, \beta) \partial^\alpha_\beta \mathbf{P}^{\perp} f^{\varepsilon}\right)\bigg|
\lesssim \sqrt{\widetilde{\mathcal{E}}_{N,l}(t)} \widetilde{\mathcal{D}}_{N,l}(t).
\end{align}
\item[(3)] For $|\alpha|=N$, there holds
\begin{align}\label{soft gamma5}
\Big|\left( \partial^\alpha \Gamma( f^{\varepsilon}, f^{\varepsilon} ),
\widetilde{w}^2_{l}(\alpha, 0) \partial^\alpha  f^{\varepsilon}\right)\Big|
\lesssim \sqrt{\widetilde{\mathcal{E}}_{N,l}(t)} \widetilde{\mathcal{D}}_{N,l}(t).
\end{align}
\item[(4)] For $| \alpha |+ | \beta | = N$, $| \beta | \geq 1$ and $| \alpha | \leq N-1$, there holds
\begin{align}\label{soft gamma6}
\bigg| \frac{1}{\varepsilon}\left( \partial^\alpha_\beta \Gamma( f^{\varepsilon}, f^{\varepsilon} ),
\widetilde{w}^2_{l}(\alpha, \beta) \partial^\alpha_\beta \mathbf{P}^{\perp} f^{\varepsilon}\right) \bigg|
\lesssim \sqrt{\widetilde{\mathcal{E}}_{N,l}(t)} \widetilde{\mathcal{D}}_{N,l}(t).
\end{align}
\end{itemize}
\end{lemma}

\begin{proof}
For the sake of brevity, we present only the proof of \eqref{soft gamma6}. Similarly, one can obtain the proof of \eqref{soft gamma3}--\eqref{soft gamma5}.

In order to verify \eqref{soft gamma6}, we write
\begin{align}\label{I1define}
I_1\equiv \frac{1}{\varepsilon}\left( \partial^\alpha_\beta \Gamma( f^{\varepsilon}, f^{\varepsilon} ),
\widetilde{w}^2_{l}(\alpha, \beta) \partial^\alpha_\beta \mathbf{P}^{\perp} f^{\varepsilon}\right) :=I_{1,1}+ I_{1,2}+ I_{1,3}+ I_{1,4},
\end{align}
where $I_{1,1}$, $I_{1,2}$, $I_{1,3}$ and $I_{1,4}$ are the terms corresponding to the decomposition \eqref{gamma decomposition}, respectively.
To proceed, we will calculate each of the four terms individually.

To begin with, for the term $I_{1,1}$, we recall \eqref{Pf define} and use \eqref{gamma2} to obtain
\begin{align*}
I_{1,1} \lesssim\;& \frac{1}{\varepsilon}\sum_{\alpha_1 \leq \alpha}
\int_{\mathbb{R}^3} \left| \partial^{\alpha_1} (a^{\varepsilon}_{\pm},b^{\varepsilon},c^{\varepsilon})\right| \left| \partial^{\alpha-\alpha_1} (a^{\varepsilon}_{\pm},b^{\varepsilon},c^{\varepsilon})\right|
\left| \widetilde{w}_{l}(\alpha, \beta) \partial^\alpha_\beta \mathbf{P}^{\perp} f^{\varepsilon}\right|_D \d x,
%------------------------------------------
\end{align*}
where the fact $\widetilde{w}_{l}(\alpha,\beta)\mu^\delta \lesssim \mu^{\frac{\delta}{2}}$ is utilized.
Thanks to  $|\alpha_1| \leq |\alpha| \leq N-1$, it is clear that $1 \leq |\alpha_1+\alpha^\prime| \leq N$ for $|\alpha^\prime|=1$. Here and hereafter, the spatial derivative $\alpha^\prime$ is generated by \eqref{Sobolev-ineq}.  Thus, combining $L^6$--$L^3$--$L^2$ H\"{o}lder inequality with \eqref{Sobolev-ineq}, one gets
\begin{align*}
I_{1,1} \lesssim  \frac{1}{\varepsilon} \left\| \partial^{\alpha_1} (a^{\varepsilon}_{\pm},b^{\varepsilon},c^{\varepsilon})\right\|_{L^6} \left\| \partial^{\alpha-\alpha_1} (a^{\varepsilon}_{\pm},b^{\varepsilon},c^{\varepsilon})\right\|_{L^3}
\left\| \widetilde{w}_{l}(\alpha, \beta) \partial^\alpha_\beta \mathbf{P}^{\perp} f^{\varepsilon}\right\|_D
%------------------------------------------
\lesssim  \sqrt{\widetilde{\mathcal{E}}_{N,l}(t)} \widetilde{\mathcal{D}}_{N,l}(t).
\end{align*}

To bound the term $I_{1,2}$, we use \eqref{gamma2} and select the first one inside the minimum function in \eqref{gamma2}, which gives that
\begin{align}
I_{1,2}\lesssim\;&\frac{1}{\varepsilon} \sum_{\substack{{\alpha_1\leq \alpha}\\{\beta_2  \leq \beta}}}
\int_{\mathbb{R}^3}
\left| \partial^{\alpha_1} (a^{\varepsilon}_{\pm},b^{\varepsilon},c^{\varepsilon})\right|
\left| (1+t)^{-\frac{1+\varepsilon_0}{4}} \partial^{\alpha-\alpha_1}_{\beta_2} \mathbf{P}^{\perp} f^{\varepsilon}\right|_D
\left| \widetilde{w}_{l}(\alpha,\beta) \partial^{\alpha}_{\beta} \mathbf{P}^{\perp} f^{\varepsilon}\right|_D \d x \nonumber \\
&+\frac{1}{\varepsilon} \sum_{\substack{{\alpha_1\leq \alpha}\\{\beta_2  \leq \beta}}}
\int_{\mathbb{R}^3}
\left| \partial^{\alpha_1} (a^{\varepsilon}_{\pm},b^{\varepsilon},c^{\varepsilon})\right|
\left| (1+t)^{-\frac{1+\varepsilon_0}{4}} e^{\frac{q\langle v \rangle}{(1+t)^\vartheta}}\partial^{\alpha-\alpha_1}_{\beta_2} \mathbf{P}^{\perp} f^{\varepsilon}\right|_{L^2}
\left| \widetilde{w}_{l}(\alpha,\beta) \partial^{\alpha}_{\beta} \mathbf{P}^{\perp} f^{\varepsilon}\right|_D \d x \nonumber \\
:=\;&I_{1,2}^{\prime}+I_{1,2}^{\prime\prime}. \nonumber
\end{align}
For the term $I_{1,2}^{\prime}$, by using $L^\infty$--$L^2$--$L^2$ H\"{o}lder inequality, one gets that for the case $\alpha_1 =0$ and $\beta_2 \leq \beta$,
\begin{align*}
I_{1,2}^{\prime} \lesssim \frac{1}{\varepsilon} \left\| (a^{\varepsilon}_{\pm},b^{\varepsilon},c^{\varepsilon})\right\|_{L^\infty}
\left\| (1+t)^{-\frac{1+\varepsilon_0}{4}} \partial^{\alpha}_{\beta_2}  \mathbf{P}^{\perp} f^{\varepsilon}\right\|_D
\left\|\widetilde{w}_{l}(\alpha,\beta) \partial^{\alpha}_\beta \mathbf{P}^{\perp} f^{\varepsilon}\right\|_D
\lesssim \sqrt{\widetilde{\mathcal{E}}_{N,l}(t)} \widetilde{\mathcal{D}}_{N,l}(t),
\end{align*}
where we have used \eqref{Sobolev-ineq} and the fact that for any $|\alpha|+|\beta| \leq N$, there holds
\begin{align}\label{weight inequality1}
(1+t)^{-\frac{1+\varepsilon_0}{4}}  \leq \widetilde{w}_{l}(\alpha,\beta).
\end{align}
For the case of $1 \leq |\alpha_1| \leq |\alpha| \leq N-1 $ and $\beta_2 \leq \beta$, it follows from $L^6$--$L^3$--$L^2$ H\"{o}lder inequality that
\begin{align*}
I_{1,2}^{\prime} \lesssim  \frac{1}{\varepsilon}
\left\| \partial^{\alpha_1} (a^{\varepsilon}_{\pm},b^{\varepsilon},c^{\varepsilon})\right\|_{L^6}
\left\| (1+t)^{-\frac{1+\varepsilon_0}{4}} \partial^{\alpha-\alpha_1}_{\beta_2} \mathbf{P}^{\perp} f^{\varepsilon}\right\|_{L^3_x L^2_D}
\left\| \widetilde{w}_{l}(\alpha,\beta) \partial^{\alpha}_\beta \mathbf{P}^{\perp} f^{\varepsilon}\right\|_D,
\end{align*}
which is further bounded by $\sqrt{\widetilde{\mathcal{E}}_{N,l}(t)} \widetilde{\mathcal{D}}_{N,l}(t)$ by utilizing \eqref{Sobolev-ineq}, \eqref{minkowski} and \eqref{weight inequality1}.
For the term $I_{1,2}^{\prime \prime}$, noticing that for any $|\alpha|+|\beta| \leq N$, there holds
\begin{align}\label{weight inequality11}
(1+t)^{-\frac{1+\varepsilon_0}{4}} e^{\frac{q\langle v \rangle}{(1+t)^\vartheta}} \leq \widetilde{w}_{l}(\alpha,\beta),
\end{align}
we use the similar argument as the term $I_{1,2}^\prime$ to get the same upper bound. In fact, we only need to replace \eqref{weight inequality1} with \eqref{weight inequality11}.

Next, we compute the term $I_{1,3}$. Selecting the second one inside the minimum function in \eqref{gamma2} and noting the norm $|\cdot|_{D}$ is stronger than $|\cdot |_{L^2_{\gamma/2+s}}$ in \eqref{norm inequality}, we can deduce that
\begin{align*}
I_{1,3} \lesssim \frac{1}{\varepsilon} \sum_{\substack{{\alpha_2\leq \alpha}\\{\beta_1  \leq \beta}}}
\int_{\mathbb{R}^3}
\left| \widetilde{w}_{l}(\alpha,\beta) \partial^{\alpha-\alpha_2}_{\beta_1}\mathbf{P}^{\perp} f^{\varepsilon}\right|_D
\left| \partial^{\alpha_2} (a^{\varepsilon}_{\pm},b^{\varepsilon},c^{\varepsilon})\right|
\left| \widetilde{w}_{l}(\alpha,\beta) \partial^{\alpha}_{\beta} \mathbf{P}^{\perp} f^{\varepsilon}\right|_D \d x.
\end{align*}
For the case of $\alpha_2=0$ and $\beta_1 \leq \beta$, we have from $L^2$--$L^\infty$--$L^2$ H\"{o}lder inequality and \eqref{Sobolev-ineq} that
\begin{align*}
I_{1,3} \lesssim \frac{1}{\varepsilon}
\left\|  \widetilde{w}_{l}(\alpha,\beta) \partial^{\alpha}_{\beta_1}\mathbf{P}^{\perp} f^{\varepsilon}\right\|_D
\left\| (a^{\varepsilon}_{\pm},b^{\varepsilon},c^{\varepsilon})\right\|_{L^\infty}
\left\| \widetilde{w}_{l}(\alpha,\beta) \partial^{\alpha}_{\beta} \mathbf{P}^{\perp} f^{\varepsilon}\right\|_D
\lesssim  \sqrt{\widetilde{\mathcal{E}}_{N,l}(t)} \widetilde{\mathcal{D}}_{N,l}(t),
\end{align*}
where we have applied the fact that for any $| \bar{\alpha} | + |\bar{\beta}| \leq | \alpha |+ |\beta|$ and $| \bar{\beta} | \leq |\beta|$, there holds
\begin{align}\label{weight inequality2}
  \widetilde{w}_l(\alpha, \beta) \leq \widetilde{w}_l( \bar{\alpha}, \bar{\beta}).
\end{align}
For the case $1 \leq |\alpha_2| \leq |\alpha| \leq N-1$ and $\beta_1 \leq \beta$, we get from $L^6$--$L^3$--$L^2$ H\"{o}lder inequality that
\begin{align*}
I_{1,3} \lesssim \frac{1}{\varepsilon}
\left\| \widetilde{w}_{l}(\alpha,\beta) \partial^{\alpha-\alpha_2}_{\beta_1}\mathbf{P}^{\perp} f^{\varepsilon}\right\|_{L^6_x L^2_D}
\left\| \partial^{\alpha_2}(a^{\varepsilon}_{\pm},b^{\varepsilon},c^{\varepsilon})\right\|_{L^3}
\left\| \widetilde{w}_{l}(\alpha,\beta) \partial^{\alpha}_{\beta} \mathbf{P}^{\perp} f^{\varepsilon}\right\|_D.
\end{align*}
Thanks to \eqref{Sobolev-ineq}, \eqref{minkowski} and \eqref{weight inequality2}, the estimate above is controlled by $\sqrt{\widetilde{\mathcal{E}}_{N,l}(t)} \widetilde{\mathcal{D}}_{N,l}(t)$.

At last, applying \eqref{gamma2}, \eqref{norm inequality} and the fact $|\cdot|_{L^2_{\gamma/2+s}}\lesssim |\cdot|_{L^2}$, $I_{1,4}$ is decomposed as
\begin{align*}
&\frac{1}{\varepsilon}\sum \int_{\mathbb{R}^3} (1+t)^{-\frac{1+\varepsilon_0}{4}}
\left| w_{l}(\alpha,\beta) \partial^{\alpha_1}_{\beta_1} \mathbf{P}^{\perp}f^{\varepsilon} \right|_{L^2}
\left| \partial^{\alpha_2}_{\beta_2} \mathbf{P}^{\perp} f^{\varepsilon}\right|_{D}
\left| \widetilde{w}_{l}(\alpha,\beta) \partial^{\alpha}_{\beta} \mathbf{P}^{\perp} f^{\varepsilon}\right|_{D} \d x  \\
&+\frac{1}{\varepsilon}\sum \int_{\mathbb{R}^3} (1+t)^{-\frac{1+\varepsilon_0}{4}}
\left| w_{l}(\alpha,\beta) \partial^{\alpha_1}_{\beta_1} \mathbf{P}^{\perp}f^{\varepsilon} \right|_{D}
\left| e^{\frac{q \langle v \rangle}{(1+t)^\vartheta}} \partial^{\alpha_2}_{\beta_2} \mathbf{P}^{\perp} f^{\varepsilon}\right|_{L^2}
\left| \widetilde{w}_{l}(\alpha,\beta) \partial^{\alpha}_{\beta} \mathbf{P}^{\perp} f^{\varepsilon}\right|_{D} \d x \\
& :=I_{1,4}^\prime + I_{1,4}^{\prime\prime},
\end{align*}
where the summation $\sum$ is taken over $\alpha_1+\alpha_2=\alpha$ and $\beta_1+\beta_2 \leq \beta$.
For the term $I_{1,4}^\prime$, we divide it into two cases. For the case $|\alpha_2|+|\beta_2| \leq N/2$, it is obvious that $|\alpha_2+\alpha^\prime|+|\beta_2| \leq N-1$ for $1 \leq |\alpha^\prime| \leq 2$.  Hence, we use $L^2$--$L^{\infty}$--$L^2$ H\"{o}lder inequality, \eqref{Sobolev-ineq}, \eqref{minkowski} and \eqref{weight inequality2} to obtain
\begin{align*}
I_{1,4}^\prime \lesssim \frac{1}{\varepsilon}
\left\| \widetilde{w}_{l}(\alpha,\beta) \partial^{\alpha_1}_{\beta_1} \mathbf{P}^{\perp}f^{\varepsilon} \right\|
\left\| \partial^{\alpha_2}_{\beta_2} \mathbf{P}^{\perp} f^{\varepsilon}\right\|_{L^\infty_x L^2_D}
\left\| \widetilde{w}_{l}(\alpha,\beta) \partial^{\alpha}_{\beta} \mathbf{P}^{\perp} f^{\varepsilon}\right\|_{D}
\lesssim \sqrt{\widetilde{\mathcal{E}}_{N,l}(t)} \widetilde{\mathcal{D}}_{N,l}(t).
\end{align*}
Here, we assign the time factor $(1+t)^{-\frac{1+\varepsilon_0}{4}}$ to the first term. In fact, after applying \eqref{Sobolev-ineq}, we always assign $(1+t)^{-\frac{1+\varepsilon_0}{4}}$ to the term containing total derivatives $N$ within the first two terms. Otherwise, we utilize the inequality $(1+t)^{-\frac{1+\varepsilon_0}{4}} \leq 1$.
For the case $N/2 < |\alpha_2|+|\beta_2| \leq N$, observing that $|\alpha_1+ \alpha^\prime|+|\beta_1| \leq N-1$ for $1 \leq |\alpha^\prime| \leq 2$, by $L^{\infty}$--$L^2$--$L^2$ H\"{o}lder inequality, one has
\begin{align*}
I_{1,4}^\prime \lesssim \frac{1}{\varepsilon}
\left\| w_{l}(\alpha,\beta) \partial^{\alpha_1}_{\beta_1} \mathbf{P}^{\perp}f^{\varepsilon} \right\|_{L^\infty_x L^2_v}
\left\| (1+t)^{-\frac{1+\varepsilon_0}{4}} \partial^{\alpha_2}_{\beta_2} \mathbf{P}^{\perp} f^{\varepsilon}\right\|_{D}
\left\| \widetilde{w}_{l}(\alpha,\beta) \partial^{\alpha}_{\beta} \mathbf{P}^{\perp} f^{\varepsilon}\right\|_{D},
\end{align*}
which is further controlled by $\sqrt{\widetilde{\mathcal{E}}_{N,l}(t)} \widetilde{\mathcal{D}}_{N,l}(t)$ by making use of \eqref{Sobolev-ineq}, \eqref{minkowski}, \eqref{weight inequality1} and the fact that $w_l(\alpha, \beta) \leq w_l( \bar{\alpha}, \bar{\beta})$ for any $| \bar{\alpha} | + |\bar{\beta}| \leq | \alpha |+ |\beta|$ and $| \bar{\beta} | \leq |\beta|$. For the term $I_{1,4}^{\prime\prime}$, the same bound can be obtained in a way similar to the term $I_{1,4}^\prime$.

Therefore, we can obtain the desired estimate \eqref{soft gamma6} by substituting all the above estimates into \eqref{I1define}. This completes the proof of Lemma \ref{soft Gamma}.
\end{proof}

\begin{lemma}\label{hard Gamma}
Let $ \gamma+2s \geq 0$, $0 < s < 1$ and $0 < \varepsilon \leq 1$.
\begin{itemize}
\setlength{\leftskip}{-6mm}
\item[(1)] For $| \alpha | \leq N$, there holds
\begin{align}\label{hard gamma3}
\bigg|\frac{1}{\varepsilon}\left( \partial^\alpha \Gamma( f^{\varepsilon}, f^{\varepsilon} ),
 \partial^\alpha f^{\varepsilon}\right)\bigg|
\lesssim \sqrt{\widetilde{\mathcal{E}}_{N,l}(t)} \mathcal{D}_{N}(t).
\end{align}
\item[(2)] For $| \alpha |+ |\beta| \leq N-1$, there holds
\begin{align}\label{hard gamma4}
\bigg|\frac{1}{\varepsilon} \left( \partial^\alpha_\beta \Gamma( f^{\varepsilon}, f^{\varepsilon} ),
\widetilde{w}^2_{l}(\alpha, \beta) \partial^\alpha_\beta \mathbf{P}^{\perp} f^{\varepsilon}\right)\bigg|
\lesssim \sqrt{\widetilde{\mathcal{E}}_{N,l}(t)} \widetilde{\mathcal{D}}_{N,l}(t).
\end{align}
\item[(3)] For $|\alpha|=N$, there holds
\begin{align}\label{hard gamma5}
\Big|\left( \partial^\alpha \Gamma( f^{\varepsilon}, f^{\varepsilon} ),
\widetilde{w}^2_{l}(\alpha, 0) \partial^\alpha  f^{\varepsilon}\right) \Big|
\lesssim \sqrt{\widetilde{\mathcal{E}}_{N,l}(t)} \widetilde{\mathcal{D}}_{N,l}(t).
\end{align}
\item[(4)] For $| \alpha |+ | \beta | = N$, $| \beta | \geq 1$ and $| \alpha | \leq N-1$, there holds
\begin{align}\label{hard gamma6}
\bigg| \frac{1}{\varepsilon}\left( \partial^\alpha_\beta \Gamma( f^{\varepsilon}, f^{\varepsilon} ),
 \widetilde{w}^2_{l}(\alpha, \beta) \partial^\alpha_\beta \mathbf{P}^{\perp} f^{\varepsilon}\right) \bigg|
\lesssim \sqrt{\widetilde{\mathcal{E}}_{N,l}(t)} \widetilde{\mathcal{D}}_{N,l}(t).
\end{align}
\end{itemize}
\end{lemma}

\begin{proof}
For brevity, we only prove that for $| \alpha |+ | \beta | = N$, $| \beta | \geq 1$ and $| \alpha | \leq N-1$, there holds
\begin{align}\label{hard gamma7}
 \frac{1}{\varepsilon}\left( \partial^\alpha_\beta \Gamma( \mathbf{P}^{\perp}f^{\varepsilon}, \mathbf{P}^{\perp}f^{\varepsilon} ),
 \widetilde{w}^2_{l}(\alpha, \beta) \partial^\alpha_\beta \mathbf{P}^{\perp} f^{\varepsilon}\right)
\lesssim \sqrt{\widetilde{\mathcal{E}}_{N,l}(t)} \widetilde{\mathcal{D}}_{N,l}(t).
\end{align}
This is a part of the estimate \eqref{hard gamma6}. Other parts concerning $\mathbf{P}f^\varepsilon$ can be proved using the similar argument as that of $I_{1,1}$--$I_{1,3}$ in Lemma \ref{soft Gamma}.
The proof of \eqref{hard gamma3}--\eqref{hard gamma5} can be derived similarly.

To prove \eqref{hard gamma7}, for simplicity, we denote
$$
I_2 \equiv \text { the left-hand side of } \eqref{hard gamma7}.
$$
Applying \eqref{gamma2} and using the fact $|\cdot|_{L^2} \lesssim |\cdot|_{L^2_{\gamma/2+s}} \lesssim |\cdot|_{D}$ for $\gamma+2s \geq 0$, $I_2$ is decomposed as
\begin{align*}
&\frac{1}{\varepsilon} \sum \int_{\mathbb{R}^3} (1+t)^{-\frac{1+\varepsilon_0}{4}}
\left| w_{l}(\alpha,\beta) \partial^{\alpha_1}_{\beta_1} \mathbf{P}^{\perp}f^{\varepsilon} \right|_{L^2_{\gamma/2+s}}
\left| \partial^{\alpha_2}_{\beta_2} \mathbf{P}^{\perp} f^{\varepsilon}\right|_{D}
\left| \widetilde{w}_{l}(\alpha,\beta) \partial^{\alpha}_{\beta} \mathbf{P}^{\perp} f^{\varepsilon}\right|_{D} \d x  \\
&+\frac{1}{\varepsilon}  \sum \int_{\mathbb{R}^3} (1+t)^{-\frac{1+\varepsilon_0}{4}}
\left| w_{l}(\alpha,\beta) \partial^{\alpha_1}_{\beta_1} \mathbf{P}^{\perp}f^{\varepsilon} \right|_{D}
\left| e^{\frac{q \langle v \rangle}{(1+t)^\vartheta}} \partial^{\alpha_2}_{\beta_2} \mathbf{P}^{\perp} f^{\varepsilon}\right|_{L^2_{\gamma/2+s}}
\left| \widetilde{w}_{l}(\alpha,\beta) \partial^{\alpha}_{\beta} \mathbf{P}^{\perp} f^{\varepsilon}\right|_{D} \d x \\
&:= I_{2}^\prime + I_{2}^{\prime\prime},
\end{align*}
where the summation $\sum$ is taken over $\alpha_1+\alpha_2=\alpha$ and $\beta_1+\beta_2 \leq \beta$.

For the term $I_2^{\prime}$, we consider the following three cases $|\alpha_2|+|\beta_2| \leq 1$, $2 \leq |\alpha_2|+|\beta_2| \leq N-2$ and $N-1 \leq |\alpha_2|+|\beta_2| \leq N$.
For the first case $|\alpha_2|+|\beta_2| \leq 1$, by $L^2$--$L^\infty$--$L^2$ H\"{o}lder inequality, we have
\begin{align*}
I_{2}^\prime \lesssim
\frac{1}{\varepsilon}
\left\| \widetilde{w}_{l}(\alpha,\beta) \partial^{\alpha_1}_{\beta_1} \mathbf{P}^{\perp}f^{\varepsilon} \right\|_{L^2_{\gamma/2+s}}
\left\| \partial^{\alpha_2}_{\beta_2} \mathbf{P}^{\perp} f^{\varepsilon}\right\|_{L^\infty_x L^2_D}
\left\| \widetilde{w}_{l}(\alpha,\beta) \partial^{\alpha}_{\beta} \mathbf{P}^{\perp} f^{\varepsilon}\right\|_{D},
\end{align*}
where we assign $(1+t)^{-\frac{1+\varepsilon_0}{4}}$ to the first term thanks to the fact $|\alpha_2+\alpha^\prime|+|\beta_2+\beta^\prime| \leq 4 \leq N-2$ for $1 \leq |\alpha^\prime| \leq 2$ and $|\beta^\prime| \leq 1$. Here and hereafter, the velocity derivative $\beta^\prime$ is generated by the norm $|\cdot|_D$. Recalling $\widetilde{w}_{l}(\alpha, \beta)$ for the case $\gamma+2s \geq 0$ in \eqref{weight function}, the above fact shows that $\langle v \rangle ^{\frac{\gamma+2s}{2}} \leq \widetilde{w}_l(\alpha_2+\alpha^\prime, \beta_2+\beta^\prime)$, which, combined with the fact $|\cdot|_{D} \lesssim |\cdot|_{H^1_{\gamma/2+s}}$, implies that
\begin{align*}
\left\| \partial^{\alpha_2}_{\beta_2} \mathbf{P}^{\perp} f^{\varepsilon}\right\|_{L^\infty_x L^2_D}
\lesssim \sum_{\substack{{1 \leq |\alpha^\prime| \leq 2}\\{|\beta^\prime| \leq 1}}}
\left\| \partial^{\alpha_2+\alpha^\prime}_{\beta_2+\beta^\prime} \mathbf{P}^{\perp} f^{\varepsilon}\right\|_{L^2_{\gamma/2+s}}
\lesssim \sqrt{\widetilde{\mathcal{E}}_{N,l}(t)}.
\end{align*}
Therefore, the combination of this, \eqref{norm inequality} and \eqref{weight inequality2} results in $I_2^{\prime} \lesssim \sqrt{\widetilde{\mathcal{E}}_{N,l}(t)} \widetilde{\mathcal{D}}_{N,l}(t)$.

For the second case $2 \leq |\alpha_2|+|\beta_2| \leq N-2$, by $L^6$--$L^3$--$L^2$ H\"{o}lder inequality, we have
\begin{align*}
I_{2}^\prime \lesssim
\frac{1}{\varepsilon}
\left\| \widetilde{w}_{l}(\alpha,\beta) \partial^{\alpha_1}_{\beta_1} \mathbf{P}^{\perp}f^{\varepsilon} \right\|_{L^6_x L^2_{\gamma/2+s}}
\left\| \partial^{\alpha_2}_{\beta_2} \mathbf{P}^{\perp} f^{\varepsilon}\right\|_{L^3_x L^2_D}
\left\| \widetilde{w}_{l}(\alpha,\beta) \partial^{\alpha}_{\beta} \mathbf{P}^{\perp} f^{\varepsilon}\right\|_{D}.
\end{align*}
Noticing that $|\alpha_1+\alpha^\prime|+|\beta_1| \leq N-1$ for $|\alpha^\prime|=1$, we have  $\langle v \rangle ^{\frac{\gamma+2s}{2}} \widetilde{w}_l(\alpha,\beta)\leq \widetilde{w}_l(\alpha_1+\alpha^\prime,\beta_1)$, which yields that
\begin{align*}
\left\| \widetilde{w}_{l}(\alpha,\beta) \partial^{\alpha_1}_{\beta_1} \mathbf{P}^{\perp}f^{\varepsilon} \right\|_{L^6_x L^2_{\gamma/2+s}}
\lesssim \sum_{ |\alpha^\prime| =1}
\left\| \widetilde{w}_{l}(\alpha,\beta) \partial^{\alpha_1+\alpha^\prime}_{\beta_1} \mathbf{P}^{\perp} f^{\varepsilon}\right\|_{L^2_{\gamma/2+s}}
\lesssim \sqrt{\widetilde{\mathcal{E}}_{N,l}(t)}.
\end{align*}
From the above two estimates and \eqref{Sobolev-ineq}, we get $I_2^{\prime} \lesssim \sqrt{\widetilde{\mathcal{E}}_{N,l}(t)} \widetilde{\mathcal{D}}_{N,l}(t)$.

For the last case $ N-1 \leq |\alpha_2|+|\beta_2| \leq  N$, it can be easily proved that $|\alpha_1+ \alpha^\prime|+|\beta_1| \leq 3$ for $1 \leq |\alpha^\prime| \leq 2$. Making use of $L^\infty$--$L^2$--$L^2$ H\"{o}lder inequality, one gets
\begin{align*}
I_2^{\prime} \lesssim
\frac{1}{\varepsilon}
\left\| w_{l}(\alpha,\beta) \partial^{\alpha_1}_{\beta_1} \mathbf{P}^{\perp}f^{\varepsilon} \right\|_{L^\infty_x L^2_{\gamma/2+s}}
\left\| (1+t)^{-\frac{1+\varepsilon_0}{4}} \partial^{\alpha_2}_{\beta_2} \mathbf{P}^{\perp} f^{\varepsilon}\right\|_{D}
\left\| \widetilde{w}_{l}(\alpha,\beta) \partial^{\alpha}_{\beta} \mathbf{P}^{\perp} f^{\varepsilon}\right\|_{D},
\end{align*}
which, together with \eqref{weight inequality1}, leads to
$I_2^{\prime} \lesssim \sqrt{\widetilde{\mathcal{E}}_{N,l}(t)} \widetilde{\mathcal{D}}_{N,l}(t)$.
Here, we have employed the following estimate
\begin{align*}
\left\| w_{l}(\alpha,\beta) \partial^{\alpha_1}_{\beta_1} \mathbf{P}^{\perp}f^{\varepsilon} \right\|_{L^\infty_x L^2_{\gamma/2+s}}
\lesssim \sum_{ 1\leq |\alpha^\prime| \leq 2}
\left\| w_{l}(\alpha,\beta) \partial^{\alpha_1+\alpha^\prime}_{\beta_1} \mathbf{P}^{\perp} f^{\varepsilon}\right\|_{L^2_{\gamma/2+s}}
\lesssim \sqrt{\widetilde{\mathcal{E}}_{N,l}(t)}.
\end{align*}

For the term $I_2^{\prime\prime}$, we simplify our calculations by dividing them into three cases.
For the first case $|\alpha_2|+|\beta_2| =0$, we take $L^2$--$L^\infty$--$L^2$ H\"{o}lder inequality to obtain
\begin{align*}
I_{2}^{\prime \prime} \lesssim
\frac{1}{\varepsilon}
\left\| \widetilde{w}_{l}(\alpha,\beta) \partial^{\alpha}_{\beta} \mathbf{P}^{\perp}f^{\varepsilon} \right\|_{D}
\left\| e^{\frac{q \langle v \rangle}{(1+t)^\vartheta}}  \mathbf{P}^{\perp} f^{\varepsilon}\right\|_{L^\infty_x L^2_{\gamma/2+s}}
\left\| \widetilde{w}_{l}(\alpha,\beta) \partial^{\alpha}_{\beta} \mathbf{P}^{\perp} f^{\varepsilon}\right\|_{D}.
\end{align*}
By virtue of the fact $e^{\frac{q \langle v \rangle}{(1+t)^\vartheta}} \langle v \rangle^{\frac{\gamma+2s}{2}} \leq \widetilde{w}_l(\alpha^\prime,0)$ for $1 \leq |\alpha^\prime| \leq 2$, one has
\begin{align*}
\left\| e^{\frac{q \langle v \rangle}{(1+t)^\vartheta}} \mathbf{P}^{\perp} f^{\varepsilon}\right\|_{L^\infty_x L^2_{\gamma/2+s}}
\lesssim \sum_{1 \leq |\alpha^\prime| \leq 2}\left\| e^{\frac{q \langle v \rangle}{(1+t)^\vartheta}} \partial^{\alpha^\prime} \mathbf{P}^{\perp} f^{\varepsilon}\right\|_{ L^2_{\gamma/2+s}}
\lesssim \sqrt{\widetilde{\mathcal{E}}_{N,l}(t)}.
\end{align*}
Combining the two estimates above, we have $I_2^{\prime \prime} \lesssim \sqrt{\widetilde{\mathcal{E}}_{N,l}(t)} \widetilde{\mathcal{D}}_{N,l}(t)$.

For the case of $1 \leq |\alpha_2|+|\beta_2| \leq N-2$, it is easy to see that $|\alpha_2 +\alpha^\prime|+|\beta_2| \leq N-1$ for $|\alpha^\prime|=1$. Thereby, we have the inequality
$e^{\frac{q \langle v \rangle}{(1+t)^\vartheta}} \langle v \rangle^{\frac{\gamma+2s}{2}} \leq \widetilde{w}_l(\alpha_2+\alpha^\prime,\beta_2)$, which gives that
\begin{align*}
\left\| e^{\frac{q \langle v \rangle}{(1+t)^\vartheta}} \partial^{\alpha_2}_{\beta_2}\mathbf{P}^{\perp} f^{\varepsilon}\right\|_{L^6_x L^2_{\gamma/2+s}}
\lesssim \sum_{|\alpha^\prime| =1}\left\| e^{\frac{q \langle v \rangle}{(1+t)^\vartheta}} \partial^{\alpha_2+\alpha^\prime}_{\beta_2} \mathbf{P}^{\perp} f^{\varepsilon}\right\|_{ L^2_{\gamma/2+s}}
\lesssim \sqrt{\widetilde{\mathcal{E}}_{N,l}(t)}.
\end{align*}
As a consequence, we have $I_2^{\prime \prime} \lesssim \sqrt{\widetilde{\mathcal{E}}_{N,l}(t)} \widetilde{\mathcal{D}}_{N,l}(t)$ by using  $L^3$--$L^6$--$L^2$ H\"{o}lder inequality and \eqref{weight inequality2}.

As for the last case $N-1 \leq |\alpha_2|+|\beta_2| \leq N$, by $L^\infty$--$L^2$--$L^2$ H\"{o}lder inequality, $I_2^{\prime \prime}$ can be bounded by
\begin{align*}
\frac{1}{\varepsilon}
\left\| w_{l}(\alpha,\beta) \partial^{\alpha_1}_{\beta_1} \mathbf{P}^{\perp}f^{\varepsilon} \right\|_{L^\infty_x L^2_D}
\!\left\| (1+t)^{-\frac{1+\varepsilon_0}{4}} e^{\frac{q \langle v \rangle}{(1+t)^\vartheta}} \partial^{\alpha_2}_{\beta_2} \mathbf{P}^{\perp} f^{\varepsilon}\right\|_{L^2_{\gamma/2+s}}
\!\!\!\left\| \widetilde{w}_{l}(\alpha,\beta) \partial^{\alpha}_{\beta} \mathbf{P}^{\perp} f^{\varepsilon}\right\|_{D}.
\end{align*}
According to \eqref{norm inequality} and the fact
$(1+t)^{-\frac{1+\varepsilon_0}{4}} e^{\frac{q \langle v \rangle}{(1+t)^\vartheta}}  \leq \widetilde{w}_{l}(\alpha_2,\beta_2)$, we obtain
\begin{align*}
\left\| (1+t)^{-\frac{1+\varepsilon_0}{4}}e^{\frac{q \langle v \rangle}{(1+t)^\vartheta}} \partial^{\alpha_2}_{\beta_2}\mathbf{P}^{\perp} f^{\varepsilon}\right\|_{L^2_{\gamma/2+s}} \lesssim \sqrt{\widetilde{\mathcal{D}}_{N,l}(t)}.
\end{align*}
Meanwhile, we have $w_{l}(\alpha,\beta) \langle v \rangle ^{\frac{\gamma+2s}{2} } \leq \widetilde{w}_{l}(\alpha_1+\alpha^\prime,\beta_1+\beta^\prime)$ for $|\alpha_1|+|\beta_1| \leq 1$, $1 \leq |\alpha^\prime| \leq 2$ and $|\beta^\prime| \leq 1$, which, together with $|\cdot|_{D} \lesssim |\cdot|_{H^1_{\gamma/2+s}}$, implies that
\begin{align*}
\left\| w_{l}(\alpha,\beta) \partial^{\alpha_1}_{\beta_1} \mathbf{P}^{\perp}f^{\varepsilon} \right\|_{L^\infty_x L^2_D}
\lesssim \sum_{\substack{{1 \leq |\alpha^\prime| \leq 2 }\\{|\beta^\prime| \leq 1}}}
\left\| w_{l}(\alpha,\beta) \partial^{\alpha_1+\alpha^\prime}_{\beta_1+\beta^\prime} \mathbf{P}^{\perp}f^{\varepsilon} \right\|_{L^2_{\gamma/2+s}}
\lesssim \sqrt{\widetilde{\mathcal{E}}_{N,l}(t)}.
\end{align*}
The combination of the above three estimates gives that $I_2^{\prime \prime} \lesssim \sqrt{\widetilde{\mathcal{E}}_{N,l}(t)} \widetilde{\mathcal{D}}_{N,l}(t)$.

Consequently, collecting all the above estimates, \eqref{hard gamma7} follows. This completes the proof of Lemma \ref{hard Gamma}.
\end{proof}

\begin{remark}
In fact, within the framework of \cite{AMUXY2012JFA, DLYZ2013}, the estimates on the collision term $\Gamma(f^\varepsilon,f^\varepsilon)$ for the hard potential case $\gamma+2s \geq 0$ are more complex than for the soft potential case $\max\{-3, -\frac{3}{2}-2s\} < \gamma <-2s$, primarily due to the extra weight function $\langle v \rangle^\frac{\gamma+2s}{2}$.
This requires us to design the weight function $\overline{w}_l(\alpha,\beta)=\langle v \rangle ^{l-m|\alpha|-m|\beta|}$ with $m \geq \frac{\gamma+2s}{2}$.
\end{remark}

The following two lemmas concern the estimates on the nonlinear term $\frac{1}{\varepsilon} q_0 (v \times B^\varepsilon) \cdot \nabla_v f^{\varepsilon}$.
\begin{lemma}\label{soft B 1}
Let $\max\left\{-3,-\frac{3}{2}-2s\right\} < \gamma <-2s $, $0 < s < 1$ and $0 < \varepsilon \leq 1$.
\begin{itemize}
\setlength{\leftskip}{-6mm}
\item[(1)] For $1 \leq | \alpha | \leq N$, there holds
\begin{align}\label{soft B 1-1}
\frac{1}{\varepsilon} \left(\partial^{\alpha}(q_0(v \times B^\varepsilon) \cdot \nabla_v f^{\varepsilon}), \partial^\alpha f^{\varepsilon} \right)
\lesssim \left\{ \delta_0 + \sqrt{\mathcal{E}_{N}(t)} \right\} \widetilde{\mathcal{D}}_{N,l}(t).
\end{align}
\item[(2)] For $ | \alpha |+ |\beta| \leq N-1$,  there holds
\begin{align}\label{soft B 1-2}
\frac{1}{\varepsilon}  \left(\partial^{\alpha}_{\beta}(q_0(v \times B^\varepsilon) \cdot \nabla_v \mathbf{P}^{\perp}f^{\varepsilon}), \widetilde{w}^2_l(\alpha,\beta)\partial^\alpha_\beta \mathbf{P}^{\perp} f^{\varepsilon} \right)
\lesssim \sqrt{\mathcal{E}_{N}(t)} \widetilde{\mathcal{D}}_{N,l}(t).
\end{align}
\item[(3)] For $ | \alpha | = N$,  there holds
\begin{align}\label{soft B 1-3}
\left(\partial^{\alpha}(q_0(v \times B^\varepsilon) \cdot \nabla_v f^{\varepsilon}),  \widetilde{w}^2_l(\alpha,0) \partial^\alpha f^{\varepsilon} \right)
\lesssim \sqrt{\mathcal{E}_{N}(t)} \widetilde{\mathcal{D}}_{N,l}(t).
\end{align}
\item[(4)] For $| \alpha | + | \beta | = N$, $| \beta | \geq 1$ and
$| \alpha | \leq N-1$,  there holds
\begin{align}\label{soft B 1-4}
\frac{1}{\varepsilon}  \left(\partial^{\alpha}_{\beta}(q_0(v \times B^\varepsilon) \cdot \nabla_v \mathbf{P}^{\perp}f^{\varepsilon}), \widetilde{w}^2_l(\alpha,\beta)\partial^\alpha_\beta \mathbf{P}^{\perp} f^{\varepsilon} \right)
\lesssim \sqrt{\mathcal{E}_{N}(t)} \widetilde{\mathcal{D}}_{N,l}(t).
\end{align}
\end{itemize}
\end{lemma}

\begin{proof}
For brevity, we only give the proof of \eqref{soft B 1-1} and \eqref{soft B 1-4}. The estimates \eqref{soft B 1-2} and \eqref{soft B 1-3} can be proved by using the similar argument.

We first estimate \eqref{soft B 1-1}. In terms of the macro-micro decomposition \eqref{f decomposition}, we set
\begin{align*}
I_3 \equiv \;& \frac{1}{\varepsilon} \left(\partial^{\alpha}(q_0(v \times B^\varepsilon) \cdot \nabla_v f^{\varepsilon}), \partial^\alpha f^{\varepsilon} \right)\\
=\;&\frac{1}{\varepsilon} \sum_{0 \neq \alpha_1 \leq \alpha } C_\alpha^{\alpha_1} \left( q_0(v \times \partial^{\alpha_1} B^\varepsilon) \cdot \nabla_v \partial^{\alpha-\alpha_1}\mathbf{P} f^{\varepsilon}, \partial^\alpha \mathbf{P} f^{\varepsilon} \right)\\
&+\frac{1}{\varepsilon} \sum_{0 \neq \alpha_1 \leq \alpha } C_\alpha^{\alpha_1} \left( q_0(v \times \partial^{\alpha_1} B^\varepsilon) \cdot \nabla_v \partial^{\alpha-\alpha_1}\mathbf{P} f^{\varepsilon}, \partial^\alpha \mathbf{P}^{\perp} f^{\varepsilon} \right)\\
&+\frac{1}{\varepsilon} \sum_{0 \neq \alpha_1 \leq \alpha } C_\alpha^{\alpha_1} \left( q_0(v \times \partial^{\alpha_1} B^\varepsilon) \cdot \nabla_v \partial^{\alpha-\alpha_1} \mathbf{P}^{\perp} f^{\varepsilon}, \partial^\alpha \mathbf{P} f^{\varepsilon} \right)\\
&+\frac{1}{\varepsilon} \sum_{0 \neq \alpha_1 \leq \alpha } C_\alpha^{\alpha_1} \left( q_0(v \times \partial^{\alpha_1} B^\varepsilon) \cdot \nabla_v \partial^{\alpha-\alpha_1} \mathbf{P}^{\perp} f^{\varepsilon}, \partial^\alpha \mathbf{P}^{\perp}f^{\varepsilon} \right)\\
:=\;&I_{3,1}+I_{3,2}+I_{3,3}+I_{3,4}.
\end{align*}
For the term $I_{3,1}$, it follows from the kernel structure of $\mathbf{P}$ that $I_{3,1}=0$, cf. \cite[pp. 28]{JL2019}. For the term $I_{3,2}$, we obtain
\begin{align*}
 I_{3,2} \lesssim\;& \frac{1}{\varepsilon} \sum_{0 \neq \alpha_1 \leq \alpha } \int_{\mathbb{R}^3} \left|\partial^{\alpha_1} B^\varepsilon \right|
 \left|\partial^{\alpha-\alpha_1}(a^\varepsilon_{\pm}, b^\varepsilon, c^\varepsilon)\right| \left|\partial^\alpha \mathbf{P}^{\perp} f^\varepsilon \right|_{D} \d x.
\end{align*}
If $|\alpha_1|=1$, we get from $L^3$--$L^6$--$L^2$ H\"{o}lder inequality and \eqref{Sobolev-ineq} that
\begin{align*}
  I_{3,2} \lesssim \frac{1}{\varepsilon}   \left\|\partial^{\alpha_1} B^\varepsilon \right\|_{L^3}
 \left\|\partial^{\alpha-\alpha_1}(a^\varepsilon_{\pm}, b^\varepsilon, c^\varepsilon)\right\|_{L^6} \left\|\partial^\alpha \mathbf{P}^{\perp} f^\varepsilon \right\|_{D}
  \lesssim \sqrt{\mathcal{E}_{N}(t)} \mathcal{D}_{N}(t).
\end{align*}
If $2 \leq |\alpha_1| \leq |\alpha| \leq N$, by $L^2$--$L^\infty$--$L^2$ H\"{o}lder inequality and \eqref{Sobolev-ineq}, one obtains
\begin{align*}
   I_{3,2} \lesssim \frac{1}{\varepsilon}  \left\|\partial^{\alpha_1} B^\varepsilon \right\|
 \left\|\partial^{\alpha-\alpha_1}(a^\varepsilon_{\pm}, b^\varepsilon, c^\varepsilon)\right\|_{L^\infty}
 \left\|\partial^\alpha \mathbf{P}^{\perp} f^\varepsilon \right\|_{D}
 \lesssim  \sqrt{\mathcal{E}_{N}(t)} \mathcal{D}_{N}(t). \nonumber
\end{align*}
For the term $I_{3,3}$, similar to the term $I_{3,2}$, we can deduce from integration by parts with respect to velocity that $I_{3,3} \lesssim \sqrt{\mathcal{E}_{N}(t)} \mathcal{D}_{N}(t)$.
Finally, for the term $I_{3,4}$, we have
\begin{align*}
I_{3,4} \lesssim \;&\frac{1}{\varepsilon} \sum_{0 \neq \alpha_1 \leq \alpha } \int_{\mathbb{R}^3} \left|\partial^{\alpha_1} B^\varepsilon \right|
 \left|\langle v \rangle ^{1-\frac{\gamma+2s}{2}} \partial^{\alpha-\alpha_1}_{e_i} \mathbf{P}^{\perp} f^\varepsilon \right|_{L^2} \left|\partial^\alpha \mathbf{P}^{\perp} f^\varepsilon \right|_{L^2_{\gamma/2+s}} \d x,
\end{align*}
where we have used the notation $e_i$ to denote the multi-index with the $i$-th element unit and the rest zero. Now, we divide our computations into three cases.
For the case $|\alpha_1|=1$, by applying $L^\infty$--$L^2$--$L^2$ H\"{o}lder inequality, the fact
$(1+t)^{-\frac{\varrho+1}{2}}\langle v \rangle ^{1-\frac{\gamma+2s}{2}}
\leq \langle v \rangle^{\frac{\gamma+2s}{2}} \widetilde{w}_{l}(\alpha-\alpha_1, e_i)$ for any $1 \leq |\alpha| \leq N$, \eqref{Sobolev-ineq} and \eqref{assumption1},
we can derive that
\begin{align*}
I_{3,4} \lesssim \frac{1}{\varepsilon} \left\|\partial^{\alpha_1} B^\varepsilon \right\|_{L^\infty}
\left\|\langle v \rangle ^{1-\frac{\gamma+2s}{2}} \partial^{\alpha-\alpha_1}_{e_i} \mathbf{P}^{\perp} f^\varepsilon \right\|
\left\|\partial^\alpha \mathbf{P}^{\perp} f^\varepsilon \right\|_D
\lesssim \delta_0 \widetilde{\mathcal{D}}_{N,l}(t).
\end{align*}
For the case of $2 \leq |\alpha_1| \leq N-2$, noting the fact that
$\langle v \rangle ^{1-\frac{\gamma+2s}{2}} \leq \langle v \rangle^{\frac{\gamma+2s}{2}} \widetilde{w}_{l}(\alpha-\alpha_1, e_i)$
 for $|\alpha-\alpha_1|+1 \leq N-1$, by $L^\infty$--$L^2$--$L^2$ H\"{o}lder inequality and \eqref{Sobolev-ineq}, one has
$I_{3,4} \lesssim \sqrt{\mathcal{E}_{N}(t)} \widetilde{\mathcal{D}}_{N,l}(t)$. It should be noted that this case does not require the assumption \eqref{assumption1}, which is also the difference from the case $|\alpha_1|=1$.
For the last case $N-1 \leq |\alpha_1| \leq N$, employing $L^2$--$L^\infty$--$L^2$ H\"{o}lder inequality and \eqref{Sobolev-ineq}, it is easy to see
$I_{3,4} \lesssim \sqrt{\mathcal{E}_{N}(t)}  \widetilde{\mathcal{D}}_{N,l}(t)$,
where we used the fact
$\langle v \rangle ^{1-\frac{\gamma+2s}{2}} \leq \langle v \rangle^{\frac{\gamma+2s}{2}} \widetilde{w}_{l}(\alpha-\alpha_1+\alpha^\prime, e_i)$
 for $1\leq |\alpha^\prime| \leq 2$.
Collecting all the estimates of $I_{3,1}$--$I_{3,4}$, $\eqref{soft B 1-1}$ is thus proved.

Then we turn to estimate \eqref{soft B 1-4}. For brevity, we denote
$$
I_4 \equiv \frac{1}{\varepsilon}  \sum_{\substack{{\alpha_1\leq \alpha}\\{|\beta_1|  \leq 1}}} \left( q_0(\partial_{\beta_1}v \times \partial^{\alpha_1}B^\varepsilon) \cdot \nabla_v \partial^{\alpha-\alpha_1}_{\beta-\beta_1} \mathbf{P}^{\perp}f^{\varepsilon}, \widetilde{w}^2_l(\alpha,\beta)\partial^\alpha_\beta \mathbf{P}^{\perp} f^{\varepsilon} \right).
$$
If  $|\alpha_1|+|\beta_1|=0$, we can deduce from integration by parts with respect to velocity and the fact $(v\times B^\varepsilon) \cdot v=0$ that $I_4=0$.
If $|\alpha_1|+|\beta_1| \geq 1$ with $|\beta_1|=0$ and $|\alpha_1|\geq 1$, we need to make use of the interpolation inequality \eqref{interpolation}. More precisely, recalling $\overline{w}_l(\alpha,\beta)$ for soft potentials $\max\{-3, -\frac{3}{2}-2s\} < \gamma < -2s$ in \eqref{weight function2}, we first choose
$$
\langle v \rangle^\ell =\{\overline{w}_l(|\alpha|-1,|\beta|)\}^{1-s}\{\overline{w}_l(|\alpha|-1,|\beta|+1)\}^{-(1-s)},
$$
which implies that
$$
\langle v \rangle^{-\frac{\ell s}{1-s}} =\{\overline{w}_l(|\alpha|-1,|\beta|)\}^{-s}\{\overline{w}_l(|\alpha|-1,|\beta|+1)\}^{s}.
$$
Substituting them into \eqref{interpolation}, we obtain from \eqref{equivalent norm1} and \eqref{norm inequality} that
\begin{align}\label{inter H1}
&\left|\langle v\rangle^{1-\frac{\gamma}{2}} w_l(\alpha, \beta) \partial^{\alpha-\alpha_1}_{\beta}\mathbf{P}^{\perp} f^{\varepsilon}\right|_{H^1} \nonumber\\
%----------------------------------
\lesssim\;&\left|\langle v \rangle ^\ell\left(\langle v\rangle^{1-\frac{\gamma}{2}} w_l(\alpha, \beta) \partial^{\alpha-\alpha_1}_{\beta}\mathbf{P}^{\perp} f^{\varepsilon}\right)\right|_{H^{s}}
%-------------------------------------
\!\!+\left|\langle v \rangle ^{-\frac{\ell s}{1-s}} \left(\langle v\rangle^{1-\frac{\gamma}{2}} w_l(\alpha, \beta) \partial^{\alpha-\alpha_1}_{\beta}\mathbf{P}^{\perp} f^{\varepsilon}\right)\right|_{H^{1+s}} \nonumber\\
  %----------------------
 \lesssim \;& \left| \langle v \rangle ^\frac{\gamma}{2} w_l(|\alpha|-1,|\beta|) \partial^{\alpha-\alpha_1}_{\beta}\mathbf{P}^{\perp} f^\varepsilon\right|_{H^s}
+\left| \langle v \rangle ^\frac{\gamma}{2} w_l(|\alpha|-1,|\beta|+1) \partial^{\alpha-\alpha_1}_{\beta}\mathbf{P}^{\perp} f^{\varepsilon} \right|_{H^{1+s}} \nonumber\\
\lesssim \;& \left| w_l(|\alpha|-1, |\beta|) \partial^{\alpha-\alpha_1}_{\beta}\mathbf{P}^{\perp} f^\varepsilon\right|_{H^s_{\gamma/2}}
+\left| w_l(|\alpha|-1,|\beta|+1) \partial^{\alpha-\alpha_1}_{\beta+e_i}\mathbf{P}^{\perp} f^{\varepsilon} \right|_{H^{s}_{\gamma/2}} \nonumber\\
\lesssim \;&\left| w_l(|\alpha|-1, |\beta|) \partial^{\alpha-\alpha_1}_{\beta}\mathbf{P}^{\perp} f^\varepsilon\right|_{D}
+\left| w_l(|\alpha|-1,|\beta|+1) \partial^{\alpha-\alpha_1}_{\beta+e_i}\mathbf{P}^{\perp} f^{\varepsilon} \right|_{D},
\end{align}
where we have applied the facts
$$\langle v \rangle ^{1-\frac{\gamma}{2}}\overline{w}_l(\alpha, \beta)\leq \langle v \rangle^{\frac{\gamma}{2}} \{\overline{w}_l(|\alpha|-1, |\beta|)\}^s\{\overline{w}_l(|\alpha|-1,|\beta|+1)\}^{1-s}
$$
and
$$\left|\nabla_v \left(\langle v \rangle^{\frac{\gamma}{2}}w_{l}(|\alpha|-1, |\beta|+1) \right)\right|
\lesssim \langle v \rangle^{\frac{\gamma}{2}}w_{l}(|\alpha|-1, |\beta|+1)
\lesssim \langle v \rangle^{\frac{\gamma}{2}}w_{l}(|\alpha|-1, |\beta|).
$$
Combining with \eqref{inter H1}, one has
\begin{align*}
 I_4 \lesssim\;& \frac{1}{\varepsilon} \sum_{1 \leq |\alpha_1| \leq N-1}\int_{\mathbb{R}^3} \left|\partial^{\alpha_1} B^\varepsilon \right|
  \left|\langle v\rangle^{1-\frac{\gamma}{2}} \widetilde{w}_l(\alpha, \beta) \partial^{\alpha-\alpha_1}_{\beta+e_i}\mathbf{P}^{\perp} f^{\varepsilon}\right|_{L^2}
  \left|\widetilde{w}_{l}(\alpha,\beta)\partial^\alpha_\beta \mathbf{P}^{\perp}f^\varepsilon \right|_{L^2_{\gamma/2}} \d x \\
  \lesssim\;& \frac{1}{\varepsilon} \sum_{1 \leq |\alpha_1| \leq N-1}\int_{\mathbb{R}^3} \left|\partial^{\alpha_1} B^\varepsilon \right|
  \left|\langle v\rangle^{1-\frac{\gamma}{2}} \widetilde{w}_l(\alpha, \beta) \partial^{\alpha-\alpha_1}_{\beta}\mathbf{P}^{\perp} f^{\varepsilon}\right|_{H^1}
  \left|\widetilde{w}_{l}(\alpha,\beta)\partial^\alpha_\beta \mathbf{P}^{\perp}f^\varepsilon \right|_{L^2_{\gamma/2}} \d x \\
  \lesssim\;& \frac{1}{\varepsilon} \sum_{1 \leq |\alpha_1| \leq N-1}\int_{\mathbb{R}^3} \left|\partial^{\alpha_1} B^\varepsilon \right|
  \left| \widetilde{w}_l(|\alpha|-1, |\beta|) \partial^{\alpha-\alpha_1}_{\beta}\mathbf{P}^{\perp} f^\varepsilon\right|_{D}
  \left|\widetilde{w}_{l}(\alpha,\beta)\partial^\alpha_\beta \mathbf{P}^{\perp}f^\varepsilon \right|_{D} \d x \\
  &\!\!\!+\frac{1}{\varepsilon} \sum_{1 \leq |\alpha_1| \leq N-1}\int_{\mathbb{R}^3} \left|\partial^{\alpha_1} B^\varepsilon \right|
  \left| \widetilde{w}_l(|\alpha|-1, |\beta|+1) \partial^{\alpha-\alpha_1}_{\beta+e_i}\mathbf{P}^{\perp} f^\varepsilon\right|_{D}
  \left|\widetilde{w}_{l}(\alpha,\beta)\partial^\alpha_\beta \mathbf{P}^{\perp}f^\varepsilon \right|_{D} \d x \\
  :=\;&I_{4,1}+I_{4,2}.
\end{align*}
For the term $I_{4,1}$, one has from the H\"{o}lder inequality, \eqref{Sobolev-ineq}, \eqref{minkowski} and \eqref{weight inequality2} that
\begin{align*}
 I_{4,1} \lesssim\;& \frac{1}{\varepsilon} \sum_{|\alpha_1|=1} \left\|\partial^{\alpha_1} B^\varepsilon \right\|_{L^\infty}
 \left\| \widetilde{w}_l(|\alpha|-1, |\beta|) \partial^{\alpha-\alpha_1}_{\beta}\mathbf{P}^{\perp} f^\varepsilon\right\|_{D}
 \left\|\widetilde{w}_{l}(\alpha,\beta)\partial^\alpha_\beta \mathbf{P}^{\perp}f^\varepsilon \right\|_{D} \\
&+ \frac{1}{\varepsilon} \sum_{2 \leq |\alpha_1|\leq N-1} \!\! \left\|\partial^{\alpha_1} B^\varepsilon \right\|_{L^6}
 \left\| \widetilde{w}_l(|\alpha|-1, |\beta|) \partial^{\alpha-\alpha_1}_{\beta}\mathbf{P}^{\perp} f^\varepsilon\right\|_{L^3_x L^2_D}
 \left\|\widetilde{w}_{l}(\alpha,\beta)\partial^\alpha_\beta \mathbf{P}^{\perp}f^\varepsilon \right\|_{D} \\
 \lesssim\;& \sqrt{\mathcal{E}_{N}(t)}\widetilde{\mathcal{D}}_{N,l}(t).
\end{align*}
The term $I_{4,2}$ can be treated in the similar way as the term $I_{4,1}$ to get the same upper bound.

It remains to consider the case $|\alpha_1| + |\beta_1| \geq  1$ with $|\beta_1| = 1$ and $|\alpha_1| \geq 0$. By using the similar way as the previous case, we can get $I_{4} \lesssim \sqrt{\mathcal{E}_{N}(t)}\widetilde{\mathcal{D}}_{N,l}(t)$.

Collecting all the estimates of $I_4$, one can see that \eqref{soft B 1-4} holds true. Therefore, Lemma \ref{soft B 1} is proved.
\end{proof}

\begin{lemma}\label{hard B 1}
Let $\gamma+2s\geq 0 $, $0 < s < 1$ and $0 < \varepsilon \leq 1$.
\begin{itemize}
\setlength{\leftskip}{-6mm}
\item[(1)] For $1 \leq | \alpha | \leq N$, there holds
\begin{align}\label{hard B 1-1}
 \frac{1}{\varepsilon} \left(\partial^{\alpha}(q_0(v \times B^\varepsilon) \cdot \nabla_v f^{\varepsilon}), \partial^\alpha f^{\varepsilon} \right)
\lesssim \left\{ \delta_0 + \sqrt{\mathcal{E}_{N}(t)} \right\} \widetilde{\mathcal{D}}_{N,l}(t).
\end{align}
\item[(2)] For $ | \alpha |+ |\beta| \leq N-1$,  there holds
\begin{align}\label{hard B 1-2}
 \frac{1}{\varepsilon}  \left(\partial^{\alpha}_{\beta}(q_0(v \times B^\varepsilon) \cdot \nabla_v \mathbf{P}^{\perp}f^{\varepsilon}), \widetilde{w}^2_l(\alpha,\beta)\partial^\alpha_\beta \mathbf{P}^{\perp} f^{\varepsilon} \right)
\lesssim \sqrt{\mathcal{E}_{N}(t)} \widetilde{\mathcal{D}}_{N,l}(t).
\end{align}
\item[(3)] For $ | \alpha | = N$,  there holds
\begin{align}\label{hard B 1-3}
\left(\partial^{\alpha}(q_0(v \times B^\varepsilon) \cdot \nabla_v f^{\varepsilon}),  \widetilde{w}^2_l(\alpha,0) \partial^\alpha f^{\varepsilon} \right)
\lesssim \sqrt{\mathcal{E}_{N}(t)} \widetilde{\mathcal{D}}_{N,l}(t).
\end{align}
\item[(4)] For $| \alpha | + | \beta | = N$, $| \beta | \geq 1$ and
$| \alpha | \leq N-1$,  there holds
\begin{align}\label{hard B 1-4}
 \frac{1}{\varepsilon}  \left(\partial^{\alpha}_{\beta}(q_0(v \times B^\varepsilon) \cdot \nabla_v \mathbf{P}^{\perp}f^{\varepsilon}), \widetilde{w}^2_l(\alpha,\beta)\partial^\alpha_\beta \mathbf{P}^{\perp} f^{\varepsilon} \right)
\lesssim \sqrt{\mathcal{E}_{N}(t)} \widetilde{\mathcal{D}}_{N,l}(t).
\end{align}
\end{itemize}
\end{lemma}

\begin{proof}
This lemma can be proved similarly as that of Lemma \ref{soft B 1}. The only difference is that we design the weight function $\overline{w}_l(\alpha,\beta)=\langle v \rangle ^{l-m|\alpha|-m|\beta|}$ with $m \geq \frac{1}{s}$ such that the inequality
$\langle v \rangle \overline{w}_{l}(\alpha,\beta) \leq \{\overline{w}_{l}(|\alpha|-1,|\beta|)\}^{s}\{\overline{w}_{l}(|\alpha|-1,|\beta|+1)\}^{1-s}$ holds.
In fact, to overcome the singularity $\frac{1}{\varepsilon}$ and the velocity growth in the above estimates, we have to use the interpolation inequality \eqref{interpolation} again. For brevity, we omit the details.
\end{proof}

The following two lemmas concern the estimates on the nonlinear term $ E^\varepsilon \cdot v f^{\varepsilon}$.
\begin{lemma}\label{soft E 1}
Let $\max\left\{-3,-\frac{3}{2}-2s\right\} < \gamma <-2s $, $0 < s < 1$ and $0 < \varepsilon \leq 1$.
\begin{itemize}
\setlength{\leftskip}{-6mm}
\item[(1)] For $ | \alpha | \leq N$, there holds
\begin{align}\label{soft E 1-1}
\left( \partial^{\alpha} (E^\varepsilon \cdot v f^\varepsilon ),  \partial^\alpha f^{\varepsilon} \right)
\lesssim \left\{\delta_0 + \sqrt{\mathcal{E}_{N}(t)}\right\} \widetilde{\mathcal{D}}_{N,l}(t).
\end{align}
\item[(2)] For $ | \alpha |+|\beta| \leq N-1$,  there holds
\begin{align}\label{soft E 1-2}
\left( \partial^{\alpha}_{\beta} (E^\varepsilon \cdot v \mathbf{P}^{\perp} f^\varepsilon ),  \widetilde{w}^2_{l}(\alpha,\beta)\partial^\alpha_\beta \mathbf{P}^{\perp} f^{\varepsilon} \right)
\lesssim \left\{\delta_0 + \sqrt{\mathcal{E}_{N}(t)}\right\} \widetilde{\mathcal{D}}_{N,l}(t).
\end{align}
\item[(3)] For $ | \alpha | = N$,  there holds
\begin{align}\label{soft E 1-3}
\varepsilon \left( \partial^{\alpha} (E^\varepsilon \cdot v  f^\varepsilon ),  \widetilde{w}^2_{l}(\alpha, 0)\partial^\alpha f^{\varepsilon} \right)
\lesssim \left\{\delta_0 + \sqrt{\mathcal{E}_{N}(t)}\right\} \widetilde{\mathcal{D}}_{N,l}(t).
\end{align}
\item[(4)] For $| \alpha | + | \beta | = N$, $| \beta | \geq 1$ and
$| \alpha | \leq N-1$,  there holds
\begin{align}\label{soft E 1-4}
\left( \partial^{\alpha}_{\beta} (E^\varepsilon \cdot v \mathbf{P}^{\perp} f^\varepsilon ),  \widetilde{w}^2_{l}(\alpha,\beta)\partial^\alpha_\beta \mathbf{P}^{\perp} f^{\varepsilon} \right)
\lesssim \left\{\delta_0 + \sqrt{\mathcal{E}_{N}(t)}\right\} \widetilde{\mathcal{D}}_{N,l}(t).
\end{align}
\end{itemize}
\end{lemma}

\begin{proof}
For the sake of simplicity, we only prove \eqref{soft E 1-4}. The estimates \eqref{soft E 1-1}--\eqref{soft E 1-3} exclusively concerning $\mathbf{P}^{\perp} f^\varepsilon$ can be proved similarly, while those concerning $\mathbf{P}f^\varepsilon$ can be proved using the similar method as that of $I_{3,2}$ in Lemma \ref{soft B 1}.

To prove \eqref{soft E 1-4}, for simplicity's sake, we write
$$
I_5 \equiv \left( \partial^{\alpha_1} E^\varepsilon \cdot \partial_{\beta_1}v \partial^{\alpha-\alpha_1}_{\beta-\beta_1}\mathbf{P}^{\perp} f^\varepsilon ,  \widetilde{w}^2_{l}(\alpha,\beta)\partial^\alpha_\beta \mathbf{P}^{\perp} f^{\varepsilon} \right).
$$
If $|\alpha_1|+|\beta_1|=0$, we can derive from $L^\infty$--$L^2$--$L^2$ H\"{o}lder inequality, \eqref{Sobolev-ineq}, \eqref{assumption1} and the fact $(1+t)^{-\frac{1+\varrho}{2}} \leq (1+t)^{-1-\vartheta}$ for $0< \vartheta \leq \frac{\varrho-1}{2}$ that
\begin{align*}
I_5 \lesssim\;& \left\|E^\varepsilon\right\|_{L^\infty} \left\| \langle v \rangle^{\frac{1}{2}}\widetilde{w}^2_{l}(\alpha,\beta)\partial^\alpha_\beta \mathbf{P}^{\perp} f^{\varepsilon}\right\|^2 \\
\lesssim\;& (1+t)^{-\frac{1+\varrho}{2}} \delta_0 \left\| \langle v \rangle^{\frac{1}{2}}\widetilde{w}^2_{l}(\alpha,\beta)\partial^\alpha_\beta \mathbf{P}^{\perp} f^{\varepsilon}\right\|^2 \\
\lesssim\;& \delta_0 \widetilde{\mathcal{D}}_{N,l}(t).
\end{align*}
If $|\alpha_1|+|\beta_1|=1$, by using $L^\infty$--$L^2$--$L^2$ H\"{o}lder inequality and \eqref{Sobolev-ineq}, we have

\begin{align*}
I_5 \lesssim\;& \int_{\mathbb{R}^3} \left|\partial^{\alpha_1} E^\varepsilon\right|
\left| \langle v \rangle^{1-\frac{\gamma+2s}{2}}\widetilde{w}_{l}(\alpha,\beta)\partial^{\alpha-\alpha_1}_{\beta-\beta_1} \mathbf{P}^{\perp} f^{\varepsilon}\right |_{L^2}
\left| \widetilde{w}_{l}(\alpha,\beta)\partial^\alpha_\beta \mathbf{P}^{\perp} f^{\varepsilon}\right|_{D} \d x \\
\lesssim\;&\left\|\partial^{\alpha_1} E^\varepsilon\right\|_{L^\infty}
\left\| \widetilde{w}_{l}(\alpha-\alpha_1,\beta-\beta_1)\partial^{\alpha-\alpha_1}_{\beta-\beta_1} \mathbf{P}^{\perp} f^{\varepsilon}\right \|_{D}
\left\| \widetilde{w}_{l}(\alpha,\beta)\partial^\alpha_\beta \mathbf{P}^{\perp} f^{\varepsilon}\right\|_{D} \\
\lesssim\;& \sqrt{\mathcal{E}_{N}(t)} \widetilde{\mathcal{D}}_{N,l}(t).
\end{align*}
Here, we used the fact
$\langle v \rangle ^{1-\frac{\gamma+2s}{2}}\widetilde{w}_{l}(\alpha,\beta) \leq \langle v \rangle^{\frac{\gamma+2s}{2}}\widetilde{w}_{l}(\alpha-\alpha_1,\beta-\beta_1)$.
For the remaining case $2 \leq |\alpha_1|+|\beta_1| \leq N$, in view of the estimate
$\langle v \rangle ^{1-\frac{\gamma+2s}{2}}\widetilde{w}_{l}(\alpha,\beta) \leq \langle v \rangle^{\frac{\gamma+2s}{2}}\widetilde{w}_{l}(\alpha-\alpha_1+\alpha^\prime,\beta-\beta_1)$ for $|\alpha^\prime| \leq 1$, one has
\begin{align*}
I_5 \lesssim\;& \int_{\mathbb{R}^3} \left|\partial^{\alpha_1} E^\varepsilon\right|
\left| \langle v \rangle^{1-\frac{\gamma+2s}{2}}\widetilde{w}_{l}(\alpha,\beta)\partial^{\alpha-\alpha_1}_{\beta-\beta_1} \mathbf{P}^{\perp} f^{\varepsilon}\right |_{L^2}
\left| \widetilde{w}_{l}(\alpha,\beta)\partial^\alpha_\beta \mathbf{P}^{\perp} f^{\varepsilon}\right|_{D} \d x \\
\lesssim\;&\left\|\partial^{\alpha_1} E^\varepsilon\right\|_{L^6}
\left\| \widetilde{w}_{l}(\alpha-\alpha_1+\alpha^\prime,\beta-\beta_1)\partial^{\alpha-\alpha_1}_{\beta-\beta_1} \mathbf{P}^{\perp} f^{\varepsilon}\right \|_{L^3_x L^2_D}
\left\| \widetilde{w}_{l}(\alpha,\beta)\partial^\alpha_\beta \mathbf{P}^{\perp} f^{\varepsilon}\right\|_{D} \\
\lesssim\;& \sqrt{\mathcal{E}_{N}(t)} \widetilde{\mathcal{D}}_{N,l}(t),
\end{align*}
where we have employed $L^6$--$L^3$--$L^2$ H\"{o}lder inequality, \eqref{norm inequality}, \eqref{Sobolev-ineq} and \eqref{minkowski}.

Therefore, collecting the above estimates, we prove \eqref{soft E 1-4}. This completes the proof of Lemma \ref{soft E 1}.
\end{proof}

\begin{lemma}\label{hard E 1}
Let $\gamma + 2s \geq 0 $, $0 < s < 1$ and $0 < \varepsilon \leq 1$.
\begin{itemize}
\setlength{\leftskip}{-6mm}
\item[(1)] For $ | \alpha | \leq N$, there holds
\begin{align}\label{hard E 1-1}
\left( \partial^{\alpha} (E^\varepsilon \cdot v f^\varepsilon ),  \partial^\alpha f^{\varepsilon} \right)
\lesssim \left\{\delta_0 + \sqrt{\mathcal{E}_{N}(t)}\right\} \widetilde{\mathcal{D}}_{N,l}(t).
\end{align}
\item[(2)] For $ | \alpha |+|\beta| \leq N-1$,  there holds
\begin{align}\label{hard E 1-2}
\left( \partial^{\alpha}_{\beta} (E^\varepsilon \cdot v \mathbf{P}^{\perp} f^\varepsilon ),  \widetilde{w}^2_{l}(\alpha,\beta)\partial^\alpha_\beta \mathbf{P}^{\perp} f^{\varepsilon} \right)
\lesssim \left\{\delta_0 + \sqrt{\mathcal{E}_{N}(t)}\right\} \widetilde{\mathcal{D}}_{N,l}(t).
\end{align}
\item[(3)] For $ | \alpha | = N$,  there holds
\begin{align}\label{hard E 1-3}
\varepsilon \left( \partial^{\alpha} (E^\varepsilon \cdot v  f^\varepsilon ),  \widetilde{w}^2_{l}(\alpha, 0)\partial^\alpha f^{\varepsilon} \right)
\lesssim \left\{\delta_0 + \sqrt{\mathcal{E}_{N}(t)}\right\} \widetilde{\mathcal{D}}_{N,l}(t).
\end{align}
\item[(4)] For $| \alpha | + | \beta | = N$, $| \beta | \geq 1$ and
$| \alpha | \leq N-1$,  there holds
\begin{align}\label{hard E 1-4}
\left( \partial^{\alpha}_{\beta} (E^\varepsilon \cdot v \mathbf{P}^{\perp} f^\varepsilon ),  \widetilde{w}^2_{l}(\alpha,\beta)\partial^\alpha_\beta \mathbf{P}^{\perp} f^{\varepsilon} \right)
\lesssim \left\{\delta_0 + \sqrt{\mathcal{E}_{N}(t)}\right\} \widetilde{\mathcal{D}}_{N,l}(t).
\end{align}
\end{itemize}
\end{lemma}

\begin{proof}
This lemma can be proved using the same argument as that of Lemma \ref{soft E 1}, and the details are omitted for brevity.
\end{proof}

The following two lemmas concern the estimates on the nonlinear term $ E^\varepsilon \cdot \nabla_v f^{\varepsilon}$.
\begin{lemma}\label{soft E 2}
Let $ \max\left\{-3,-\frac{3}{2}-2s\right\} < \gamma <-2s$, $0 < s < 1$ and $0 < \varepsilon \leq 1$.
\begin{itemize}
\setlength{\leftskip}{-6mm}
\item[(1)] For $1 \leq | \alpha | \leq N$, there holds
\begin{align}\label{soft E 2-1}
\left( \partial^\alpha ( E^\varepsilon \cdot \nabla_v f^\varepsilon ),  \partial^\alpha f^{\varepsilon} \right)
\lesssim \left\{ \delta_0 + \sqrt{\mathcal{E}_{N}(t)}\right\} \widetilde{\mathcal{D}}_{N,l}(t).
\end{align}
\item[(2)] For $| \alpha | + |\beta| \leq N-1$, there holds
\begin{align}\label{soft E 2-2}
\left( \partial^\alpha_{\beta} ( E^\varepsilon \cdot \nabla_v \mathbf{P}^{\perp}f^\varepsilon ),  \widetilde{w}^2_{l}(\alpha,\beta)\partial^\alpha_\beta \mathbf{P}^{\perp}f^{\varepsilon} \right)
\lesssim \left\{ \delta_0 + \sqrt{\mathcal{E}_{N}(t)}\right\} \widetilde{\mathcal{D}}_{N,l}(t).
\end{align}
\item[(3)] For $ | \alpha | = N$, there holds
\begin{align}\label{soft E 2-3}
\varepsilon \left( \partial^\alpha ( E^\varepsilon \cdot \nabla_v f^\varepsilon),
\widetilde{w}^2_{l}(\alpha,0)\partial^\alpha f^{\varepsilon} \right)
\lesssim \left\{ \delta_0 + \sqrt{\mathcal{E}_{N}(t)}\right\} \widetilde{\mathcal{D}}_{N,l}(t).
\end{align}
\item[(4)] For $| \alpha | + | \beta | =N$, $| \beta | \geq 1$ and
$| \alpha | \leq N-1$,  there holds
\begin{align}\label{soft E 2-4}
\left( \partial^\alpha_{\beta} ( E^\varepsilon \cdot \nabla_v \mathbf{P}^{\perp}f^\varepsilon ),  \widetilde{w}^2_{l}(\alpha,\beta)\partial^\alpha_\beta \mathbf{P}^{\perp}f^{\varepsilon} \right)
\lesssim \left\{ \delta_0 + \sqrt{\mathcal{E}_{N}(t)}\right\} \widetilde{\mathcal{D}}_{N,l}(t).
\end{align}
\end{itemize}
\end{lemma}

\begin{proof}
This lemma can be proved using the same argument as that of Lemma \ref{soft B 1}. The only difference is the estimate when all derivatives act on $\mathbf{P}^{\perp} f^\varepsilon$. In fact, by using integration by parts with respect to velocity, $L^\infty$--$L^2$--$L^2$ H\"{o}lder inequality  and \eqref{assumption1}, it is easy to obtain
\begin{align*}
\left( E^\varepsilon \cdot \nabla_v \partial^\alpha_{\beta}\mathbf{P}^{\perp}f^\varepsilon ,  \widetilde{w}^2_{l}(\alpha,\beta)\partial^\alpha_\beta \mathbf{P}^{\perp}f^{\varepsilon} \right)
\lesssim  \delta_0  \widetilde{\mathcal{D}}_{N,l}(t).
\end{align*}
For brevity, we omit the details.
\end{proof}

\begin{lemma}\label{hard E 2}
Let $\gamma+2s\geq 0$, $0 < s < 1$ and $0 < \varepsilon \leq 1$.
\begin{itemize}
\setlength{\leftskip}{-6mm}
\item[(1)] For $1 \leq | \alpha | \leq N$, there holds
\begin{align}\label{hard E 2-1}
\left( \partial^\alpha ( E^\varepsilon \cdot \nabla_v f^\varepsilon ),  \partial^\alpha f^{\varepsilon} \right)
\lesssim \left\{ \delta_0 + \sqrt{\mathcal{E}_{N}(t)}\right\} \widetilde{\mathcal{D}}_{N,l}(t).
\end{align}
\item[(2)] For $| \alpha | + |\beta| \leq N-1$, there holds
\begin{align}\label{hard E 2-2}
\left( \partial^\alpha_{\beta} ( E^\varepsilon \cdot \nabla_v \mathbf{P}^{\perp}f^\varepsilon ),
\widetilde{w}^2_{l}(\alpha,\beta)\partial^\alpha_\beta \mathbf{P}^{\perp}f^{\varepsilon} \right)
\lesssim \left\{ \delta_0 + \sqrt{\mathcal{E}_{N}(t)}\right\} \widetilde{\mathcal{D}}_{N,l}(t).
\end{align}
\item[(3)] For $ | \alpha | = N$, there holds
\begin{align}\label{hard E 2-3}
\varepsilon \left( \partial^\alpha ( E^\varepsilon \cdot \nabla_v f^\varepsilon ),
\widetilde{w}^2_{l}(\alpha,0)\partial^\alpha f^{\varepsilon} \right)
\lesssim \left\{ \delta_0 + \sqrt{\mathcal{E}_{N}(t)}\right\} \widetilde{\mathcal{D}}_{N,l}(t).
\end{align}
\item[(4)] For $ | \alpha | + | \beta | = N$, $| \beta | \geq 1$ and
$| \alpha | \leq N-1$, there holds
\begin{align}\label{hard E 2-4}
\left( \partial^\alpha_{\beta} ( E^\varepsilon \cdot \nabla_v \mathbf{P}^{\perp}f^\varepsilon ),
\widetilde{w}^2_{l}(\alpha,\beta)\partial^\alpha_\beta \mathbf{P}^{\perp}f^{\varepsilon} \right)
\lesssim \left\{ \delta_0 + \sqrt{\mathcal{E}_{N}(t)}\right\} \widetilde{\mathcal{D}}_{N,l}(t).
\end{align}
\end{itemize}
\end{lemma}

\begin{proof}
This lemma can be proved similarly as that of Lemma \ref{hard B 1} and \ref{soft E 2}, and the details are omitted for brevity.
\end{proof}
\medskip

\section{The a Priori Estimates}\label{The a Priori Estimate}

In this section, we deduce the uniform a priori estimates  of the VMB system \eqref{rVMB} with respect to $\varepsilon\in (0,1]$
globally in time.

For this purpose, we define the following time-weighted energy norm $X(t)$ by
\begin{equation}\label{X define}
X(t):= \sup_{0 \leq \tau \leq t} \overline{\mathcal{E}}_{N,l}(\tau)
+ \sup_{0 \leq \tau \leq t} (1+\tau)^{\varrho+1} \sum_{1 \leq |\alpha| \leq N_0}\left\| \partial^\alpha (E^\varepsilon, B^\varepsilon)(\tau)\right\|^2.
\end{equation}
Here, all the involved parameters are fixed to satisfy \eqref{mainth1 assumption} for the case $\gamma > \max\{-3, -\frac{3}{2}-2s\}$.

The construction of the weighted instant energy functionals $\widetilde{\mathcal{E}}_{N,l}(t)$, $\overline{\mathcal{E}}_{N,l}(t)$ and the  dissipation rate functionals $\widetilde{\mathcal{D}}_{N,l}(t)$, $\overline{\mathcal{D}}_{N,l}(t)$ will be given in the following. Suppose that the VMB system  \eqref{rVMB} admits a smooth solution $\left(f^\varepsilon, E^\varepsilon, B^\varepsilon \right)$ over $0 \leq t \leq T$ for $0 < T \leq \infty$, and  the solution
$\left(f^\varepsilon, E^\varepsilon, B^\varepsilon \right)$ also satisfies
\begin{align}\label{priori assumption}
\sup_{0 \leq t \leq T} X(t) \leq \delta^2_0,
\end{align}
where $\delta_0$ is a suitable small positive constant.

\subsection{Macroscopic Estimate}
\hspace*{\fill}

In this subsection, we consider the macroscopic estimate of the VMB system \eqref{rVMB} with respect to $\varepsilon\in (0,1]$. As in \cite{DLYZ2013}, by applying the macro-micro decomposition \eqref{f decomposition} introduced in \cite{Guo2003} and defining moment functions
$$
\Theta_{ij}(f^{\varepsilon}):=\int_{\mathbb{R}^3}(v_i v_j -1)\mu^{1/2}f^{\varepsilon} \d v, \qquad
\Lambda_{i}(f^{\varepsilon}):=\frac{1}{10}\int_{\mathbb{R}^3}(|v|^2 -5)v_i \mu^{1/2}f^{\varepsilon} \d v,
$$
one can derive from \eqref{rVMB} the following two fluid-type systems of equations
\begin{equation}\label{macro equation 1}
\left\{\begin{aligned}
& \partial_t a^{\varepsilon}_{ \pm}+ \frac{1}{\varepsilon} \nabla_x \cdot b^{\varepsilon}
+\frac{1}{\varepsilon} \nabla_x \cdot \langle v \mu^{1 / 2}, \mathbf{P}_{\pm}^{\perp} f^{\varepsilon} \rangle= \langle \mu^{1/2}, g^\varepsilon_{\pm} \rangle, \\
& \partial_t \big(b^{\varepsilon}_i+\langle v_i \mu^{1 / 2},\mathbf{P}_{\pm}^{\perp} f^{\varepsilon}\rangle \big)
+ \frac{1}{\varepsilon} \partial_i(a^{\varepsilon}_{ \pm}+2 c^{\varepsilon}) \mp \frac{1}{\varepsilon} E^\varepsilon_i \\
&\qquad\qquad
=-\frac{1}{\varepsilon} \langle v_i \mu^{1 / 2}, v \cdot \nabla_x \mathbf{P}_{\pm}^{\perp} f^\varepsilon \rangle
+\langle v_i \mu^{1 / 2}, g^\varepsilon_{ \pm}-\frac{1}{\varepsilon^2}L_{ \pm} f^\varepsilon \rangle, \\
& \partial_t \Big(c^\varepsilon+\frac{1}{6}\langle (|v|^2-3) \mu^{1 / 2}, \mathbf{P}_{\pm}^{\perp} f^\varepsilon \rangle \Big)
+\frac{1}{3\varepsilon} \nabla_x \cdot b^\varepsilon \\
&\qquad\qquad
=-\frac{1}{6\varepsilon}\langle (|v|^2-3) \mu^{1 / 2}, v \cdot \nabla_x \mathbf{P}_{\pm}^{\perp} f^\varepsilon \rangle
+\frac{1}{6}\langle (|v|^2-3) \mu^{1 / 2}, g^\varepsilon_{ \pm}- \frac{1}{\varepsilon^2} L_{ \pm} f^\varepsilon\rangle
\end{aligned}\right.
\end{equation}
and
\begin{equation}\label{macro equation 2}
\left\{\begin{aligned}
\!\!\!&\partial_t (\Theta_{i i}(\mathbf{P}_{\pm}^{\perp} f^{\varepsilon})+2 c^{\varepsilon} )+ \frac{2}{\varepsilon} \partial_i b^\varepsilon_i =\Theta_{i i}(g^\varepsilon_{ \pm}+h^\varepsilon_{ \pm}), \\
&\partial_t \Theta_{i j} ( \mathbf{P}_{\pm}^{\perp} f^{\varepsilon} )+\frac{1}{\varepsilon} \partial_i b^\varepsilon_j
+\frac{1}{\varepsilon} \partial_j b^\varepsilon_i+\frac{1}{\varepsilon} \nabla_x \cdot \langle v \mu^{1 / 2},
\mathbf{P}_{\pm}^{\perp} f^{\varepsilon} \rangle\\
&\qquad\qquad\qquad\qquad\qquad=\Theta_{i j}(g^\varepsilon_{ \pm}+h^\varepsilon_{ \pm} )+\langle \mu^{1/2}, g^\varepsilon_{\pm} \rangle, \;\;\;\;
i \neq j, \\
&\partial_t \Lambda_i (\mathbf{P}_{\pm}^{\perp} f^{\varepsilon})+\frac{1}{\varepsilon} \partial_i c^{\varepsilon}
=\Lambda_i(g^\varepsilon_{ \pm}+h^\varepsilon_{ \pm}),
\end{aligned}\right.
\end{equation}
where
\begin{align}
\begin{split}\label{g,h define}
g^\varepsilon_{\pm}:=&\pm \frac{1}{2}E^{\varepsilon} \cdot vf_{\pm}^\varepsilon
\mp \Big(E^{\varepsilon}+ \frac{1}{\varepsilon}v \times B^\varepsilon\Big) \cdot \nabla_v f_{\pm}^\varepsilon
+ \frac{1}{\varepsilon} \Gamma_{\pm}(f^\varepsilon, f^\varepsilon), \\
h^\varepsilon_{\pm}:=&-\frac{1}{\varepsilon}v \cdot \nabla_x \mathbf{P}_{\pm}^{\perp}f^{\varepsilon}-\frac{1}{\varepsilon^2}L_{\pm} f^{\varepsilon}.
\end{split}
\end{align}
Furthermore, it is easy to calculate through integration by parts that
\begin{align}
\begin{split}\label{g estimate1}
\langle \mu^{1/2}, g^\varepsilon_{\pm} \rangle=\;&0, \\
\langle v \mu^{1 / 2}, g^\varepsilon_{ \pm} \rangle=\;&\pm E^\varepsilon a^\varepsilon_{\pm} \pm \frac{1}{\varepsilon} b^\varepsilon \times B^\varepsilon
\pm \frac{1}{\varepsilon} \langle v\mu^{1/2}, \mathbf{P}_{\pm}^{\perp}f^\varepsilon \rangle \times B^\varepsilon \\
&+\frac{1}{\varepsilon} \langle v \mu^{1/2}, \Gamma_{\pm}(f^\varepsilon,f^\varepsilon) \rangle, \\
\frac{1}{6} \langle (|v|^2-3)\mu^{1/2}, g^\varepsilon_{\pm} \rangle=\;& \pm \frac{1}{3} b^\varepsilon \cdot E^\varepsilon \pm \frac{1}{3} \langle v \mu^{1/2}, \mathbf{P}_{\pm}^{\perp}f^\varepsilon \rangle \cdot E^\varepsilon \\
&+\frac{1}{6\varepsilon} \langle (|v|^2-3)\mu^{1/2}, \Gamma_{\pm}(f^\varepsilon,f^\varepsilon) \rangle.
\end{split}
\end{align}
Set
\begin{align}\label{G define}
G^\varepsilon:=\langle v\mu^{1/2}, \mathbf{P}^{\perp}f^{\varepsilon} \cdot q_1\rangle.
\end{align}
By taking the mean value of every two equations with $\pm$ sign for \eqref{macro equation 1}--\eqref{macro equation 2}, one has from \eqref{g estimate1} that
\begin{equation}\label{macro equation 3}
\!\!\!\!\left\{\begin{aligned}
& \!\partial_t \Big(\frac{a^{\varepsilon}_{+}+a^{\varepsilon}_{-}}{2}\Big) + \frac{1}{\varepsilon} \nabla_x \cdot b^{\varepsilon}=0, \\
& \!\partial_t b^{\varepsilon}_i
+ \frac{1}{\varepsilon} \partial_i \Big(\frac{a^{\varepsilon}_{+}+a^{\varepsilon}_{-}}{2}+2 c^{\varepsilon}\Big)
\!+\!\frac{1}{2\varepsilon} \sum_{j=1}^{3} \partial_j \Theta_{ij} ( \mathbf{P}^{\perp} f^\varepsilon \cdot q_2 )
= \frac{a^\varepsilon_{+}-a^\varepsilon_{-}}{2}E^\varepsilon_i +\frac{1}{2\varepsilon} (G^\varepsilon \times B^\varepsilon)_i, \\
& \!\partial_t c^\varepsilon +\frac{1}{3\varepsilon} \nabla_x \cdot b^\varepsilon
+\frac{5}{6\varepsilon} \sum_{i=1}^{3} \partial_i \Lambda_i ( \mathbf{P}^{\perp} f^\varepsilon \cdot q_2 )
=\frac{1}{6} G^\varepsilon \cdot E^\varepsilon
\end{aligned}\right.
\end{equation}
and
\begin{equation}\label{macro equation 4}
\left\{\begin{aligned}
&\partial_t \Big(\frac{1}{2}\Theta_{i j}(\mathbf{P}^{\perp} f^{\varepsilon} \cdot q_2)+2 c^{\varepsilon}\delta_{ij} \Big)+ \frac{1}{\varepsilon} \partial_i b^\varepsilon_j + \frac{1}{\varepsilon} \partial_j b^\varepsilon_i
=\frac{1}{2}\Theta_{i j}(g^\varepsilon_{+}+g^\varepsilon_{-}+h^\varepsilon_{+}+h^\varepsilon_{-}), \\
&\frac{1}{2}\partial_t \Lambda_i (\mathbf{P}^{\perp} f^{\varepsilon} \cdot q_2)+\frac{1}{\varepsilon} \partial_i c^{\varepsilon}
=\frac{1}{2}\Lambda_i(g^\varepsilon_{+}+g^\varepsilon_{-}+h^\varepsilon_{+}+h^\varepsilon_{-}),
\end{aligned}\right.
\end{equation}
where $q_2=[1,1]$, $1 \leq i, j \leq 3$, $\delta_{ij}$ denotes the Kronecker delta.
Moreover,  by taking the difference of the first two equations with $\pm$ sign for \eqref{macro equation 1}, we have from \eqref{g estimate1} that
\begin{equation}\label{macro equation 5}
\left\{\begin{aligned}
&\! \partial_t (a^{\varepsilon}_{+}-a^{\varepsilon}_{-}) + \frac{1}{\varepsilon} \nabla_x \cdot G^\varepsilon=0, \\
&\! \partial_t G^\varepsilon + \frac{1}{\varepsilon} \nabla_x(a^{\varepsilon}_{+}-a^{\varepsilon}_{-})-\frac{2}{\varepsilon}E^{\varepsilon}
+\frac{1}{\varepsilon} \nabla_x \cdot \Theta ( \mathbf{P}^{\perp} f^\varepsilon \cdot q_1 ) \\
&\qquad\qquad= E(a^\varepsilon_{+}+a_{-}^\varepsilon) + \frac{2}{\varepsilon} (b^\varepsilon \times B^\varepsilon) + \big\langle [v,-v] \mu^{1 / 2}, \frac{1}{\varepsilon}\Gamma(f^\varepsilon, f^\varepsilon)- \frac{1}{\varepsilon^2}Lf^\varepsilon \big\rangle.
\end{aligned}\right.
\end{equation}

Now we define the macro dissipation $\mathcal{D}_{mac}^{N}(t)$ by
\begin{align*}
 \mathcal{D}_{mac}^{N}(t) := \sum_{1 \leq |\alpha| \leq N} \left \| \partial^{\alpha} (a^\varepsilon_{\pm}, b^\varepsilon, c^\varepsilon) \right \|^2
+ \left\| a^\varepsilon_{+}-a^\varepsilon_{-}\right\|^2
 +\sum_{|\alpha| \leq N-1} \left\|\partial^{\alpha}E^{\varepsilon}\right\|^2
+\sum_{1 \leq |\alpha| \leq N-1} \left\|\partial^{\alpha}B^{\varepsilon}\right\|^2.
\end{align*}

\begin{lemma}\label{macroscopic estimate}
Let $\left(f^\varepsilon, E^\varepsilon, B^\varepsilon \right)$ be the solution to the VMB system \eqref{rVMB}.
Then there exists an interactive functional $\mathcal{E}^N_{int}(t)$  satisfying
\begin{align}\label{macroscopic estimate11}
\left|\mathcal{E}^N_{int}(t)\right| \lesssim\;&\sum_{|\alpha| \leq N} \left\| \partial^\alpha (f^{\varepsilon}, E^\varepsilon, B^\varepsilon) \right\|^2,
\end{align}
such that for any $t \geq 0$, there holds
\begin{align}\label{macro estimate}
 \frac{\d}{\d t}\mathcal{E}^N_{int}(t)+\mathcal{D}_{mac}^{N}(t)
\lesssim  \frac{1}{\varepsilon^2} \sum_{|\alpha| \leq N}\left\| \partial^\alpha \mathbf{P}^{\perp}f^{\varepsilon}\right\|^2_{D}
+\mathcal{E}_{N}(t) \mathcal{D}_{N}(t).
\end{align}
\end{lemma}

\begin{proof}
Based on the previous research \cite{DS2011, JL2023ARXIV}, it is a standard process to arrive at \eqref{macro estimate} with \eqref{macroscopic estimate11} being satisfied. For details, see \cite[Appendix A.3]{JL2023ARXIV}. We omit further discussion for brevity.
\end{proof}

\subsection{Energy Estimate Without Weight}
\hspace*{\fill}

This subsection provides the energy estimate without weight function for the potential case $\gamma > \max\left\{-3, -\frac{3}{2}-2s\right\}$.
\begin{lemma}\label{energy estimate without weight}
There holds
\begin{align}\label{estimate without weight}
\frac{1}{2}\frac{\d}{\d t}\sum_{|\alpha| \leq N}\left\| \partial^\alpha (f^\varepsilon, E^\varepsilon, B^\varepsilon)\right\|^2
+ \frac{\lambda}{\varepsilon^2}\sum_{|\alpha| \leq N}\left\|\partial^\alpha \mathbf{P}^{\perp}f^\varepsilon\right\|^2_{D}
\lesssim  \left\{ \delta_0 + \sqrt{\widetilde{\mathcal{E}}_{N,l}(t)} \right\} \widetilde{\mathcal{D}}_{N,l}(t).
\end{align}
\end{lemma}

\begin{proof}
Applying $\partial^\alpha$ with $| \alpha | \leq N $ to the first equation of \eqref{rVMB} and taking the inner product with $  \partial^\alpha f^\varepsilon$ over $\mathbb{R}_x^3$ $\times$  $\mathbb{R}_v^3$, we obtain
\begin{align}\label{without weight}
&\frac{1}{2}\frac{\d}{\d t}\left\| \partial^\alpha (f^\varepsilon, E^\varepsilon, B^\varepsilon) \right\|^2
+ \left( \frac{1}{\varepsilon^2} L \partial^\alpha f^\varepsilon, \partial^\alpha f^\varepsilon \right) \nonumber \\
=\;&-\chi_{|\alpha| \geq 1}\left( \partial^\alpha  (q_0 E^\varepsilon \cdot \nabla_v f^\varepsilon ), \partial^\alpha f^\varepsilon \right)
-\chi_{|\alpha| \geq 1}\left( \frac{1}{\varepsilon }\partial^\alpha (q_0  (v \times B^\varepsilon) \cdot \nabla_v f^\varepsilon ), \partial^\alpha f^\varepsilon \right)  \\
&+ \left( \partial^\alpha \big( \frac{q_0}{2}E^\varepsilon \cdot v  f^\varepsilon \big), \partial^\alpha f^\varepsilon \right)
+ \left( \frac{1}{\varepsilon }\partial^\alpha \Gamma(f^\varepsilon, f^\varepsilon), \partial^\alpha f^\varepsilon \right), \nonumber
\end{align}
where we have used the facts
$$
\frac{1}{\varepsilon} \left( v \cdot \nabla_x \partial^\alpha f^\varepsilon, \partial^\alpha f^\varepsilon \right)=0,\;\;
\left( q_0   E^\varepsilon \cdot \nabla_v f^\varepsilon, f^\varepsilon \right)
=\left( \frac{1}{\varepsilon} q_0  (v \times B^\varepsilon) \cdot \nabla_v f^\varepsilon,  f^\varepsilon \right)=0
$$
and
$$
-\left( \frac{1}{\varepsilon} \partial^\alpha E^\varepsilon \cdot v \mu^{1/2} q_1, \partial^\alpha f^\varepsilon \right)
=\frac{1}{2} \frac{\d }{\d t}\left\| \partial^\alpha (E^\varepsilon, B^\varepsilon) \right\|^2.
$$

The second term on the left-hand side of \eqref{without weight} follows from \eqref{L coercive1} that
\begin{align*}
\left( \frac{1}{\varepsilon^2} L \partial^\alpha f^\varepsilon, \partial^\alpha f^\varepsilon \right)
\geq \frac{\lambda}{\varepsilon^2} \left\| \partial^\alpha \mathbf{P}^{\perp}f^\varepsilon \right\|^2_{D}.
\end{align*}

Now, let's estimate the terms on the right-hand side of \eqref{without weight} individually.
By \eqref{soft B 1-1}, \eqref{hard B 1-1}, \eqref{soft E 1-1}, \eqref{hard E 1-1}, \eqref{soft E 2-1} and \eqref{hard E 2-1}, we can obtain that the first three terms on the right-hand side of \eqref{without weight} are bounded by
$\left\{ \delta_0 + \sqrt{\mathcal{E}_{N}(t)} \right\} \widetilde{\mathcal{D}}_{N,l}(t)$.
By virtue of \eqref{soft gamma3} and \eqref{hard gamma3}, we can control the last term on the right-hand side of \eqref{without weight} by
$\sqrt{\widetilde{\mathcal{E}}_{N,l}(t)} \mathcal{D}_{N}(t)$.

Collecting all the above estimates and taking summation over $|\alpha| \leq N$, we can prove that \eqref{estimate without weight} holds true. This completes the proof of Lemma \ref{energy estimate without weight}.
\end{proof}

\subsection{Weighted Energy Estimate}
\hspace*{\fill}

In this subsection, we give the energy estimates with the weight function $\widetilde{w}_l(\alpha, \beta)$ for the case of $\gamma > \max\left\{-3, -\frac{3}{2}-2s\right\}$. The following lemma is the energy estimate about the pure spatial derivatives $|\alpha| \leq N-1$.
\begin{lemma}\label{weighted 1}
There holds
\begin{align}\label{weighted estimate1}
&\frac{1}{2}\frac{\d}{\d t} \sum_{|\alpha|\leq N-1} \left\|  \widetilde{w}_l(\alpha,0) \partial^\alpha \mathbf{P}^{\perp} f^\varepsilon \right\|^2
+ \frac{\lambda}{\varepsilon^2}\sum_{|\alpha| \leq N-1} \left\| \widetilde{w}_l(\alpha,0)\partial^\alpha \mathbf{P}^{\perp} f^\varepsilon\right\|_{D}^2 \nonumber\\
&+ \frac{q\vartheta}{(1+t)^{1+\vartheta}}  \sum_{|\alpha| \leq N-1} \left\| \langle v \rangle ^{\frac{1}{2}} \widetilde{w}_l(\alpha,0)\partial^\alpha \mathbf{P}^{\perp} f^\varepsilon\right\|^2 \nonumber\\
\lesssim\;&\frac{1}{\varepsilon^2} \sum_{|\alpha| \leq N} \left\| \partial^\alpha \mathbf{P}^{\perp} f^\varepsilon\right\|_{D}^2
+ \sum_{|\alpha| \leq N-1} \left\| \partial^\alpha E^\varepsilon \right\|^2
+ \sum_{1 \leq |\alpha| \leq N} \left\| \partial^\alpha (a^\varepsilon_{\pm}, b^\varepsilon, c^\varepsilon) \right\|^2
\nonumber \\
&+ \left\{ \delta_0 + \mathcal{E}_{N}(t)+ \sqrt{\widetilde{\mathcal{E}}_{N,l}(t)} \right\} \widetilde{\mathcal{D}}_{N,l}(t).
\end{align}
\end{lemma}
\begin{proof}
We apply the micro-projection operator $\mathbf{P}^\perp$ into the first equation of \eqref{rVMB} and use the macro-micro decomposition \eqref{f decomposition} to obtain
\begin{align}\label{rrVMB}
&\partial_t  \mathbf{P}^{\perp} f^\varepsilon + \frac{1}{\varepsilon} v \cdot \nabla_x \mathbf{P}^{\perp} f^\varepsilon
+\frac{1}{\varepsilon^2} L\mathbf{P}^{\perp} f^\varepsilon -\frac{1}{\varepsilon} E^\varepsilon \cdot v\mu^{1/2}q_1 \nonumber\\
=\;&-\frac{1}{\varepsilon } q_0 (\varepsilon E^\varepsilon +v \times B^\varepsilon) \cdot \nabla_v \mathbf{P}^{\perp} f^\varepsilon
+\frac{q_0}{2} E^\varepsilon \cdot v \mathbf{P}^{\perp} f^\varepsilon
+\frac{1}{\varepsilon} \Gamma(f^\varepsilon, f^\varepsilon) \nonumber \\
&+ \mathbf{P} \left( \frac{1}{\varepsilon} \Big[v \cdot \nabla_x
+ q_0 (\varepsilon E^\varepsilon + v \times B^\varepsilon) \cdot \nabla_v \Big]\mathbf{P}^{\perp} f^\varepsilon
-\frac{q_0}{2} E^\varepsilon \cdot v \mathbf{P}^{\perp} f^\varepsilon \right) \nonumber\\
&- \mathbf{P}^{\perp} \left( \frac{1}{\varepsilon} \Big[ v \cdot \nabla_x
+ q_0 (\varepsilon E^\varepsilon + v \times B^\varepsilon) \cdot \nabla_v \Big] \mathbf{P} f^\varepsilon
-\frac{q_0}{2} E^\varepsilon \cdot v \mathbf{P} f^\varepsilon \right).
\end{align}
By applying $\partial^\alpha$ with $|\alpha| \leq N-1$ to \eqref{rrVMB} and taking the inner product with $ \widetilde{w}^2_l(\alpha,0)\partial^\alpha \mathbf{P}^{\perp} f^\varepsilon$ over $\mathbb{R}_x^3 \times \mathbb{R}_v^3$, one has
\begin{align}\label{with weight 1}
&\frac{1}{2} \frac{\d}{\d t}  \left\|  \widetilde{w}_l(\alpha,0) \partial^\alpha \mathbf{P}^{\perp} f^\varepsilon \right\|^2
+ \frac{q \vartheta}{(1+t)^{1+\vartheta}} \left\| \langle v \rangle ^{\frac{1}{2}} \widetilde{w}_l(\alpha,0) \partial^\alpha \mathbf{P}^{\perp} f^\varepsilon \right\|^2 \nonumber \\
&+\left( \frac{1}{\varepsilon} v \cdot \nabla_x \partial^\alpha \mathbf{P}^{\perp}f^\varepsilon, \widetilde{w}^2_l(\alpha,0)\partial^\alpha\mathbf{P}^{\perp}f^\varepsilon\right)
+ \left( \frac{1}{\varepsilon^2} L\partial^\alpha \mathbf{P}^{\perp} f^\varepsilon,  \widetilde{w}^2_l(\alpha,0) \partial^\alpha \mathbf{P}^{\perp} f^\varepsilon \right) \nonumber\\
=\;& \left( \frac{1}{\varepsilon} \partial^\alpha E^\varepsilon \cdot v\mu^{1/2}q_1, \widetilde{w}^2_l(\alpha,0) \partial^\alpha \mathbf{P}^{\perp} f^\varepsilon \right)
+ \left( \frac{1}{\varepsilon} \partial^\alpha \Gamma(f^\varepsilon, f^\varepsilon), \widetilde{w}^2_l(\alpha,0) \partial^\alpha \mathbf{P}^{\perp} f^\varepsilon \right) \nonumber \\
&+ \left( \partial^\alpha \Big( -\frac{1}{\varepsilon }q_0 (\varepsilon E^\varepsilon +v \times B^\varepsilon) \cdot \nabla_v \mathbf{P}^{\perp} f^\varepsilon
+ \frac{q_0}{2} E^\varepsilon \cdot v \mathbf{P}^{\perp} f^\varepsilon \Big),  \widetilde{w}^2_l(\alpha,0) \partial^\alpha \mathbf{P}^{\perp} f^\varepsilon \right) \nonumber \\
&+ \left( \partial^\alpha \mathbf{P} \left( \frac{1}{\varepsilon} \Big[v \cdot \nabla_x
+ q_0 (\varepsilon E^\varepsilon + v \times B^\varepsilon) \cdot \nabla_v \Big]\mathbf{P}^{\perp} f^\varepsilon \right), \widetilde{w}^2_l(\alpha,0) \partial^\alpha \mathbf{P}^{\perp} f^\varepsilon \right) \nonumber \\
&- \left( \partial^\alpha \mathbf{P} \left( \frac{q_0}{2} E^\varepsilon \cdot v \mathbf{P}^{\perp} f^\varepsilon \right)
- \partial^\alpha \mathbf{P}^{\perp} \left( \frac{q_0}{2} E^\varepsilon \cdot v \mathbf{P} f^\varepsilon \right) ,
 \widetilde{w}^2_l(\alpha,0) \partial^\alpha \mathbf{P}^{\perp} f^\varepsilon \right) \nonumber\\
&- \left( \partial^\alpha \mathbf{P}^{\perp} \left( \frac{1}{\varepsilon} \Big[v \cdot \nabla_x
+ q_0 (\varepsilon E^\varepsilon + v \times B^\varepsilon) \cdot \nabla_v \Big]\mathbf{P} f^\varepsilon \right), \widetilde{w}^2_l(\alpha,0) \partial^\alpha \mathbf{P}^{\perp} f^\varepsilon \right).
\end{align}

The third term on the left-hand side of \eqref{with weight 1} vanishes using integration by parts. And the last term on the left-hand side of \eqref{with weight 1} follows from \eqref{L coercive2} that
\begin{align*}
\left( \frac{1}{\varepsilon^2} L\partial^\alpha \mathbf{P}^{\perp} f^\varepsilon,  \widetilde{w}^2_l(\alpha,0) \partial^\alpha \mathbf{P}^{\perp} f^\varepsilon \right)
\geq \frac{\lambda}{\varepsilon^2} \left\| \widetilde{w}_l(\alpha,0) \partial^\alpha \mathbf{P}^{\perp} f^\varepsilon \right\|_{D}^2
-\frac{C}{\varepsilon^2} \left\|  \partial^\alpha \mathbf{P}^{\perp} f^\varepsilon \right\|^2_{D}.
\end{align*}

Next, we will compute each term for the right-hand side of \eqref{with weight 1}. Firstly, applying the H\"{o}lder inequality, the first term on the
right-hand side of \eqref{with weight 1} is dominated by
$$
\left\| \partial^\alpha E^\varepsilon \right\|^2 + \frac{1}{\varepsilon^2}\left\| \partial^\alpha \mathbf{P}^{\perp} f^\varepsilon \right\|_{D}^2.
$$
Due to \eqref{soft gamma4} and \eqref{hard gamma4}, we can control the second term on the right-hand side of \eqref{with weight 1} by
$\sqrt{\widetilde{\mathcal{E}}_{N,l}(t)} \widetilde{\mathcal{D}}_{N,l}(t)$.
Moreover, it follows from \eqref{soft B 1-2}, \eqref{hard B 1-2}, \eqref{soft E 1-2}, \eqref{hard E 1-2}, \eqref{soft E 2-2} and \eqref{hard E 2-2} that the third term on the right-hand side of \eqref{with weight 1} is bounded by
$\left\{ \delta_0 + \sqrt{\mathcal{E}_{N}(t)} \right\} \widetilde{\mathcal{D}}_{N,l}(t)$.
Finally, from the definition of $\mathbf{P}f^\varepsilon$ in \eqref{Pf define}, the H\"{o}lder inequality and \eqref{Sobolev-ineq}, we can derive that the last three terms on the right-hand side of \eqref{with weight 1} can be controlled by
$$
\frac{\eta}{\varepsilon^2} \left\| \widetilde{w}_l(\alpha,0)\partial^\alpha\mathbf{P}^{\perp}f^\varepsilon \right\|^2_{D}
+ \left\| \partial^\alpha \nabla_x \mathbf{P}^{\perp}f^\varepsilon \right\|_{D}^2
+ \left\| \partial^\alpha \nabla_x (a_{\pm}^\varepsilon, b^\varepsilon, c^\varepsilon) \right\|^2
+ \mathcal{E}_{N}(t)\mathcal{D}_{N}(t).
$$

Therefore, by plugging the estimates mentioned above into \eqref{with weight 1} and taking summation over $|\alpha| \leq N-1$, \eqref{weighted estimate1} follows. This completes the proof of Lemma \ref{weighted 1}.
\end{proof}

Next, we provide the weighted energy estimate with the mixed derivatives $|\alpha|+|\beta| \leq N-1$ and $|\beta| \geq 1$.
\begin{lemma}\label{weighted 2}
There holds
\begin{align}\label{weighted estimate2}
& \sum_{m=1}^{N-1}C_m \sum_{\substack{{|\alpha|+|\beta| \leq N-1}\\{|\beta| =m}}} \bigg\{\frac{1}{2} \frac{\d}{\d t} \left\| \widetilde{w}_l(\alpha,\beta) \partial^\alpha_\beta \mathbf{P}^{\perp}f^\varepsilon \right\|^2
+ \frac{\lambda}{\varepsilon^2} \left\| \widetilde{w}_l(\alpha,\beta)\partial^\alpha_\beta \mathbf{P}^{\perp} f^\varepsilon\right\|_{D}^2 \nonumber\\
&\qquad \qquad \qquad \qquad \quad\; + \frac{q\vartheta}{(1+t)^{1+\vartheta}}  \left\| \langle v \rangle ^{\frac{1}{2}} \widetilde{w}_l(\alpha, \beta)\partial^\alpha_\beta \mathbf{P}^{\perp} f^\varepsilon\right\|^2  \bigg\} \nonumber\\
\lesssim\;&\frac{1}{\varepsilon^2} \sum_{|\alpha| \leq N-1}  \left\| \widetilde{w}_{l}(\alpha, 0) \partial^\alpha \mathbf{P}^{\perp} f^\varepsilon\right\|_{D}^2
+\frac{1}{\varepsilon^2} \sum_{|\alpha| \leq N-1} \left\| \partial^\alpha \mathbf{P}^{\perp} f^\varepsilon\right\|_{D}^2
+ \sum_{|\alpha| \leq N-2} \left\| \partial^\alpha E^\varepsilon \right\|^2 \nonumber\\
&+ \sum_{1 \leq |\alpha| \leq N-1} \left\| \partial^\alpha (a^\varepsilon_{\pm}, b^\varepsilon, c^\varepsilon) \right\|^2
+ \left\{ \delta_0 + \mathcal{E}_{N}(t)+ \sqrt{\widetilde{\mathcal{E}}_{N,l}(t)} \right\} \widetilde{\mathcal{D}}_{N,l}(t).
\end{align}
Here, $C_m$ is a fixed constant satisfying $C_m \gg C_{m+1}$.
\end{lemma}

\begin{proof}
Applying $\partial^\alpha_\beta$ with $|\alpha|+\beta| \leq N-1$ and $|\beta|\geq 1$ to \eqref{rrVMB} and taking the inner product with $ \widetilde{w}^2_l(\alpha, \beta)\partial^\alpha_\beta \mathbf{P}^{\perp} f^\varepsilon$ over $\mathbb{R}_x^3 \times \mathbb{R}_v^3$, we have
\begin{align}\label{with weight 2}
&\frac{1}{2} \frac{\d}{\d t}  \left\|  \widetilde{w}_l(\alpha, \beta) \partial^\alpha_\beta \mathbf{P}^{\perp} f^\varepsilon \right\|^2
+\frac{q \vartheta}{(1+t)^{1+\vartheta}} \left\| \langle v \rangle ^{\frac{1}{2}} \widetilde{w}_l(\alpha, \beta) \partial^\alpha_\beta \mathbf{P}^{\perp} f^\varepsilon \right\|^2 \nonumber \\
&+ \left( \frac{1}{\varepsilon^2} \partial^\alpha_\beta L \mathbf{P}^{\perp} f^\varepsilon,  \widetilde{w}^2_l(\alpha,\beta) \partial^\alpha_\beta \mathbf{P}^{\perp} f^\varepsilon \right) \nonumber\\
=\;& -\left( \frac{1}{\varepsilon} \partial^\alpha_\beta (v \cdot \nabla_x \mathbf{P}^{\perp} f^\varepsilon), \widetilde{w}^2_l(\alpha,\beta) \partial^\alpha_\beta \mathbf{P}^{\perp} f^\varepsilon \right) \nonumber \\
&+\left( \frac{1}{\varepsilon} \partial^\alpha E^\varepsilon \cdot \partial_\beta(v\mu^{1/2})q_1, \widetilde{w}^2_l(\alpha,\beta) \partial^\alpha_\beta \mathbf{P}^{\perp} f^\varepsilon \right)
+ \left( \frac{1}{\varepsilon} \partial^\alpha_\beta \Gamma(f^\varepsilon, f^\varepsilon), \widetilde{w}^2_l(\alpha,\beta) \partial^\alpha_\beta \mathbf{P}^{\perp} f^\varepsilon \right) \nonumber \\
&+ \left( \partial^\alpha_\beta \Big( -\frac{1}{\varepsilon }q_0 (\varepsilon E^\varepsilon +v \times B^\varepsilon) \cdot \nabla_v \mathbf{P}^{\perp} f^\varepsilon
+ \frac{q_0}{2} E^\varepsilon \cdot v \mathbf{P}^{\perp} f^\varepsilon \Big),  \widetilde{w}^2_l(\alpha,\beta) \partial^\alpha_\beta \mathbf{P}^{\perp} f^\varepsilon \right) \nonumber \\
&+ \left( \partial^\alpha_\beta \mathbf{P} \left( \frac{1}{\varepsilon} \Big[v \cdot \nabla_x
+ q_0 (\varepsilon E^\varepsilon + v \times B^\varepsilon) \cdot \nabla_v \Big]\mathbf{P}^{\perp} f^\varepsilon \right), \widetilde{w}^2_l(\alpha,\beta) \partial^\alpha_\beta \mathbf{P}^{\perp} f^\varepsilon \right) \nonumber \\
&- \left( \partial^\alpha_\beta \mathbf{P} \left( \frac{q_0}{2} E^\varepsilon \cdot v \mathbf{P}^{\perp} f^\varepsilon \right)
- \partial^\alpha_\beta \mathbf{P}^{\perp} \left( \frac{q_0}{2} E^\varepsilon \cdot v \mathbf{P} f^\varepsilon \right) ,
 \widetilde{w}^2_l(\alpha,\beta) \partial^\alpha_\beta \mathbf{P}^{\perp} f^\varepsilon \right) \nonumber\\
&- \left( \partial^\alpha_\beta \mathbf{P}^{\perp} \left( \frac{1}{\varepsilon} \Big[v \cdot \nabla_x
+ q_0 (\varepsilon E^\varepsilon + v \times B^\varepsilon) \cdot \nabla_v \Big]\mathbf{P} f^\varepsilon \right), \widetilde{w}^2_l(\alpha,\beta) \partial^\alpha_\beta \mathbf{P}^{\perp} f^\varepsilon \right).
\end{align}

The last term on the left-hand side of \eqref{with weight 2} follows from \eqref{L coercive3} that
\begin{align*}
&\left( \frac{1}{\varepsilon^2} \partial^\alpha_\beta  L \mathbf{P}^{\perp} f^\varepsilon,  \widetilde{w}^2_l(\alpha,\beta) \partial^\alpha_\beta \mathbf{P}^{\perp} f^\varepsilon \right) \\
\geq \;& \frac{\lambda}{\varepsilon^2} \left\| \widetilde{w}_l(\alpha,\beta) \partial^\alpha_\beta \mathbf{P}^{\perp} f^\varepsilon \right\|_{D}^2
-\frac{C}{\varepsilon^2} \sum_{|\beta^\prime| < |\beta|} \left\|  \widetilde{w}_l(\alpha,\beta^\prime) \partial^\alpha_{\beta^\prime} \mathbf{P}^{\perp} f^\varepsilon \right\|^2_{D}
- \frac{C}{\varepsilon^2} \left\| \partial^\alpha \mathbf{P}^{\perp} f^\varepsilon \right\|^2_{D}.
\end{align*}

Let's now carefully treat the terms on the right-hand side of \eqref{with weight 2} one by one. For the first term on the right-hand side of \eqref{with weight 2}, recalling the definition of $\widetilde{w}_l(\alpha,\beta)$ in \eqref{weight function1}, we have
\begin{align}
&-\left( \frac{1}{\varepsilon} \partial^\alpha_\beta (v \cdot \nabla_x \mathbf{P}^{\perp} f^\varepsilon), \widetilde{w}^2_l(\alpha,\beta) \partial^\alpha_\beta \mathbf{P}^{\perp} f^\varepsilon \right) \nonumber \\
=\;&-C_{\beta}^{e_i} \left( \frac{1}{\varepsilon} \partial^{\alpha+e_i}_{\beta-e_i} \mathbf{P}^{\perp} f^\varepsilon, \widetilde{w}^2_l(\alpha,\beta) \partial^\alpha_\beta \mathbf{P}^{\perp} f^\varepsilon \right)\nonumber \\
\lesssim\;& \frac{1}{\varepsilon } \left\| \widetilde{w}_l(\alpha+e_i,\beta-e_i) \langle v \rangle^{\frac{\gamma+2s}{2}} \partial^{\alpha+e_i}_{\beta-e_i} \mathbf{P}^{\perp} f^\varepsilon \right\|
\left\| \widetilde{w}_l(\alpha,\beta) \langle v \rangle^{\frac{\gamma+2s}{2}} \partial^{\alpha}_{\beta} \mathbf{P}^{\perp} f^\varepsilon \right\| \nonumber \\
\lesssim\;& \frac{\eta}{\varepsilon^2} \left\| \widetilde{w}_l(\alpha,\beta)   \partial^{\alpha}_{\beta} \mathbf{P}^{\perp} f^\varepsilon \right\|_D^2
+\left\| \widetilde{w}_l(\alpha+e_i,\beta-e_i)  \partial^{\alpha+e_i}_{\beta-e_i} \mathbf{P}^{\perp} f^\varepsilon \right\|_D^2, \label{transport1}
\end{align}
where we used the fact $\widetilde{w}_l(\alpha,\beta) \leq \widetilde{w}_l(\alpha+e_i,\beta-e_i) \langle v \rangle^{\gamma+2s}$ for the case $\gamma > \max\{-3, -\frac{3}{2}-2s\}$.
The second term on the right-hand side of \eqref{with weight 2} can be bounded by
$$
\frac{\eta}{\varepsilon^2}\left\| \widetilde{w}_l(\alpha,\beta) \partial^\alpha_\beta \mathbf{P}^{\perp} f^\varepsilon \right\|_{D}^2 + \left\| \partial^\alpha E^\varepsilon \right\|^2
$$
for $|\alpha| \leq N-2$. Thanks to \eqref{soft gamma4} and \eqref{hard gamma4}, the third term on the right-hand side of \eqref{with weight 2} is bounded by
$\sqrt{\widetilde{\mathcal{E}}_{N,l}(t)} \widetilde{\mathcal{D}}_{N,l}(t)$.
For the fourth term on the right-hand side of \eqref{with weight 2}, from \eqref{soft B 1-2}, \eqref{hard B 1-2}, \eqref{soft E 1-2}, \eqref{hard E 1-2}, \eqref{soft E 2-2} and \eqref{hard E 2-2}, we can bound it by
$\left\{ \delta_0 + \sqrt{\mathcal{E}_{N}(t)} \right\} \widetilde{\mathcal{D}}_{N,l}(t)$.
Finally, from the definition of $\mathbf{P}f^\varepsilon$ in \eqref{Pf define}, the H\"{o}lder inequality and \eqref{Sobolev-ineq}, we can obtain that the last three terms on the right-hand side of \eqref{with weight 2} can be governed by
$$
\frac{\eta}{\varepsilon^2} \left\| \widetilde{w}_l(\alpha,\beta)\partial^\alpha_\beta \mathbf{P}^{\perp}f^\varepsilon \right\|^2_{D}
+ \left\| \partial^\alpha \nabla_x \mathbf{P}^{\perp}f^\varepsilon \right\|_{D}^2
+ \left\| \partial^\alpha \nabla_x (a_{\pm}^\varepsilon, b^\varepsilon, c^\varepsilon) \right\|^2
+ \mathcal{E}_{N}(t)\mathcal{D}_{N}(t).
$$

As a consequence, by plugging all the above estimates into \eqref{with weight 2}, taking the summation over $\{|\beta|=m, |\alpha|+|\beta| \leq N-1\}$ for each given $1 \leq m \leq N-1$, and then taking combination of those $N-1$ estimates with properly chosen constant $C_m > 0$ $(1 \leq m \leq N-1)$ and $\eta$ small enough, we can deduce \eqref{weighted estimate2}. We thus finish the proof of Lemma \ref{weighted 2}.
\end{proof}

What's more, we provide the weighted energy estimate with the pure spatial derivative $|\alpha|=N$.
\begin{lemma}\label{weighted 3}
There holds
\begin{align}\label{weighted estimate3}
&\frac{\varepsilon}{2} \frac{\d}{\d t} \sum_{|\alpha| = N} \left\| \widetilde{w}_l(\alpha,0) \partial^\alpha f^\varepsilon \right\|^2
+ \frac{\lambda}{\varepsilon}\sum_{|\alpha| = N} \left\| \widetilde{w}_l(\alpha,0)\partial^\alpha \mathbf{P}^{\perp} f^\varepsilon\right\|_{D}^2 \nonumber\\
&+ \frac{\varepsilon q\vartheta}{(1+t)^{1+\vartheta}}  \sum_{|\alpha| = N}\left\| \langle v \rangle ^{\frac{1}{2}} \widetilde{w}_l(\alpha, 0)\partial^\alpha f^\varepsilon\right\|^2
+\frac{\varepsilon(1+\varepsilon_0)}{4(1+t)} \sum_{|\alpha| = N} \left\| \widetilde{w}_l(\alpha, 0)\partial^\alpha  f^\varepsilon\right\|^2 \nonumber \\
\lesssim\;&\frac{1}{\varepsilon^2}\sum_{|\alpha| = N}\left\|\partial^\alpha \mathbf{P}^{\perp}f^\varepsilon\right\|^2_{D}
+\frac{1}{(1+t)^{1+\varepsilon_0}}\sum_{|\alpha| = N}\left\| \partial^\alpha E^\varepsilon \right\|^2
+\sum_{|\alpha| = N}\left\|\partial^\alpha(a^\varepsilon_{\pm},b^\varepsilon,c^\varepsilon)\right\|^2 \nonumber\\
&+ \left\{ \delta_0 + \sqrt{\widetilde{\mathcal{E}}_{N,l}(t)} \right\} \widetilde{\mathcal{D}}_{N,l}(t).
\end{align}
\end{lemma}

\begin{proof}
Applying $\partial^\alpha $ with $|\alpha|=N$ to the first equation of \eqref{rVMB} and taking the inner product with $\varepsilon \widetilde{w}^2_l(\alpha,0)\partial^\alpha f^\varepsilon$ over $\mathbb{R}_x^3 \times \mathbb{R}_v^3$, we have
\begin{align}\label{with weight 3}
&\frac{\varepsilon}{2} \frac{\d}{\d t}  \left\|  \widetilde{w}_l(\alpha, 0) \partial^\alpha f^\varepsilon \right\|^2
+ \frac{\varepsilon q\vartheta}{(1+t)^{1+\vartheta}}  \left\| \langle v \rangle ^{\frac{1}{2}} \widetilde{w}_l(\alpha, 0)\partial^\alpha f^\varepsilon\right\|^2 \nonumber \\
&+\frac{\varepsilon(1+\varepsilon_0)}{4(1+t)} \left\| \widetilde{w}_l(\alpha, 0)\partial^\alpha  f^\varepsilon\right\|^2
+ \left( \frac{1}{\varepsilon^2} L \partial^\alpha \mathbf{P}^{\perp} f^\varepsilon,  \varepsilon \widetilde{w}^2_l(\alpha, 0) \partial^\alpha \mathbf{P}^{\perp} f^\varepsilon \right) \nonumber\\
=\;& -\left( \frac{1}{\varepsilon^2} L \partial^\alpha \mathbf{P}^{\perp} f^\varepsilon,  \varepsilon \widetilde{w}^2_l(\alpha, 0) \partial^\alpha \mathbf{P} f^\varepsilon \right)
+\left( \frac{1}{\varepsilon} \partial^\alpha E^\varepsilon \cdot v\mu^{1/2}q_1, \varepsilon \widetilde{w}^2_l(\alpha, 0) \partial^\alpha f^\varepsilon \right) \nonumber\\
&-\left( \partial^\alpha (q_0 E^\varepsilon \cdot \nabla_v f^\varepsilon ), \varepsilon \widetilde{w}^2_l(\alpha, 0) \partial^\alpha f^\varepsilon \right)
-\left( \frac{1}{\varepsilon }\partial^\alpha (q_0  (v \times B^\varepsilon) \cdot \nabla_v f^\varepsilon ), \varepsilon \widetilde{w}^2_l(\alpha, 0) \partial^\alpha f^\varepsilon \right) \nonumber \\
&+ \left( \partial^\alpha \big(\frac{q_0}{2} E^\varepsilon \cdot v  f^\varepsilon \big), \varepsilon \widetilde{w}^2_l(\alpha, 0) \partial^\alpha f^\varepsilon \right)
+ \left( \frac{1}{\varepsilon }\partial^\alpha \Gamma(f^\varepsilon, f^\varepsilon), \varepsilon \widetilde{w}^2_l(\alpha, 0) \partial^\alpha f^\varepsilon \right),
\end{align}
where we have utilized the facts
$$
\left( \frac{1}{\varepsilon} v \cdot \nabla_x \partial^\alpha f^\varepsilon, \varepsilon \widetilde{w}^2_l(\alpha, 0) \partial^\alpha f^\varepsilon \right) =0,
$$
and
$$
\partial_t \widetilde{w}^2_{l}(\alpha, 0)= -\frac{2 q\vartheta}{(1+t)^{1+\vartheta}} \langle v \rangle\widetilde{w}^2_l(\alpha, 0)-\frac{1+\varepsilon_0}{2(1+t)}\widetilde{w}^2_l(\alpha, 0)
$$
for $|\alpha|=N$ by recalling $\widetilde{w}_{l}(\alpha, \beta)$ in \eqref{weight function1}.

The last term on the left-hand side of \eqref{with weight 3} follows from \eqref{L coercive2} that
\begin{align*}
\left( \frac{1}{\varepsilon^2} L\partial^\alpha \mathbf{P}^{\perp} f^\varepsilon,  \varepsilon \widetilde{w}^2_l(\alpha,0) \partial^\alpha \mathbf{P}^{\perp} f^\varepsilon \right)
\geq  \frac{\lambda}{\varepsilon} \left\| \widetilde{w}_l(\alpha,0) \partial^\alpha \mathbf{P}^{\perp} f^\varepsilon \right\|_{D}^2
-\frac{C}{\varepsilon} \left\|  \partial^\alpha \mathbf{P}^{\perp} f^\varepsilon \right\|^2_{D},
\end{align*}
where we used the fact $(1+t)^{-\frac{1+\varepsilon_0}{2}} \leq 1$.

Now, let's estimate the terms on the right-hand side of \eqref{with weight 3} individually. For the first term on the right-hand side of \eqref{with weight 3}, we can get from \eqref{Gamma zeta} that
$$
-\left( \frac{1}{\varepsilon^2} L\partial^\alpha \mathbf{P}^{\perp} f^\varepsilon,  \varepsilon \widetilde{w}^2_l(\alpha,0) \partial^\alpha \mathbf{P} f^\varepsilon \right)
\lesssim \frac{1}{\varepsilon^2} \left\| \partial^\alpha \mathbf{P}^{\perp} f^\varepsilon \right\|^2_{D}+ \left\| \partial^\alpha (a^\varepsilon_{\pm}, b^\varepsilon, c^\varepsilon) \right\|^2.
$$
The second term on the right-hand side of \eqref{with weight 3} can be bounded by
$$
\frac{1}{(1+t)^{1+\varepsilon_0}}\left\| \partial^\alpha E^\varepsilon \right\|^2 + \left\| \partial^\alpha \mathbf{P}^{\perp} f^\varepsilon \right\|_{D}^2 + \left\| \partial^\alpha (a^\varepsilon_{\pm}, b^\varepsilon, c^\varepsilon) \right\|^2.
$$
By using \eqref{soft B 1-3}, \eqref{hard B 1-3}, \eqref{soft E 1-3}, \eqref{hard E 1-3}, \eqref{soft E 2-3}, \eqref{hard E 2-3}, \eqref{soft gamma5} and \eqref{hard gamma5}, the upper bound of the last four terms on the right-hand side of \eqref{with weight 3} is
$$
\left\{ \delta_0 + \sqrt{\widetilde{\mathcal{E}}_{N,l}(t)} \right\} \widetilde{\mathcal{D}}_{N,l}(t).
$$

Collecting all the estimates above with $|\alpha|=N$, we can obtain the desired estimate \eqref{weighted estimate3}. This completes the proof of Lemma \ref{weighted 3}.
\end{proof}

Finally, we present the weighted energy estimate with the mixed derivatives $|\alpha|+|\beta| \leq N$ and $|\beta| \geq 1$.
\begin{lemma}\label{weighted 4}
There holds
\begin{align}\label{weighted estimate4}
& \sum_{m=1}^{N} \widetilde{C}_m \sum_{\substack{{|\alpha|+|\beta| = N}\\{|\beta| =m}}} \bigg\{\frac{1}{2} \frac{\d}{\d t} \left\| \widetilde{w}_l(\alpha,\beta) \partial^\alpha_\beta \mathbf{P}^{\perp}f^\varepsilon \right\|^2
+ \frac{\lambda}{\varepsilon^2} \left\| \widetilde{w}_l(\alpha,\beta)\partial^\alpha_\beta \mathbf{P}^{\perp} f^\varepsilon\right\|_{D}^2 \nonumber\\
&\qquad \; + \frac{q\vartheta}{(1+t)^{1+\vartheta}}  \left\| \langle v \rangle ^{\frac{1}{2}} \widetilde{w}_l(\alpha, \beta)\partial^\alpha_\beta \mathbf{P}^{\perp} f^\varepsilon\right\|^2
+\frac{1+\varepsilon_0}{4(1+t)} \left\| \widetilde{w}_l(\alpha, \beta)\partial^\alpha_\beta \mathbf{P}^{\perp} f^\varepsilon\right\|^2 \bigg\} \nonumber\\
\lesssim\;&\frac{1}{\varepsilon^2}  \sum_{|\alpha|+|\beta| \leq N-1} \left\| \widetilde{w}_{l}(\alpha, \beta) \partial^\alpha_\beta \mathbf{P}^{\perp} f^\varepsilon\right\|_{D}^2
+  \sum_{|\alpha| = N} \left\|  \widetilde{w}_{l}(\alpha, 0) \partial^\alpha \mathbf{P}^{\perp} f^\varepsilon\right\|_{D}^2 \nonumber\\
&+\frac{1}{\varepsilon^2} \sum_{|\alpha| \leq N} \!\left\| \partial^\alpha \mathbf{P}^{\perp} f^\varepsilon\right\|_{D}^2
+ \sum_{|\alpha| \leq N-1} \left\| \partial^\alpha E^\varepsilon \right\|^2
+ \sum_{1 \leq |\alpha| \leq N} \left\| \partial^\alpha (a^\varepsilon_{\pm}, b^\varepsilon, c^\varepsilon) \right\|^2 \nonumber\\
&+ \left\{ \delta_0 + \mathcal{E}_{N}(t)+ \sqrt{\widetilde{\mathcal{E}}_{N,l}(t)} \right\} \widetilde{\mathcal{D}}_{N,l}(t).
\end{align}
Here, $\widetilde{C}_m$ is a fixed constant satisfying $\widetilde{C}_m \gg \widetilde{C}_{m+1}$.
\end{lemma}
\begin{proof}
Let $1 \leq m \leq N$. By applying $\partial^\alpha_\beta$ with $|\beta|=m$ and $|\alpha|+|\beta| = N$ to \eqref{rrVMB} and taking the inner product with $\widetilde{w}^2_l(\alpha,\beta)\partial^\alpha_\beta \mathbf{P}^{\perp}f^\varepsilon$ over $\mathbb{R}_x^3 \times \mathbb{R}_v^3$, one has
\begin{align}\label{with weight 4}
&\frac{1}{2} \frac{\d}{\d t}  \left\|  \widetilde{w}_l(\alpha, \beta) \partial^\alpha_\beta \mathbf{P}^{\perp} f^\varepsilon \right\|^2
+\frac{q \vartheta}{(1+t)^{1+\vartheta}} \left\| \langle v \rangle ^{\frac{1}{2}} \widetilde{w}_l(\alpha, \beta) \partial^\alpha_\beta \mathbf{P}^{\perp} f^\varepsilon \right\|^2 \nonumber \\
&+\frac{1+\varepsilon_0}{4(1+t)}  \left\| \widetilde{w}_l(\alpha, \beta)\partial^\alpha_\beta \mathbf{P}^{\perp}  f^\varepsilon\right\|^2
+ \left( \frac{1}{\varepsilon^2} \partial^\alpha_\beta L \mathbf{P}^{\perp} f^\varepsilon,  \widetilde{w}^2_l(\alpha,\beta) \partial^\alpha_\beta \mathbf{P}^{\perp} f^\varepsilon \right) \nonumber\\
=\;& -\left( \frac{1}{\varepsilon} \partial^\alpha_\beta (v \cdot \nabla_x \mathbf{P}^{\perp} f^\varepsilon), \widetilde{w}^2_l(\alpha,\beta) \partial^\alpha_\beta \mathbf{P}^{\perp} f^\varepsilon \right) \nonumber \\
&+\left( \frac{1}{\varepsilon} \partial^\alpha E^\varepsilon \cdot \partial_\beta(v\mu^{1/2})q_1, \widetilde{w}^2_l(\alpha,\beta) \partial^\alpha_\beta \mathbf{P}^{\perp} f^\varepsilon \right)
+ \left( \frac{1}{\varepsilon} \partial^\alpha_\beta \Gamma(f^\varepsilon, f^\varepsilon), \widetilde{w}^2_l(\alpha,\beta) \partial^\alpha_\beta \mathbf{P}^{\perp} f^\varepsilon \right) \nonumber \\
&+ \left( \partial^\alpha_\beta \Big( -\frac{1}{\varepsilon } q_0 (\varepsilon E^\varepsilon +v \times B^\varepsilon) \cdot \nabla_v \mathbf{P}^{\perp} f^\varepsilon
+ \frac{q_0}{2} E^\varepsilon \cdot v \mathbf{P}^{\perp} f^\varepsilon \Big),  \widetilde{w}^2_l(\alpha,\beta) \partial^\alpha_\beta \mathbf{P}^{\perp} f^\varepsilon \right) \nonumber \\
&+ \left( \partial^\alpha_\beta \mathbf{P} \left( \frac{1}{\varepsilon} \Big[v \cdot \nabla_x
+ q_0 (\varepsilon E^\varepsilon + v \times B^\varepsilon) \cdot \nabla_v \Big]\mathbf{P}^{\perp} f^\varepsilon \right), \widetilde{w}^2_l(\alpha,\beta) \partial^\alpha_\beta \mathbf{P}^{\perp} f^\varepsilon \right) \nonumber \\
&- \left( \partial^\alpha_\beta \mathbf{P} \left( \frac{q_0}{2} E^\varepsilon \cdot v \mathbf{P}^{\perp} f^\varepsilon \right)
- \partial^\alpha_\beta \mathbf{P}^{\perp} \left( \frac{q_0}{2} E^\varepsilon \cdot v \mathbf{P} f^\varepsilon \right) ,
 \widetilde{w}^2_l(\alpha,\beta) \partial^\alpha_\beta \mathbf{P}^{\perp} f^\varepsilon \right) \nonumber\\
&- \left( \partial^\alpha_\beta \mathbf{P}^{\perp} \left( \frac{1}{\varepsilon} \Big[v \cdot \nabla_x
+ q_0 (\varepsilon E^\varepsilon + v \times B^\varepsilon) \cdot \nabla_v \Big]\mathbf{P} f^\varepsilon \right), \widetilde{w}^2_l(\alpha,\beta) \partial^\alpha_\beta \mathbf{P}^{\perp} f^\varepsilon \right).
\end{align}

The last term on the left-hand side of \eqref{with weight 4} follows from \eqref{L coercive3} that
\begin{align*}
&\left( \frac{1}{\varepsilon^2} \partial^\alpha_\beta  L \mathbf{P}^{\perp} f^\varepsilon,  \widetilde{w}^2_l(\alpha,\beta) \partial^\alpha_\beta \mathbf{P}^{\perp} f^\varepsilon \right) \\
\geq \;& \frac{\lambda}{\varepsilon^2} \left\| \widetilde{w}_l(\alpha,\beta) \partial^\alpha_\beta \mathbf{P}^{\perp} f^\varepsilon \right\|_{D}^2
-\frac{C}{\varepsilon^2} \sum_{|\beta^\prime| < |\beta|} \left\|  \widetilde{w}_l(\alpha,\beta^\prime) \partial^\alpha_{\beta^\prime} \mathbf{P}^{\perp} f^\varepsilon \right\|^2_{D}
- \frac{C}{\varepsilon^2} \left\| \partial^\alpha \mathbf{P}^{\perp} f^\varepsilon \right\|^2_{D},
\end{align*}
where we have employed the fact $(1+t)^{-\frac{1+\varepsilon_0}{2}} \leq 1$.

We will estimate the terms on the R.H.S. of \eqref{with weight 4} one by one. For the first term on the right-hand side of \eqref{with weight 4}, recalling the definition of $\widetilde{w}_l(\alpha,\beta)$ in \eqref{weight function1}, similar to \eqref{transport1}, we have
\begin{align*}
&-\left( \frac{1}{\varepsilon} \partial^\alpha_\beta (v \cdot \nabla_x \mathbf{P}^{\perp} f^\varepsilon), \widetilde{w}^2_l(\alpha,\beta) \partial^\alpha_\beta \mathbf{P}^{\perp} f^\varepsilon \right) \\
\lesssim\;& \frac{\eta}{\varepsilon^2} \left\| \widetilde{w}_l(\alpha,\beta)   \partial^{\alpha}_{\beta} \mathbf{P}^{\perp} f^\varepsilon \right\|_D^2
+\left\| \widetilde{w}_l(\alpha+e_i,\beta-e_i)  \partial^{\alpha+e_i}_{\beta-e_i} \mathbf{P}^{\perp} f^\varepsilon \right\|_D^2.
\end{align*}
The second term on the right-hand side of \eqref{with weight 4} can be bounded by
$$
\frac{\eta}{\varepsilon^2}\left\| \widetilde{w}_l(\alpha,\beta) \partial^\alpha_\beta \mathbf{P}^{\perp} f^\varepsilon \right\|_{D}^2
+ \left\| \partial^\alpha E^\varepsilon \right\|^2 $$
for $|\alpha| \leq N-1$.
Thanks to \eqref{soft gamma6} and \eqref{hard gamma6}, the third term on the right-hand side of \eqref{with weight 4} is bounded by
$\sqrt{\widetilde{\mathcal{E}}_{N,l}(t)} \widetilde{\mathcal{D}}_{N,l}(t)$.
For the fourth term on the right-hand side of \eqref{with weight 4}, from \eqref{soft B 1-4}, \eqref{hard B 1-4}, \eqref{soft E 1-4}, \eqref{hard E 1-4}, \eqref{soft E 2-2} and \eqref{hard E 2-2}, we can bound it by
$\left\{ \delta_0 + \sqrt{\mathcal{E}_{N}(t)} \right\} \widetilde{\mathcal{D}}_{N,l}(t)$.
Finally, from the definition of $\mathbf{P}f^\varepsilon$ in \eqref{Pf define}, the H\"{o}lder inequality and \eqref{Sobolev-ineq}, we can obtain that the last three terms on the right-hand side of \eqref{with weight 4} can be governed by
$$
\frac{\eta}{\varepsilon^2} \left\| \widetilde{w}_l(\alpha,\beta)\partial^\alpha_\beta \mathbf{P}^{\perp}f^\varepsilon \right\|^2_{D}
+ \left\| \partial^\alpha \nabla_x \mathbf{P}^{\perp}f^\varepsilon \right\|_{D}^2
+ \left\| \partial^\alpha \nabla_x (a_{\pm}^\varepsilon, b^\varepsilon, c^\varepsilon) \right\|^2
+ \mathcal{E}_{N}(t)\mathcal{D}_{N}(t).
$$

Consequently, by plugging all the estimates above into \eqref{with weight 4}, taking the summation over $\left\{ |\beta|=m,|\alpha|+|\beta| \leq N\right\}$ for each given $1 \leq m \leq N$, and then taking combination of those $N$ estimates with properly chosen constant $\widetilde{C}_m > 0 $ $(1 \leq m \leq N)$ and $\eta$ small enough,
we obtain \eqref{weighted estimate4}. This completes the proof of Lemma \ref{weighted 4}.
\end{proof}

\subsection{Energy Estimate in Negative Sobolev Space}
\hspace*{\fill}

This subsection aims to provide an energy estimate in negative Sobolev space.
\begin{lemma}\label{negative sobolev lemma}
There holds
\begin{align}
\begin{split}\label{negative sobolev estimate}
&\frac{1}{2}\frac{\d}{\d t}  \left\|\Lambda^{-\varrho}(f^\varepsilon, E^\varepsilon, B^\varepsilon)\right\|^2
+\frac{\lambda}{\varepsilon^2}\left\|\Lambda^{-\varrho}\mathbf{P}^{\perp}f^\varepsilon\right\|_{D}^2
\lesssim \sqrt{ \overline{\mathcal{E}}_{N,l}(t) }\overline{\mathcal{D}}_{N,l}(t) .
\end{split}
\end{align}
\end{lemma}

\begin{proof}
Applying $\Lambda^{-\varrho}$ to the first equation of \eqref{rVMB} and taking the inner product with $\Lambda^{-\varrho} f^\varepsilon$ over $\mathbb{R}_x^3 \times \mathbb{R}_v^3$, one has
\begin{align}\label{negative sobolev estimate 1}
&\frac{1}{2}\frac{\d}{\d t} \left\| \Lambda^{-\varrho} f^\varepsilon \right\|^2
+\left( \frac{1}{\varepsilon} \Lambda^{-\varrho} E^\varepsilon \cdot v \mu^{1/2}q_1 ,\Lambda^{-\varrho} f^\varepsilon \right)
+\left( \frac{1}{\varepsilon^2} L\Lambda^{-\varrho} f^\varepsilon, \Lambda^{-\varrho} f^\varepsilon\right) \nonumber \\
=\;&-\left( \frac{1}{\varepsilon } \Lambda^{-\varrho} (q_0  (v \times B^\varepsilon) \cdot \nabla_v f^\varepsilon ), \Lambda^{-\varrho} f^\varepsilon \right)
- \left( \Lambda^{-\varrho} ( q_0 E^\varepsilon \cdot \nabla_v f^\varepsilon ), \Lambda^{-\varrho} f^\varepsilon \right)\\
&+ \left( \Lambda^{-\varrho} \Big(\frac{q_0}{2} E^\varepsilon \cdot v  f^\varepsilon \Big), \Lambda^{-\varrho} f^\varepsilon \right)
+ \left( \frac{1}{\varepsilon } \Lambda^{-\varrho} \Gamma(f^\varepsilon, f^\varepsilon), \Lambda^{-\varrho} f^\varepsilon \right), \nonumber
\end{align}

For the second term on the left-hand side of \eqref{negative sobolev estimate 1}, we have
\begin{align}\nonumber
\left( \frac{1}{\varepsilon} \Lambda^{-\varrho} E^\varepsilon \cdot v \mu^{1/2}q_1 ,\Lambda^{-\varrho} f^\varepsilon \right)
=\frac{1}{2}\frac{\d}{\d t} \left\| \Lambda^{-\varrho} (E^\varepsilon, B^\varepsilon) \right\|^2.
\end{align}
For the third term on the left-hand side of \eqref{negative sobolev estimate 1}, one has from \eqref{L coercive1} that
\begin{align}\nonumber
\left( \frac{1}{\varepsilon^2} L\Lambda^{-\varrho} f^\varepsilon, \Lambda^{-\varrho} f^\varepsilon \right)
\geq \frac{\lambda}{\varepsilon^2} \left\| \Lambda^{-\varrho} \mathbf{P}^{\perp}f^\varepsilon \right\|_{D}^2.
\end{align}

To estimate the first term on the right-hand side of \eqref{negative sobolev estimate 1}, in terms of the macro-micro decomposition \eqref{f decomposition}, we set
\begin{align}\label{J1 estimate}
J_1=\;&\left( \frac{1}{\varepsilon } \Lambda^{-\varrho} (q_0  (v \times B^\varepsilon) \cdot \nabla_v f^\varepsilon ), \Lambda^{-\varrho} f^\varepsilon \right)  \nonumber \\
=\;&\left( \frac{1}{\varepsilon} \Lambda^{-\varrho}  ( q_0(v \times B^\varepsilon) \cdot \nabla_v \mathbf{P} f^{\varepsilon} ), \Lambda^{-\varrho} \mathbf{P} f^{\varepsilon} \right) \nonumber\\
&+ \left( \frac{1}{\varepsilon} \Lambda^{-\varrho}  ( q_0(v \times B^\varepsilon) \cdot \nabla_v \mathbf{P}^{\perp} f^{\varepsilon} ), \Lambda^{-\varrho} \mathbf{P} f^{\varepsilon} \right) \nonumber\\
&+\left( \frac{1}{\varepsilon} \Lambda^{-\varrho}  ( q_0(v \times B^\varepsilon) \cdot \nabla_v  f^{\varepsilon} ), \Lambda^{-\varrho} \mathbf{P}^{\perp} f^{\varepsilon} \right) \nonumber\\
:=\;& J_{1,1}+ J_{1,2}+J_{1,3}.
\end{align}
For the term $J_{1,1}$, it follows from the kernel structure of $\mathbf{P}$ that $J_{1,1}=0$, cf. \cite[pp. 28]{JL2019}.
For the term $J_{1,2}$, owing to $1 < \varrho < \frac{3}{2}$, by using \eqref{negative embed 1}, \eqref{negative embed 2}, \eqref{minkowski}, the H\"{o}lder inequality and the Cauchy--Schwarz inequality, one has
\begin{align*}
J_{1,2} \lesssim\;& \frac{1}{\varepsilon}\left\| \Lambda^{-\varrho} (B^\varepsilon \mu^\delta \mathbf{P}^{\perp}f^\varepsilon) \right\|
\left\| \Lambda^{-\varrho} \mathbf{P} f^{\varepsilon} \right\| \\
\lesssim\;& \frac{1}{\varepsilon}\left\| B^\varepsilon \mu^\delta \mathbf{P}^{\perp}f^\varepsilon \right\|_{L^2_v L_x^{\frac{6}{3+2\varrho}}}
\left\| \Lambda^{-\varrho} \mathbf{P} f^{\varepsilon}  \right\| \\
\lesssim\;&\frac{1}{\varepsilon}\left\| B^\varepsilon \right\|_{L^\frac{12}{3+2\varrho}} \left\| \mu^\delta \mathbf{P}^{\perp}f^\varepsilon \right\|_{L^2_v L^\frac{12}{3+2\varrho}_x}
\left\| \Lambda^{-\varrho} \mathbf{P} f^{\varepsilon} \right\| \\
\lesssim\;&\frac{1}{\varepsilon}\left\| \Lambda^{\frac{3}{4}-\frac{\varrho}{2}}B^\varepsilon \right\|
\left\| \Lambda^{\frac{3}{4}-\frac{\varrho}{2}} \big(\mu^\delta \mathbf{P}^{\perp}f^\varepsilon \big) \right\|
\left\| \Lambda^{-\varrho} \mathbf{P} f^{\varepsilon} \right\| \\
\lesssim\;& \sqrt{ \overline{\mathcal{E}}_{N,l}(t) } \left\{ \left\| \Lambda^{\frac{3}{4}-\frac{\varrho}{2}} B^\varepsilon \right\|^2 +
\frac{1}{\varepsilon^2}\left\| \Lambda^{\frac{3}{4}-\frac{\varrho}{2}} \mathbf{P}^{\perp}f^\varepsilon \right\|^2_D \right\}.
\end{align*}
Using the similar argument as the term $J_{1,2}$, we have that
\begin{align*}
J_{1,3} \lesssim\;& \frac{1}{\varepsilon}\left\| \Lambda^{-\varrho} \big(B^\varepsilon \langle v \rangle ^{1-\frac{\gamma+2s}{2}} \nabla_v f^\varepsilon \big) \right\|
\left\| \Lambda^{-\varrho} \mathbf{P}^{\perp} f^{\varepsilon} \right\|_D \\
\lesssim\;& \frac{1}{\varepsilon}\left\| B^\varepsilon \langle v \rangle ^{1-\frac{\gamma+2s}{2}} \nabla_v f^\varepsilon \right\|_{L^2_v L_x^{\frac{6}{3+2\varrho}}}
\left\| \Lambda^{-\varrho} \mathbf{P}^{\perp} f^{\varepsilon} \right\|_D \\
\lesssim\;&\frac{1}{\varepsilon}\left\| B^\varepsilon \right\|_{L^\frac{3}{\varrho}}
\left\| \langle v \rangle ^{1-\frac{\gamma+2s}{2}} \nabla_v f^\varepsilon \right\|
\left\| \Lambda^{-\varrho} \mathbf{P}^{\perp} f^{\varepsilon} \right\|_D \\
\lesssim\;&\frac{1}{\varepsilon}\left\| \Lambda^{\frac{3}{2}-\varrho}B^\varepsilon \right\|
\left\| \langle v \rangle ^{1-\frac{\gamma+2s}{2}} \nabla_v f^\varepsilon \right\|
\left\| \Lambda^{-\varrho} \mathbf{P}^{\perp} f^{\varepsilon} \right\|_D \\
\lesssim\;& \sqrt{ \overline{\mathcal{E}}_{N,l}(t) } \left\{ \left\| \Lambda^{\frac{3}{2}-\varrho}B^\varepsilon \right\|^2 +
\frac{1}{\varepsilon^2}\left\| \Lambda^{-\varrho} \mathbf{P}^{\perp} f^{\varepsilon} \right\|^2_D \right\}.
\end{align*}
Therefore, by substituting the estimates of $J_{1,1}$--$J_{1,3}$ into \eqref{J1 estimate}, one has
\begin{align*}
J_1 \lesssim \sqrt{ \overline{\mathcal{E}}_{N,l}(t) } \left\{ \left\| \Lambda^{\frac{3}{4}-\frac{\varrho}{2}} B^\varepsilon \right\|^2 +
\frac{1}{\varepsilon^2}\left\| \Lambda^{\frac{3}{4}-\frac{\varrho}{2}} \mathbf{P}^{\perp}f^\varepsilon \right\|^2_D
+\left\| \Lambda^{\frac{3}{2}-\varrho}B^\varepsilon \right\|^2 +
\frac{1}{\varepsilon^2}\left\| \Lambda^{-\varrho} \mathbf{P}^{\perp} f^{\varepsilon} \right\|^2_D\right\}.
\end{align*}
Similarly to $J_{1,2}$ and $J_{1,3}$, the second and third terms on the right-hand side of \eqref{negative sobolev estimate 1} have the following upper bound
\begin{align*}
\sqrt{ \overline{\mathcal{E}}_{N,l}(t) } \left\{ \left\| \Lambda^{\frac{3}{4}-\frac{\varrho}{2}} E^\varepsilon \right\|^2 +
\left\| \Lambda^{\frac{3}{4}-\frac{\varrho}{2}} f^\varepsilon \right\|^2_D
+ \left\| \Lambda^{\frac{3}{2}-\varrho}E^\varepsilon \right\|^2 +
\left\| \Lambda^{-\varrho} \mathbf{P}^{\perp} f^{\varepsilon} \right\|^2_D \right\}.
\end{align*}
For the last term on the right-hand side of \eqref{negative sobolev estimate 1}, by the collision invariant property, we have
\begin{align} \label{negative sobolev Gamma1}
\left( \frac{1}{\varepsilon} \Lambda^{-\varrho}\Gamma(f^\varepsilon, f^\varepsilon), \Lambda^{-\varrho} f^\varepsilon \right)
=\;&\left( \frac{1}{\varepsilon} \Lambda^{-\varrho}\Gamma(f^\varepsilon, f^\varepsilon),
\Lambda^{-\varrho} \mathbf{P}^{\perp}f^\varepsilon \right) \nonumber\\
\lesssim \;& \frac{1}{\varepsilon} \left\|\Lambda^{-\varrho}\left( \langle v \rangle ^{-\frac{\gamma+2s}{2}}\Gamma(f^\varepsilon, f^\varepsilon)\right) \right\|
\left\| \Lambda^{-\varrho} \mathbf{P}^{\perp}f^\varepsilon \right\|_{D} \nonumber\\
\lesssim\;& \frac{1}{\varepsilon} \left\| \langle v \rangle ^{-\frac{\gamma+2s}{2}}\Gamma(f^\varepsilon, f^\varepsilon) \right\|_{L^2_{v}L_x^{\frac{6}{3+2\varrho}}}
\left\| \Lambda^{-\varrho} \mathbf{P}^{\perp}f^\varepsilon \right\|_{D} \nonumber\\
\lesssim\;&\frac{1}{\varepsilon} \left\| \langle v \rangle ^{-\frac{\gamma+2s}{2}}\Gamma(f^\varepsilon, f^\varepsilon) \right\|_{L_x^{\frac{6}{3+2\varrho}}L^2_{v}}
\left\| \Lambda^{-\varrho} \mathbf{P}^{\perp}f^\varepsilon \right\|_{D} \\
\lesssim\;&\frac{1}{\varepsilon} \left\| |f^\varepsilon|_{L^2} |f^\varepsilon|_{H^4} \right\|_{L_x^{\frac{6}{3+2\varrho}}}
\left\| \Lambda^{-\varrho} \mathbf{P}^{\perp}f^\varepsilon \right\|_{D} \nonumber\\
\lesssim\;&\frac{1}{\varepsilon} \left\|f^\varepsilon\right\|_{L_x^{\frac{3}{\varrho}}L^2_v} \left\|f^\varepsilon\right\|_{H^4_vL^2_x}
\left\| \Lambda^{-\varrho} \mathbf{P}^{\perp}f^\varepsilon \right\|_{D} \nonumber\\
\lesssim\;& \sqrt{ \overline{\mathcal{E}}_{N,l}(t) } \left\{ \left\| \Lambda^{\frac{3}{2}-\varrho}f^\varepsilon\right\|^2
+\frac{1}{\varepsilon^2} \left\| \Lambda^{-\varrho} \mathbf{P}^{\perp} f^\varepsilon \right\|_{D}^2 \right\} \nonumber
\end{align}
for soft potentials $\max\{-3, -\frac{3}{2}-2s\} < \gamma < -2s$, where we have used \eqref{soft gamma L2}, \eqref{negative embed 1}, \eqref{negative embed 2}, \eqref{minkowski}, the H\"{o}lder inequality and the Cauchy--Schwarz inequality. The case $\gamma+2s \geq 0$ can be treated in the similar way as above. Hence, by using \eqref{hard gamma L2}, we can obtain the same upper bound as \eqref{negative sobolev Gamma1} for the hard potential case  $\gamma+2s \geq 0$.

Recalling the definition of $\overline{\mathcal{D}}_{N,l}(t)$ in \eqref{negative sobolev dissipation}, by the interpolation inequality with respect to the spatial derivatives, it is easy to derive that
\begin{align*}
&\left\| \Lambda^{\frac{3}{4}-\frac{\varrho}{2}} (E^\varepsilon, B^\varepsilon) \right\|^2
+\left\| \Lambda^{\frac{3}{2}-\varrho} (E^\varepsilon, B^\varepsilon) \right\|^2
\lesssim \left\| \Lambda^{1-\varrho} (E^\varepsilon, B^\varepsilon) \right\|^2
+\left\| \nabla_x (E^\varepsilon, B^\varepsilon) \right\|^2
\lesssim \overline{\mathcal{D}}_{N,l}(t),   \\
&\frac{1}{\varepsilon^2}\left\| \Lambda^{\frac{3}{4}-\frac{\varrho}{2}} \mathbf{P}^{\perp}f^\varepsilon \right\|^2_D
\lesssim \frac{1}{\varepsilon^2} \left\| \Lambda^{-\varrho} \mathbf{P}^{\perp}f^\varepsilon \right\|^2_D
+ \frac{1}{\varepsilon^2} \left\| \nabla_x \mathbf{P}^{\perp}f^\varepsilon \right\|^2_D
\lesssim \overline{\mathcal{D}}_{N,l}(t), \\
&\left\| \Lambda^{\frac{3}{4}-\frac{\varrho}{2}}f^\varepsilon  \right\|^2_{D}
\lesssim  \left\| \Lambda^{1-\varrho}\mathbf{P}f^\varepsilon \right\|^2
+\left\| \nabla_x \mathbf{P}f^\varepsilon \right\|^2
+\left\| \Lambda^{-\varrho}\mathbf{P}^{\perp}f^\varepsilon \right\|^2_{D}
+\left\| \nabla_x\mathbf{P}^{\perp}f^\varepsilon \right\|^2_{D}
\lesssim \overline{\mathcal{D}}_{N,l}(t), \\
&\left\| \Lambda^{\frac{3}{2}-\varrho}f^\varepsilon \right\|^2
\lesssim  \left\| \Lambda^{1-\varrho} \mathbf{P}f^\varepsilon \right\|^2
+\left\| \nabla_x \mathbf{P}f^\varepsilon \right\|^2
+\left\| \langle v \rangle ^{-\frac{\gamma+2s}{2}} \mathbf{P}^{\perp}f^\varepsilon \right\|_{H^1_x L^2_D}^2
\lesssim \overline{\mathcal{D}}_{N,l}(t),
\end{align*}
where we employed the fact $1-\varrho < \frac{3}{4}-\frac{\varrho}{2} < \frac{3}{2}-\varrho < 1$ for $1 < \varrho < \frac{3}{2}$.

As a result, combining all the estimates above, we can get \eqref{negative sobolev estimate}. This completes the proof of Lemma \ref{negative sobolev lemma}.
\end{proof}

\subsection{Uniform Energy Estimates}
\hspace*{\fill}

In this subsection, based on Lemma \ref{macroscopic estimate}--\ref{negative sobolev lemma}, we give the uniform a priori estimates
with respect to $\varepsilon\in (0,1]$.
\begin{proposition}\label{basic energy estimate}
There is $\mathcal{E}_{N}(t)$ satisfying \eqref{without weight energy functional} such that
\begin{align}\label{a priori estimate 1}
\frac{\d}{\d t} \mathcal{E}_{N}(t) + \lambda  \mathcal{D}_{N}(t)
\lesssim \left\{ \delta_0 + \mathcal{E}_{N}(t) + \sqrt{\widetilde{\mathcal{E}}_{N,l}(t)} \right\} \widetilde{\mathcal{D}}_{N,l}(t).
\end{align}
\end{proposition}

\begin{proof}
Under the a priori assumption \eqref{priori assumption}, the combination $\kappa_1 \times \eqref{macro estimate} + \eqref{estimate without weight}$ gives \eqref{a priori estimate 1} for given $0 < \kappa_1 \ll 1$. This completes the proof of Proposition \ref{basic energy estimate}.
\end{proof}

\begin{proposition}\label{energy estimate 1}
There is $\widetilde{\mathcal{E}}_{N,l}(t)$ satisfying \eqref{energy functional} such that
\begin{align}\label{a priori estimate 2}
\frac{\d}{\d t} \widetilde{\mathcal{E}}_{N,l}(t) + \lambda  \widetilde{\mathcal{D}}_{N,l}(t)
\lesssim 0.
\end{align}
\end{proposition}

\begin{proof}
First of all, multiplying $(1+t)^{-\varepsilon_0}$ by \eqref{a priori estimate 1}, we have
\begin{align}
&\frac{\d}{\d t} \left\{ (1+t)^{-\varepsilon_0} \mathcal{E}_{N}(t) \right\}
+ \frac{\varepsilon_0}{(1+t)^{1+\varepsilon_0}} \mathcal{E}_{N}(t)
+ (1+t)^{-\varepsilon_0} \mathcal{D}_{N}(t) \nonumber \\
\lesssim\;& \left\{ \delta_0 +  \mathcal{E}_{N}(t) + \sqrt{\widetilde{\mathcal{E}}_{N,l}(t)} \right\} \widetilde{\mathcal{D}}_{N,l}(t).
\label{a priori estimate 3}
\end{align}
Next, for $\overline{C}_1 \gg \overline{C}_2 $ large enough, the combination $\overline{C}_1 \times \left\{ \eqref{weighted estimate1} + \eqref{weighted estimate3} \right\} + \overline{C}_2 \times \eqref{weighted estimate2} + \eqref{weighted estimate4}$ gives that
\begin{align}
&\frac{\d}{\d t} \widetilde{\mathcal{E}}^h_{N,l}(t) + \lambda  \widetilde{\mathcal{D}}^h_{N,l}(t) \nonumber\\
\lesssim\;& \mathcal{D}_{N}(t)
+ \left\{ \delta_0 + \mathcal{E}_{N}(t)+ \sqrt{\widetilde{\mathcal{E}}_{N,l}(t)} \right\} \widetilde{\mathcal{D}}_{N,l}(t)
+\frac{1}{(1+t)^{1+\varepsilon_0}}\sum_{|\alpha| = N}\left\| \partial^\alpha E^\varepsilon \right\|^2, \label{a priori estimate 4}
\end{align}
where $\widetilde{\mathcal{E}}^h_{N,l}(t)$ is given by
\begin{align*}
\widetilde{\mathcal{E}}^h_{N,l}(t)
:=\; &\overline{C}_1\left\{ \sum_{|\alpha|\leq N-1} \left\|  \widetilde{w}_l(\alpha,0) \partial^\alpha \mathbf{P}^{\perp} f^\varepsilon \right\|^2
+ \varepsilon \sum_{|\alpha| = N} \left\| \widetilde{w}_l(\alpha,0) \partial^\alpha f^\varepsilon \right\|^2 \right\}\\
&+\overline{C}_2 \sum_{m=1}^{N-1}C_m \sum_{\substack{{|\alpha|+|\beta| \leq N-1}\\{|\beta| =m}}}  \left\| \widetilde{w}_l(\alpha,\beta) \partial^\alpha_\beta \mathbf{P}^{\perp}f^\varepsilon \right\|^2\\
&+\sum_{m=1}^{N} \widetilde{C}_m \sum_{\substack{{|\alpha|+|\beta| = N}\\{|\beta| =m}}} \left\| \widetilde{w}_l(\alpha,\beta) \partial^\alpha_\beta \mathbf{P}^{\perp}f^\varepsilon \right\|^2.
\end{align*}
Finally, we take the proper linear combination of the estimates above as follows.
For $\overline{C}_3 > 0$ large enough,  the combination
$\overline{C}_3 \times \left\{ \eqref{a priori estimate 1} + \eqref{a priori estimate 3} \right\} + \eqref{a priori estimate 4}$ gives
\begin{align}\label{a priori estimate 5}
\frac{\d}{\d t} \widetilde{\mathcal{E}}_{N,l}(t) + \lambda  \widetilde{\mathcal{D}}_{N,l}(t)
\lesssim \left\{ \delta_0 + \mathcal{E}_{N}(t)+ \sqrt{\widetilde{\mathcal{E}}_{N,l}(t)} \right\} \widetilde{\mathcal{D}}_{N,l}(t).
\end{align}

Recalling the a priori assumption \eqref{priori assumption}, the desired estimate \eqref{a priori estimate 2} follows directly from \eqref{a priori estimate 5}. This completes the proof of Proposition \ref{energy estimate 1}.
\end{proof}

\begin{proposition}\label{negative sobolev energy estimates total}
There is $\overline{\mathcal{E}}_{N,l}(t)$ satisfying \eqref{negative sobolev energy} such that
\begin{align}\label{a priori estimate 6}
\frac{\d}{\d t} \overline{\mathcal{E}}_{N,l}(t) + \lambda  \overline{\mathcal{D}}_{N,l}(t) \lesssim 0.
\end{align}
\end{proposition}

\begin{proof}
Similar to Lemma \ref{macroscopic estimate}, there exists an interactive functional $\mathcal{E}^{-\varrho}_{int}(t)$ satisfying
\begin{align}\nonumber
\left|\mathcal{E}^{-\varrho}_{int}(t)\right| \lesssim \left\| \Lambda^{-\varrho} (f^\varepsilon, E^\varepsilon, B^\varepsilon) \right\|^2
+ \left\| \Lambda^{1-\varrho} (f^\varepsilon, E^\varepsilon, B^\varepsilon) \right\|^2,
\end{align}
such that
\begin{align}\label{macro negative sobolev 1}
& \frac{\d}{\d t} \mathcal{E}^{-\varrho}_{int}(t)
+ \left\| \Lambda^{1-\varrho} (a_{\pm}^\varepsilon, b^\varepsilon, c^\varepsilon) \right\|^2
+ \left\| \Lambda^{-\varrho} (a_{+}^ \varepsilon-a_{-}^\varepsilon) \right\|^2
+ \left\| \Lambda^{-\varrho} E^ \varepsilon \right\|_{H^1}^2
+ \left\| \Lambda^{1-\varrho} B^ \varepsilon \right\|^2 \nonumber \\
\lesssim\;& \frac{1}{\varepsilon^2}\left\| \Lambda^{-\varrho}\mathbf{P}^{\perp}f^\varepsilon \right\|^2_{H^2_x L^2_D}
+\overline{\mathcal{E}}_{N,l}(t)\overline{\mathcal{D}}_{N,l}(t).
\end{align}

Under the a priori assumption \eqref{priori assumption},  one can deduce \eqref{a priori estimate 6} by a proper linear combination of \eqref{negative sobolev estimate}, \eqref{a priori estimate 2} and \eqref{macro negative sobolev 1}. This completes the proof of Proposition \ref{negative sobolev energy estimates total}.
\end{proof}
\medskip

\section{Global Existence}\label{Global Existence}
In this section, we give the proof of Theorem \ref{mainth1} by first establishing the closed uniform estimates on $X(t)$ defined in \eqref{X define}.

\subsection{Time Decay Estimate}
\hspace*{\fill}

Our main idea to deduce the time decay estimate is based on the approach proposed by \cite{GW2012CPDE}.
First of all, we give a more refined energy estimate for the pure spatial derivatives than \eqref{estimate without weight}.

\begin{lemma}\label{0-N0 1-N0 lemma}
There hold
\begin{align}
&\frac{\d}{\d t}  \sum_{k \leq N_0} \left\| \nabla^k_x (f^\varepsilon, E^\varepsilon, B^\varepsilon) \right\|^2
+ \frac{\lambda}{\varepsilon^2} \sum_{k \leq N_0 }  \left\| \nabla^k_x \mathbf{P}^{\perp} f^\varepsilon \right\|^2_{D} \nonumber \\
\lesssim\;& \sqrt{\widetilde{\mathcal{E}}_{N,l}(t)} \bigg\{ \sum_{k \leq N_0-1} \left\| \nabla^k_x E^\varepsilon \right\|^2
+ \sum_{2 \leq k \leq N_0-1} \left\| \nabla^k_x B^\varepsilon \right\|^2
+  \sum_{1 \leq k \leq N_0} \left\| \nabla^k_x \mathbf{P}f^\varepsilon \right\|^2 \bigg\}  \label{0-N0 estimate decay}
\end{align}
and
\begin{align}
&\frac{\d}{\d t}  \sum_{1 \leq k \leq N_0} \left\| \nabla^k_x (f^\varepsilon, E^\varepsilon, B^\varepsilon) \right\|^2
+ \frac{\lambda}{\varepsilon^2} \sum_{1 \leq k \leq N_0 }  \left\| \nabla^k_x \mathbf{P}^{\perp} f^\varepsilon \right\|^2_{D} \nonumber \\
\lesssim\;& \sqrt{\widetilde{\mathcal{E}}_{N,l}(t)} \bigg\{ \sum_{1 \leq k \leq N_0-1} \left\| \nabla^k_x E^\varepsilon \right\|^2
+ \sum_{2 \leq k \leq N_0-1} \left\| \nabla^k_x B^\varepsilon \right\|^2
+  \sum_{2 \leq k \leq N_0} \left\| \nabla^k_x \mathbf{P}f^\varepsilon \right\|^2 \bigg\}. \label{1-N0 estimate decay}
\end{align}
\end{lemma}

\begin{proof}
Similar to \eqref{estimate without weight}, by applying $\nabla^k_x$ to the first equation of \eqref{rVMB} and taking the inner product with $\nabla^k_x f^\varepsilon$ over $\mathbb{R}_x^3 \times \mathbb{R}_v^3$, one has
\begin{align} \label{pure spatial estimate}
&\frac{1}{2}\frac{\d}{\d t}  \left\| \nabla^k_x (f^\varepsilon, E^\varepsilon, B^\varepsilon) \right\|^2
+\frac{\lambda}{\varepsilon^2}  \left\| \nabla^k_x \mathbf{P}^{\perp} f^\varepsilon \right\|^2_{D} \nonumber\\
=\;&- \chi_{k \geq 1} \left( \frac{1}{\varepsilon }\nabla^k_x (q_0  (v \times B^\varepsilon) \cdot \nabla_v f^\varepsilon ), \nabla^k_x f^\varepsilon \right)
- \chi_{k \geq 1} \left( \nabla^k_x (q_0 E^\varepsilon \cdot \nabla_v f^\varepsilon ), \nabla^k_x f^\varepsilon \right) \nonumber\\
&+ \left( \nabla^k_x \big(\frac{q_0}{2} E^\varepsilon \cdot v  f^\varepsilon \big), \nabla^k_x f^\varepsilon \right)
+ \left( \frac{1}{\varepsilon } \nabla^k_x \Gamma(f^\varepsilon, f^\varepsilon), \nabla^k_x f^\varepsilon \right) \nonumber\\
:=\;& J_{2}+ J_{3}+ J_{4}+J_{5}.
\end{align}

We first consider $J_{2}$--$J_{5}$ with $1 \leq k \leq N_0$ for the case $\max\{-3, -\frac{3}{2}-2s\} < \gamma <-2s$.
To estimate the term $J_{2}$, in terms of the macro-micro decomposition \eqref{f decomposition}, we set
\begin{align*}
J_2=\;&  \sum_{1 \leq j \leq k} C_k^j \left(\frac{1}{\varepsilon } q_0  (v \times \nabla^j_x B^\varepsilon) \cdot \nabla_v \nabla^{k-j}_x f^\varepsilon , \nabla^k_x f^\varepsilon \right)  \nonumber \\
=\;& \sum_{1 \leq j \leq k} C_k^j \left( \frac{1}{\varepsilon} q_0(v \times \nabla^j_x B^\varepsilon) \cdot \nabla_v \nabla^{k-j}_x \mathbf{P}^{\perp} f^{\varepsilon} , \nabla^k_x \mathbf{P} f^{\varepsilon} \right) \nonumber\\
&+ \sum_{1 \leq j \leq k} C_k^j \left( \frac{1}{\varepsilon}  q_0(v \times \nabla^j_x B^\varepsilon) \cdot \nabla_v  \nabla^{k-j}_x \mathbf{P} f^{\varepsilon}, \nabla^k_x \mathbf{P}^{\perp} f^{\varepsilon} \right) \nonumber\\
&+ \sum_{1 \leq j \leq k} C_k^j \left( \frac{1}{\varepsilon}   q_0(v \times \nabla^j_x B^\varepsilon) \cdot \nabla_v \nabla^{k-j}_x  \mathbf{P}^{\perp} f^{\varepsilon} , \nabla^k_x \mathbf{P}^{\perp} f^{\varepsilon} \right) \nonumber\\
:=\;& J_{2,1}+ J_{2,2}+J_{2,3}.
\end{align*}
Here, we have used the fact that from the kernel structure of $\mathbf{P}$, there holds
$$
\sum_{1 \leq j \leq k} C_k^j \left( \frac{1}{\varepsilon}  q_0(v \times \nabla^j_x B^\varepsilon) \cdot \nabla_v \nabla^{k-j}_x \mathbf{P} f^{\varepsilon} , \nabla^k_x \mathbf{P} f^{\varepsilon} \right) =0.
$$
For the term $J_{2,1}$, we divide our computations into the following two cases.

{\em Case 1.} $k=1$.\; In this case, $j=k=1$. From integration by parts,  \eqref{sobolev interpolation1},  \eqref{Sobolev-ineq}, \eqref{minkowski}, the H\"{o}lder inequality and the Cauchy--Schwarz inequality, one has
\begin{align*}
J_{2,1} \lesssim\;& \frac{1}{\varepsilon} \left\| \nabla_x B^\varepsilon \right\|
\left\| \mathbf{P}^{\perp} f^\varepsilon \right\|_{L^6_x L^2_D}
\left\| \nabla_x \mathbf{P} f^{\varepsilon} \right\|_{L^3_x L^2_v} \\
\lesssim\;& \frac{1}{\varepsilon} \left\|  B^\varepsilon \right\|^{\frac{1}{2}} \left\| \nabla^2_x B^\varepsilon \right\|^{\frac{1}{2}}
\left\| \nabla_x \mathbf{P}^{\perp} f^\varepsilon \right\|_{D}
\left\| \nabla_x \mathbf{P} f^{\varepsilon} \right\|^{\frac{1}{2}} \left\| \nabla^2_x \mathbf{P} f^{\varepsilon} \right\|^{\frac{1}{2}} \\
\lesssim\;& \sqrt{\mathcal{E}_{N}(t)}
\left\{ \left\| \nabla^2_x B^\varepsilon \right\|^2
+ \left\| \nabla^2_x \mathbf{P} f^{\varepsilon} \right\|^2
+ \frac{1}{\varepsilon^2}\left\| \nabla_x \mathbf{P}^{\perp} f^{\varepsilon} \right\|^2_D \right\}.
\end{align*}

{\em Case 2.} $2 \leq k \leq N_0$.\; Using the similar method as the case $k=1$, we have
\begin{align*}
J_{2,1} \lesssim \;& \frac{1}{\varepsilon} \sum_{1 \leq j \leq k} \left\| \nabla^j_x B^\varepsilon \right\|_{L^3}
\left\| \nabla^{k-j}_x \mathbf{P}^{\perp} f^\varepsilon \right\|_{L^6_x L^2_D}
\left\| \nabla^k_x \mathbf{P} f^{\varepsilon} \right\| \\
\lesssim\;& \sqrt{\mathcal{E}_{N}(t)}
\bigg\{ \sum_{2 \leq \ell \leq N_0}  \left\| \nabla^\ell_x \mathbf{P} f^{\varepsilon} \right\|^2
+ \frac{1}{\varepsilon^2} \sum_{1 \leq \ell \leq N_0} \left\| \nabla^\ell_x \mathbf{P}^{\perp} f^{\varepsilon} \right\|^2_D \bigg\}.
\end{align*}
Next, we estimate the term $J_{2,2}$. It follows from \eqref{sobolev interpolation1}, \eqref{sobolev interpolation2}, \eqref{Sobolev-ineq}, \eqref{minkowski}, the H\"{o}lder inequality and the Cauchy--Schwarz inequality that
\begin{align*}
J_{2,2} \lesssim \;& \frac{1}{\varepsilon} \sum_{1 \leq j \leq N_0-2} \left\| \nabla^j_x B^\varepsilon \right\|_{L^6}
\left\| \nabla^{k-j}_x \mathbf{P} f^\varepsilon \right\|_{L^3_x L^2_v}
\left\| \nabla^k_x \mathbf{P}^{\perp} f^{\varepsilon} \right\|_D \\
& + \frac{1}{\varepsilon} \sum_{ j = N_0-1} \left\| \nabla^j_x B^\varepsilon \right\|
\left\| \nabla^{k-j}_x \mathbf{P} f^\varepsilon \right\|_{L^\infty_x L^2_v}
\left\| \nabla^k_x \mathbf{P}^{\perp} f^{\varepsilon} \right\|_D \\
&+ \frac{1}{\varepsilon} \left\| \nabla^{N_0}_x B^\varepsilon \right\|
\left\| \mathbf{P} f^\varepsilon \right\|_{L^\infty_x L^2_v}
\left\| \nabla^{N_0}_x \mathbf{P}^{\perp} f^{\varepsilon} \right\|_D \\
\lesssim\;& \sqrt{\mathcal{E}_{N}(t)}
\bigg\{ \sum_{2 \leq \ell \leq N_0-1}  \left\| \nabla^\ell_x B^{\varepsilon} \right\|^2
+ \left\| \nabla^2_x \mathbf{P} f^\varepsilon \right\|^2
+ \frac{1}{\varepsilon^2}  \sum_{1 \leq \ell \leq N_0}  \left\| \nabla^{\ell}_x \mathbf{P}^{\perp} f^{\varepsilon} \right\|^2_D \bigg\},
\end{align*}
where we used that
\begin{align*}
&\frac{1}{\varepsilon} \left\| \nabla^{N_0}_x B^\varepsilon \right\|
\left\| \mathbf{P} f^\varepsilon \right\|_{L^\infty_x L^2_v}
\left\| \nabla^{N_0} _x\mathbf{P}^{\perp} f^{\varepsilon} \right\|_D \\
\lesssim\;&\frac{1}{\varepsilon} \left\| \nabla^{N_0-1}_x B^\varepsilon \right\|^{\frac{1}{2}}
\left\| \nabla^{N_0+1}_x B^\varepsilon \right\|^{\frac{1}{2}}
\left\| \nabla_x \mathbf{P} f^\varepsilon \right\|^{\frac{1}{2}}
\left\| \nabla^2_x \mathbf{P} f^\varepsilon \right\|^{\frac{1}{2}}
\left\| \nabla^{N_0}_x \mathbf{P}^{\perp} f^{\varepsilon} \right\|_D \\
\lesssim\;& \sqrt{\mathcal{E}_{N}(t)}
\bigg\{ \left\| \nabla^{N_0-1}_x B^\varepsilon \right\|^2+ \left\| \nabla^2_x \mathbf{P} f^\varepsilon \right\|^2
+ \frac{1}{\varepsilon^2} \left\| \nabla^{N_0}_x \mathbf{P}^{\perp} f^{\varepsilon} \right\|_D^2 \bigg\}.
\end{align*}
Similarly, for the term $J_{2,3}$, we obtain
\begin{align*}
J_{2,3} \lesssim \;& \frac{1}{\varepsilon} \sum_{1 \leq j \leq N_0-2}  \left\| \nabla_x^j B^\varepsilon \right\|_{L^6}
\left\| \langle v \rangle ^{1-\frac{\gamma+2s}{2}} \nabla_v \nabla^{k-j}_x \mathbf{P}^{\perp}  f^\varepsilon \right\|_{L^3_x L^2_v}
\left\| \nabla^k_x \mathbf{P}^{\perp} f^{\varepsilon} \right\|_D \\
&+ \frac{1}{\varepsilon} \sum_{ j = N_0-1} \left\| \nabla^j_x B^\varepsilon \right\|
\left\| \langle v \rangle ^{1-\frac{\gamma+2s}{2}}  \nabla_v \nabla_x^{k-j} \mathbf{P}^{\perp}  f^\varepsilon \right\|_{L^\infty_x L^2_v}
\left\| \nabla^k_x \mathbf{P}^{\perp} f^{\varepsilon} \right\|_D \\
&+ \frac{1}{\varepsilon}  \left\| \nabla^{N_0}_x B^\varepsilon \right\|
\left\| \nabla_v  \mathbf{P}^{\perp}  f^\varepsilon \right\|_{L^\infty_x L^2_v}
\left\| \langle v \rangle \nabla^{N_0}_x \mathbf{P}^{\perp} f^{\varepsilon} \right\| \\
\lesssim\;& \sqrt{\widetilde{\mathcal{E}}_{N,l}(t)}
\bigg\{ \sum_{2 \leq \ell \leq N_0-1}  \left\| \nabla^\ell_x B^{\varepsilon} \right\|^2
+ \frac{1}{\varepsilon^2}  \sum_{1 \leq \ell \leq N_0}\left\| \nabla^{\ell}_x \mathbf{P}^{\perp} f^{\varepsilon} \right\|^2_D \bigg\}.
\end{align*}
Here, we have made use of the estimate
\begin{align*}
&\frac{1}{\varepsilon}  \left\| \nabla^{N_0}_x B^\varepsilon \right\|
\left\| \nabla_v  \mathbf{P}^{\perp}  f^\varepsilon \right\|_{L^\infty_x L^2_v}
\left\| \langle v \rangle \nabla^{N_0}_x \mathbf{P}^{\perp} f^{\varepsilon} \right\|  \\
\lesssim\;& \frac{1}{\varepsilon}  \left\| \nabla^{N_0-1}_x B^\varepsilon \right\|^{\frac{3}{4}}
\left\| \nabla^{N_0+3}_x B^\varepsilon \right\|^{\frac{1}{4}}
\left\| \nabla_x \mathbf{P}^{\perp} f^\varepsilon \right\|_{H^1_x L^2_v}^{\frac{2}{3}}
\left\| \nabla^3_v \nabla_x \mathbf{P}^{\perp} f^\varepsilon \right\|_{H^1_x L^2_v}^{\frac{1}{3}} \\
&\times \left\| \langle v \rangle^{\frac{\gamma+2s}{2}} \nabla^{N_0}_x \mathbf{P}^{\perp} f^{\varepsilon} \right\|^{\frac{3}{4}}
\left\| \langle v \rangle^{4-\frac{3(\gamma+2s)}{2}} \nabla^{N_0}_x \mathbf{P}^{\perp} f^{\varepsilon} \right\|^{\frac{1}{4}} \\
\lesssim\;& \frac{1}{\varepsilon}  \left\| \nabla^{N_0-1}_x B^\varepsilon \right\|^{\frac{3}{4}}
\left\| \nabla^{N_0+3}_x B^\varepsilon \right\|^{\frac{1}{4}}
\left\| \langle v \rangle ^{\frac{\gamma+2s}{2}} \nabla_x \mathbf{P}^{\perp} f^\varepsilon  \right\|_{H^1_x L^2_v}^{\frac{2}{3} \times \frac{3}{4}}
\left\| \langle v \rangle ^{-\frac{3(\gamma+2s)}{2}} \nabla_x \mathbf{P}^{\perp} f^\varepsilon  \right\|_{H^1_x L^2_v}^{\frac{2}{3} \times \frac{1}{4}} \\
&\times \left\| \nabla^3_v \nabla_x \mathbf{P}^{\perp} f^\varepsilon \right\|_{H^1_x L^2_v}^{\frac{1}{3}}
\left\| \langle v \rangle^{\frac{\gamma+2s}{2}} \nabla^{N_0}_x \mathbf{P}^{\perp} f^{\varepsilon} \right\|^{\frac{3}{4}}
\left\| \langle v \rangle^{4-\frac{3(\gamma+2s)}{2}} \nabla^{N_0}_x \mathbf{P}^{\perp} f^{\varepsilon} \right\|^{\frac{1}{4}} \\
\lesssim\;& \sqrt{\widetilde{\mathcal{E}}_{N,l}(t)}
\bigg\{  \left\| \nabla^{N_0-1}_x B^{\varepsilon} \right\|^2
+ \frac{1}{\varepsilon^2}  \sum_{1 \leq \ell \leq 2}\left\| \nabla^{\ell}_x \mathbf{P}^{\perp} f^{\varepsilon} \right\|^2_D
+ \frac{1}{\varepsilon^2}  \left\| \nabla^{N_0}_x \mathbf{P}^{\perp} f^{\varepsilon} \right\|^2_D \bigg\}.
\end{align*}
It should be mentioned that in the above estimate, we need to use the definition of $l$ in \eqref{mainth1 assumption} and the fact $N=N_0+3$.
Thus, collecting all the above estimates of $J_{2,1}$--$J_{2,3}$, we can derive that for $1 \leq k \leq N_0$,
\begin{align*}
J_2 \lesssim \sqrt{\widetilde{\mathcal{E}}_{N,l}(t)}
\bigg\{ \sum_{2 \leq \ell \leq N_0-1}\left\| \nabla^{\ell}_x B^{\varepsilon} \right\|^2
+ \sum_{2 \leq \ell \leq N_0 }\left\| \nabla^{\ell}_x \mathbf{P} f^{\varepsilon} \right\|^2
+ \frac{1}{\varepsilon^2}  \sum_{1 \leq \ell \leq N_0}\left\| \nabla^{\ell}_x \mathbf{P}^{\perp} f^{\varepsilon} \right\|^2_D \bigg\}.
\end{align*}
The terms $J_3$ and $J_4$ can be handled similarly. We thus obtain that for $1 \leq k \leq N_0$,
\begin{align*}
J_3 + J_4 \lesssim \sqrt{\widetilde{\mathcal{E}}_{N,l}(t)}
\bigg\{  \sum_{1 \leq \ell \leq N_0-1}\left\| \nabla^{\ell}_x E^{\varepsilon} \right\|^2
+ \sum_{2 \leq \ell \leq N_0 }\left\| \nabla^{\ell}_x \mathbf{P} f^{\varepsilon} \right\|^2
+ \frac{1}{\varepsilon^2}  \sum_{1 \leq \ell \leq N_0}\left\| \nabla^{\ell}_x \mathbf{P}^{\perp} f^{\varepsilon} \right\|^2_D \bigg\}.
\end{align*}
Finally, we estimate the term $J_5$ for $1 \leq k \leq N_0$. We denote $J_{5,1}$, $J_{5,2}$, $J_{5,3}$, $J_{5,4}$ to be the terms corresponding to the decomposition \eqref{gamma decomposition}.  For the term $J_{5,1}$, we divide our computations into the following two cases.

{\em Case 1}. $k=1$.\;  Making use of the collision invariant property, \eqref{gamma1}, \eqref{Sobolev-ineq}, the H\"{o}lder inequality and the Cauchy--Schwarz inequality, we obtain the following result
\begin{align*}
J_{5,1}=\;& \left( \frac{1}{\varepsilon} \nabla_x \Gamma( \mathbf{P}f^\varepsilon, \mathbf{P}f^\varepsilon),
\nabla_x \mathbf{P}^{\perp}f^\varepsilon \right)\\
\lesssim\;& \frac{1}{\varepsilon} \left\| \mathbf{P}f^\varepsilon \right\|_{L^3_x L^2_v}\left\| \nabla_x \mathbf{P}f^\varepsilon \right\|_{L^6_x L^2_v}
\left\| \nabla_x \mathbf{P}^{\perp}f^\varepsilon \right\|_{D} \\
\lesssim\;& \sqrt{\mathcal{E}_{N}(t)} \bigg\{ \left\| \nabla^2_x \mathbf{P}f^\varepsilon\right\|^2
+\frac{1}{\varepsilon^2} \left\| \nabla_x \mathbf{P}^{\perp}f^\varepsilon\right\|_{D}^2 \bigg\}.
\end{align*}

{\em Case 2}. $2 \leq k \leq N_0$.\; Using the similar method as the case $k=1$, one has
\begin{align*}
J_{5,1}=\;& \left( \frac{1}{\varepsilon} \nabla^k_x \Gamma( \mathbf{P}f^\varepsilon, \mathbf{P}f^\varepsilon),
\nabla^k_x \mathbf{P}^{\perp}f^\varepsilon \right)\\
\lesssim\;& \frac{1}{\varepsilon} \left\| \mathbf{P}f^\varepsilon \right\|_{L^\infty_x L^2_v}\left\| \nabla^k_x \mathbf{P}f^\varepsilon \right\|
\left\| \nabla^k_x \mathbf{P}^{\perp}f^\varepsilon \right\|_{D} \\
&+ \frac{1}{\varepsilon} \sum_{1 \leq j \leq k-1}
\left\| \nabla^j_x \mathbf{P}f^\varepsilon \right\|_{L^6_x L^2_v}\left\| \nabla^{k-j}_x \mathbf{P}f^\varepsilon \right\|_{L^3_x L^2_v}
\left\| \nabla^k_x \mathbf{P}^{\perp}f^\varepsilon \right\|_{D}\\
\lesssim\;&\sqrt{\mathcal{E}_{N}(t)} \bigg\{ \sum_{2 \leq \ell \leq N_0} \left\| \nabla^\ell_x \mathbf{P}f^\varepsilon\right\|^2
+\frac{1}{\varepsilon^2} \sum_{2 \leq \ell \leq N_0} \left\| \nabla^{\ell}_x \mathbf{P}^{\perp}f^\varepsilon\right\|_{D}^2 \bigg\}.
\end{align*}
In a similar fashion, for the term $J_{5,2}$, selecting the first term within the minimum function in \eqref{gamma1} yields
\begin{align}\nonumber
J_{5,2} \lesssim\; \sqrt{\mathcal{E}_{N}(t)}
\frac{1}{\varepsilon^2} \sum_{1 \leq \ell \leq N_0} \left\| \nabla^{\ell}_x \mathbf{P}^{\perp}f^\varepsilon\right\|_{D}^2 .
\end{align}
Opting for the second term inside the minimum function in \eqref{gamma1}, $J_{5,3}$ has the same upper as  $J_{5,2}$.
Regarding $J_{5,4}$, from the collision invariant property, \eqref{gamma1}, \eqref{Sobolev-ineq}, the H\"{o}lder inequality and the Cauchy--Schwarz inequality, we have
\begin{align*}
J_{5,4}=\;& \left( \frac{1}{\varepsilon} \nabla_x^k \Gamma( \mathbf{P}^{\perp} f^\varepsilon, \mathbf{P}^{\perp} f^\varepsilon),
\nabla_x^k \mathbf{P}^{\perp}f^\varepsilon \right)\\
\lesssim\;&\frac{1}{\varepsilon} \sum_{ j \leq k} \int_{\mathbb{R}^3} \left| \nabla_x^j \mathbf{P}^{\perp} f^\varepsilon \right|_{L^2}
\left| \nabla_x^{k-j} \mathbf{P}^{\perp} f^\varepsilon \right|_{D}
\left| \nabla_x^k \mathbf{P}^{\perp} f^\varepsilon \right|_{D} \d x \\
&+\frac{1}{\varepsilon} \sum_{ j \leq k} \int_{\mathbb{R}^3} \left| \nabla_x^j \mathbf{P}^{\perp} f^\varepsilon \right|_{D}
\left| \nabla_x^{k-j} \mathbf{P}^{\perp} f^\varepsilon \right|_{L^2}
\left| \nabla_x^k \mathbf{P}^{\perp} f^\varepsilon \right|_{D} \d x \\
:=\;& J_{5,4}^{\prime}+J_{5,4}^{\prime \prime}.
\end{align*}
where we used \eqref{norm inequality} and the fact $|\cdot|_{L^2_{\gamma/2+s}} \lesssim |\cdot|_{L^2}$
for the case $ \max\{-3, -\frac{3}{2}-2s\} < \gamma < -2s$.
For the term $J_{5,4}^{\prime}$, by the collision invariant property, \eqref{gamma1}, \eqref{Sobolev-ineq}, the H\"{o}lder inequality and the Cauchy--Schwarz inequality, we can obtain that
\begin{align*}
 J_{5,4}^{\prime} \lesssim\;& \frac{1}{\varepsilon}  \left\| \mathbf{P}^{\perp} f^\varepsilon \right\|_{L^\infty_x L^2_{v}}
\left\| \nabla_x^{k} \mathbf{P}^{\perp} f^\varepsilon \right\|_{D}
\left\| \nabla_x^k \mathbf{P}^{\perp} f^\varepsilon \right\|_{D}  \\
&+\frac{1}{\varepsilon}  \sum_{1 \leq j \leq k-1 }\left\| \nabla_x^j \mathbf{P}^{\perp} f^\varepsilon \right\|_{L^3_x L^2_{v}}
\left\| \nabla_x^{k-j} \mathbf{P}^{\perp} f^\varepsilon \right\|_{L^6_x L^2_D}
\left\| \nabla_x^k \mathbf{P}^{\perp} f^\varepsilon \right\|_{D} \\
&+\frac{1}{\varepsilon}  \left\| \nabla_x^k \mathbf{P}^{\perp} f^\varepsilon \right\|
\left\|  \mathbf{P}^{\perp} f^\varepsilon \right\|_{L^\infty_x L^2_D}
\left\| \nabla_x^k \mathbf{P}^{\perp} f^\varepsilon \right\|_{D} \\
\lesssim\;& \sqrt{ \mathcal{E}_{N}(t)}
\frac{1}{\varepsilon^2}\sum_{1 \leq \ell \leq N_0} \left\| \nabla_x^\ell \mathbf{P}^{\perp} f^\varepsilon \right\|^2_{D}.
\end{align*}
Similarly, $J_{5,4}^{\prime \prime}$ has the same upper as $J_{5,4}^{\prime}$.
Therefore, combining all the estimates mentioned above with the a priori assumption \eqref{priori assumption}
and taking the summation over $1 \leq k \leq N_0$, we can conclude \eqref{1-N0 estimate decay} for soft potentials
$\max\{-3,-\frac{3}{2}-2s\} < \gamma < -2s$.

In order to prove \eqref{0-N0 estimate decay}, we estimate \eqref{pure spatial estimate} for $k=0$. Under the a priori assumption \eqref{priori assumption}, by combining \eqref{gamma1} along with \eqref{Sobolev-ineq}, it is easy to derive that
\begin{align}\label{k=0 estimate decay}
\frac{\d}{\d t}  \left\| (f^\varepsilon, E^\varepsilon, B^\varepsilon) \right\|^2
+\frac{\lambda}{\varepsilon^2}  \left\|  \mathbf{P}^{\perp} f^\varepsilon \right\|^2_{D}
\lesssim \sqrt{\widetilde{\mathcal{E}}_{N,l}(t)} \left\{ \left\| E^\varepsilon \right\|^2
+ \left\| \nabla_x \mathbf{P}f^\varepsilon \right\|^2 \right\}.
\end{align}
As a result, for the case $\max\{-3,-\frac{3}{2}-2s\} < \gamma < -2s$, the desired estimate \eqref{0-N0 estimate decay} follows directly by \eqref{1-N0 estimate decay} and \eqref{k=0 estimate decay}.

Then, we turn to the hard potential case $\gamma+2s \geq 0$. In fact,  in this case, the estimates of $ J_{2}$--$J_{5}$ in \eqref{pure spatial estimate} with $1 \leq k \leq N$ can be derived in a similar but simpler way than the case $\max\{-3,-\frac{3}{2}-2s\} < \gamma < -2s$.  The details are omitted for simplicity. This completes the proof of Lemma \ref{0-N0 1-N0 lemma}.
\end{proof}

\begin{lemma}\label{k-N0 decay lemma}
There exist the energy functional $\mathcal{E}^k_{N_0}(t)$ and the corresponding energy dissipation rate functional $\mathcal{D}^k_{N_0}(t)$ satisfying \eqref{negative sobolev energy} and \eqref{negative sobolev dissipation} respectively such that
\begin{align}\label{k-N0 estimate}
\frac{\d}{\d t} \mathcal{E}^k_{N_0}(t)+ \lambda \mathcal{D}^k_{N_0}(t) \leq 0
\end{align}
holds for $k=0,1$ and all $0 \leq t \leq T$.

Furthermore, we can obtain that for $k=0, 1$,
\begin{align} \label{k-N0 decay}
\mathcal{E}^k_{N_0}(t) \lesssim (1+t)^{-(k+\varrho)} \sup_{0 \leq \tau \leq t} \overline{\mathcal{E}}_{N,l}(\tau).
\end{align}
\end{lemma}

\begin{proof}
Similar to Lemma \ref{macroscopic estimate}, there are two interactive energy functionals $\mathcal{E}^{N_0}_{int}(t)$
and $\mathcal{E}^{1 \rightarrow N_0}_{int}(t)$ satisfying
\begin{align*}
\left|\mathcal{E}^{N_0}_{int}(t)\right|
\lesssim  \sum_{k \leq N_0} \left\| \nabla_x^k (f^{\varepsilon}, E^\varepsilon, B^\varepsilon) \right\|^2
\;\text{~and~}\; \left| \mathcal{E}^{1 \rightarrow N_0}_{int}(t) \right|
\lesssim  \sum_{1 \leq k \leq N_0} \left\| \nabla_x^k (f^{\varepsilon}, E^\varepsilon, B^\varepsilon) \right\|^2,
\end{align*}
such that
\begin{align}\label{0-N0 macro 1}
& \frac{\d}{\d t} \mathcal{E}^{N_0}_{int}(t) + \sum_{1 \leq k \leq N_0} \left\| \nabla^k_x \mathbf{P}f^\varepsilon\right\|^2
+ \left \| a_{+}^\varepsilon-a_{-}^\varepsilon \right \|^2
+ \sum_{k \leq N_0-1} \left\|\nabla^k_x E^{\varepsilon}\right\|^2
+ \sum_{1 \leq k \leq N_0-1} \left\|\nabla^k_x B^{\varepsilon}\right\|^2 \nonumber \\
&\lesssim\;\frac{1}{\varepsilon^2} \sum_{k \leq N_0} \left\| \nabla^k_x \mathbf{P}^{\perp}f^\varepsilon\right\|_{D}^2
+\mathcal{E}_{N_0}(t) \frac{1}{\varepsilon^2} \sum_{k \leq N_0}\left\| \nabla^k_x \mathbf{P}^{\perp}f^\varepsilon\right\|_{D}^2,
\end{align}
and
\begin{align} \label{1-N0 macro 1}
& \frac{\d}{\d t} \mathcal{E}^{1 \rightarrow N_0}_{int}(t) \!+ \!\!\sum_{2 \leq k \leq N_0} \left\| \nabla^k_x \mathbf{P}f^\varepsilon\right\|^2
+ \left \| \nabla_x (a_{+}^\varepsilon-a_{-}^\varepsilon) \right \|^2
\!+ \!\!\sum_{1 \leq k \leq N_0-1} \left\|\nabla^k_x E^{\varepsilon}\right\|^2
\!+ \!\!\sum_{2 \leq k \leq N_0-1} \left\|\nabla^k_x B^{\varepsilon}\right\|^2 \nonumber \\
&\lesssim\;\frac{1}{\varepsilon^2} \sum_{1 \leq k \leq N_0} \left\| \nabla^k_x \mathbf{P}^{\perp}f^\varepsilon\right\|_{D}^2
+\mathcal{E}_{N_0}(t) \frac{1}{\varepsilon^2} \sum_{1 \leq k \leq N_0} \left\| \nabla^k_x \mathbf{P}^{\perp}f^\varepsilon\right\|_{D}^2.
\end{align}
Here, we have made use of the a priori assumption \eqref{priori assumption}.
It should be pointed out that the above two estimates are more precise compared to \eqref{macro estimate}.

Therefore, recalling \eqref{low k energy} for $\mathcal{E}^k_{N_0}(t)$, \eqref{k-N0 estimate} for $k=0$ can be derived from \eqref{0-N0 estimate decay}, \eqref{0-N0 macro 1} and the a priori assumption \eqref{priori assumption}.
Similarly, \eqref{k-N0 estimate} for $k=1$ follows from \eqref{1-N0 estimate decay}, \eqref{1-N0 macro 1} and the a priori assumption \eqref{priori assumption}.

To deduce \eqref{k-N0 decay}, for $k=0,1$, one has
\begin{align*}
\left\| \nabla^k_x (\mathbf{P}f^\varepsilon, B^\varepsilon) \right\|
\lesssim \left\| \nabla^{k+1}_x (\mathbf{P}f^\varepsilon, B^\varepsilon) \right\|^{\frac{k+\varrho}{k+1+\varrho}}
\left\| \Lambda^{-\varrho} (\mathbf{P}f^\varepsilon, B^\varepsilon) \right\|^{\frac{1}{k+1+\varrho}}.
\end{align*}
The above inequality, together with the facts that for $m \leq N_0$,
\begin{align*}
\left\| \nabla^m_x \mathbf{P}^{\perp}f^\varepsilon \right\|
\lesssim\;& \left( \frac{1}{\varepsilon} \left\| \langle v \rangle ^{\frac{\gamma+2s}{2}}\nabla^m_x \mathbf{P}^{\perp}f^\varepsilon\right\|\right)^{\frac{k+\varrho}{k+1+\varrho}}
 \left\| \langle v \rangle ^{-\frac{\gamma+2s}{2}(k+\varrho)}\nabla^m_x \mathbf{P}^{\perp}f^\varepsilon\right\|^{\frac{1}{k+1+\varrho}}, \\
\left\| \nabla^{N_0}_x (E^\varepsilon, B^\varepsilon) \right\|
\lesssim\;& \left\| \nabla^{N_0-1}_x (E^\varepsilon, B^\varepsilon) \right\|^{\frac{k+\varrho}{k+1+\varrho}}
\left\| \nabla^{N_0+k+\varrho}_x (E^\varepsilon, B^\varepsilon) \right\|^{\frac{1}{k+1+\varrho}},
\end{align*}
implies that
\begin{align} \nonumber
\mathcal{E}^k_{N_0}(t) \lesssim \left\{\overline{\mathcal{E}}_{N,l}(t)\right\}^{\frac{1}{k+1+\varrho}}
\left\{ \mathcal{D}^k_{N_0}(t) \right\}^{\frac{k+\varrho}{k+1+\varrho}}
\lesssim \left\{\sup_{0 \leq \tau \leq t} \overline{\mathcal{E}}_{N,l}(\tau)\right\}^{\frac{1}{k+1+\varrho}}
\left\{ \mathcal{D}^k_{N_0}(t) \right\}^{\frac{k+\varrho}{k+1+\varrho}}.
\end{align}
Here, we have used $1 < \varrho <\frac{3}{2}$, $k=0,1$ and $N=N_0+3$.
Hence, we deduce that
\begin{align*}
\frac{\d}{\d t} \mathcal{E}^k_{N_0}(t)+ \lambda \left\{\sup_{0 \leq \tau \leq t} \overline{\mathcal{E}}_{N,l}(\tau)\right\}^{-\frac{1}{k+\varrho}}
\left\{ \mathcal{E}^k_{N_0}(t) \right\}^{1+\frac{1}{k+\varrho}} \leq 0.
\end{align*}
Therefore, \eqref{k-N0 decay} can be deduced by solving the above inequality directly. This completes the proof of Lemma \ref{k-N0 decay lemma}.
\end{proof}
\medskip

\subsection{Proof of Global Existence}
\hspace*{\fill}

In this subsection, we complete the proof of Theorem \ref{mainth1}.

\begin{proof}[{\bf Proof of Theorem \ref{mainth1}.}] \
By applying \eqref{a priori estimate 6} and taking the time integration, one has that for any $0 \leq t \leq T$,
\begin{align}\label{existence estimate 1}
\overline{\mathcal{E}}_{N,l}(t) + \lambda  \int_{0}^{t} \overline{\mathcal{D}}_{N,l}(\tau) \d \tau \leq  \overline{\mathcal{E}}_{N,l}(0).
\end{align}
This, together with \eqref{k-N0 decay}, implies that for any $0 \leq t \leq T$,
\begin{align}\label{existence estimate 2}
\sum_{1 \leq |\alpha| \leq N_0}\left\| \partial^\alpha (E^\varepsilon, B^\varepsilon)(t)\right\|^2
\lesssim (1+t)^{-(1+\varrho)} \sup_{0 \leq \tau \leq t} \overline{\mathcal{E}}_{N,l}(\tau)
\lesssim \overline{\mathcal{E}}_{N,l}(0) (1+t)^{-(1+\varrho)}.
\end{align}
Now, recalling the definition of $X(t)$ in \eqref{X define}, it follows immediately from \eqref{existence estimate 1} and \eqref{existence estimate 2} that
$$
X(t) \lesssim \overline{\mathcal{E}}_{N,l}(0)
$$
holds for any $0 \leq t \leq T$. As long as $\overline{\mathcal{E}}_{N,l}(0)$ is sufficiently small, we can close the a priori assumption \eqref{priori assumption}. The rest is to prove the local existence and uniqueness of solutions in terms of the energy norm $\overline{\mathcal{E}}_{N,l}(t)$ and the non-negativity of $F^\varepsilon=\mu+\varepsilon \mu^{1/2}f^\varepsilon$. One can use the iteration method with the iterating sequence $f_n^\varepsilon$ $(n \geq 0)$ on the system
\begin{align*}
	\left\{\begin{array}{l}
\displaystyle \partial_{t} f_{n+1}^{\varepsilon}+\frac{1}{\varepsilon}v \cdot \nabla_{x} {f}_{n+1}^{\varepsilon}-\frac{1}{\varepsilon} (E_{n+1}^{\varepsilon} \cdot v) \mu^{1/2}q_1+\frac{1}{\varepsilon^{2}} L {f}_{n+1}^{\varepsilon}
\\[3mm]
\displaystyle \qquad\qquad= \frac{q_0}{2} E_{n}^{\varepsilon} \cdot v {f}_{n+1}^{\varepsilon}
-q_0 E_{n}^{\varepsilon} \cdot \nabla_{v} {f}_{n+1}^{\varepsilon}
-\frac{1}{\varepsilon} q_0 (v \times B_{n}^\varepsilon) \cdot \nabla_v f_{n+1}^\varepsilon
+\frac{1}{\varepsilon} \Gamma\left({f}_{n}^{\varepsilon}, {f}_{n}^{\varepsilon}\right),  \\ [2mm]
%----------------------------------
\displaystyle \partial_t E_{n+1}^{\varepsilon}-\nabla_x \times B_{n+1}^{\varepsilon}=-\frac{1}{\varepsilon} \int_{\mathbb{R}^3} f_{n+1}^{\varepsilon} \cdot q_1 v \mu^{1/2} \d v, \\ [2.5mm]
%----------------------------------
\displaystyle \partial_t B_{n+1}^{\varepsilon}+\nabla_x \times E_{n+1}^{\varepsilon}=0, \\ [2mm]
%-----------------------------------
\displaystyle \nabla_x \cdot E_{n+1}^{\varepsilon}=\int_{\mathbb{R}^3} f_{n+1}^{\varepsilon} \cdot q_1 \mu^{1/2} \d v, \quad \nabla_x \cdot B_{n+1}^{\varepsilon}=0.
	\end{array}\right.
\end{align*}
The details of proof are omitted for brevity; see also \cite{JL2019, ZMQ1994}. Therefore, the global existence of solutions follows with the help of the continuity argument, and the estimates \eqref{thm1 estimate 1} and \eqref{thm1 estimate 2} hold by the definition of $X(t)$ in \eqref{X define} and \eqref{k-N0 decay}. This completes the proof of Theorem \ref{mainth1}.
\end{proof}
\medskip

\section{Limit to Two-fluid Incompressible NSFM System with Ohm's Law}\label{Limit section}

In this section, our goal is to derive the two-fluid incompressible NSFM system with Ohm's law from the VMB system \eqref{rVMB} as $\varepsilon \to  0$.

\begin{proof}[{\bf Proof of Theorem \ref{mainth2}.}] \
Based on \eqref{thm1 estimate 1} in Theorem \ref{mainth1},  the VMB system (\ref{rVMB})
admits a global solution $(f^\varepsilon, E^\varepsilon, B^\varepsilon) \in L^\infty(\mathbb{R}^+; H^N_x L^2_v)$, so there exists a positive constant $C$ which is independent of $\varepsilon$, such that
\begin{align*}
&\sup_{t\geq 0}  \sum_{|\alpha| \leq N}\left\| \partial^\alpha f^{\varepsilon}(t) \right\|^2 \leq  C, \\
&\sum_{|\alpha|\leq N}\int_0^\infty \left\| \partial^\alpha \mathbf{P}^{\perp}{f}^{\varepsilon}(\tau) \right\|^2_{D}\d \tau\leq C \varepsilon^2,
\end{align*}
which implies that
\begin{align*}
f^{\varepsilon} &\to  f \qquad \quad\;\text{~~weakly}\!-\!* ~\text{in}~ L^\infty(\mathbb{R}^+; H^N_x L^2_v), \\
\mathbf{P}^{\perp}{f}^{\varepsilon} &\to  0 \quad \qquad\;\;\!\text{~~strongly} {\text{~in}~L^2(\mathbb{R}^+;H^N_x L^2_D)}
\end{align*}
as $\varepsilon \to 0$. We thereby deduce from the above convergences that
$ \mathbf{P}^{\perp}f=0$.

As in \cite{JL2019}, we introduce the following fluid variables
\begin{align*}
\begin{split}%\label{limit:1}
&\rho^\varepsilon =\frac{1}{2}\Big\langle f^{\varepsilon}, q_2\mu^{1/2} \Big\rangle,\;\;\;\;
u^\varepsilon= \frac{1}{2} \Big \langle f^{\varepsilon}, q_2 v\mu^{1/2} \Big \rangle,\;\;\;\;
\theta^\varepsilon = \frac{1}{2} \Big\langle f^{\varepsilon}, q_2 \big(\frac{|v|^2}{3}-1 \big)\mu^{1/2}\Big\rangle, \\
&n^\varepsilon = \Big\langle f^{\varepsilon}, q_1\mu^{1/2} \Big\rangle, \;\;\;\;\;\;\;
j^\varepsilon = \frac{1}{\varepsilon} \Big\langle f^{\varepsilon}, q_1 v\mu^{1/2} \Big\rangle, \;\;\;
\omega^\varepsilon = \frac{1}{\varepsilon} \Big\langle f^{\varepsilon}, q_1\big(\frac{|v|^2}{3}-1\big)\mu^{1/2} \Big\rangle.
%%%%%%%%%%%%%%%%%%%%%%%%%%%%%%%%%%%%%%%%%%%%%%%%%%%%%%%%%%%%%%%%%%%%%%%%%%%%%%%%%%%%%%%%%%%%%%%%%%%%%%
\end{split}
\end{align*}
Similar to \eqref{macro equation 3}, we can derive the following local conservation laws
\begin{equation*}
 \left\{
\begin{array}{ll}%\label{limit:2}
\displaystyle
\partial_t \rho^\varepsilon+\frac{1}{\varepsilon}\nabla_x\cdot u^\varepsilon=0,~\\[2mm]
%-----------------------------------------------------------------------------
\displaystyle\partial_t u^\varepsilon+\frac{1}{\varepsilon}\nabla_x(\rho^\varepsilon+\theta^\varepsilon)+\frac{1}{\varepsilon}
\nabla_x\cdot \Big\langle \widehat{A}(v)\mu^{1/2}, \mathcal{L} \big(\frac{f^{\varepsilon} \cdot q_2}{2}\big)\Big\rangle =\frac{1}{2}(n^\varepsilon E^{\varepsilon}+j^\varepsilon \times B^\varepsilon),~\\[2mm]
%-----------------------------------------------------------------------------
\displaystyle\partial_t \theta^\varepsilon+\frac{2}{3\varepsilon}\nabla_x\cdot u^\varepsilon+\frac{2}{3\varepsilon}\nabla_x\cdot \Big\langle \widehat{B}(v)\mu^{1/2},
\mathcal{L} \big(\frac{f^{\varepsilon} \cdot q_2}{2}\big) \Big\rangle=\frac{\varepsilon}{3}j^\varepsilon \cdot E^\varepsilon, \\[3mm]
%-----------------------------------------------------------------------------
\displaystyle\partial_t n^\varepsilon + \nabla_x \cdot j^\varepsilon =0,\\[2mm]
%-----------------------------------------------------------------------------
\displaystyle\partial_t E^\varepsilon - \nabla_x \times B^\varepsilon = -j^\varepsilon,\\[2mm]
%-----------------------------------------------------------------------------
\displaystyle\partial_t B^\varepsilon + \nabla_x \times E^\varepsilon =0,\\[2mm]
%-----------------------------------------------------------------------------
\displaystyle\nabla_x \cdot E^\varepsilon = n^\varepsilon, \;\;
%-----------------------------------------------------------------------------
\nabla_x \cdot B^\varepsilon =0,
%-----------------------------------------------------------------------------
\end{array}
\right.%\tag{+}
\end{equation*}
where
$$A (v)=v\otimes v -\frac{|v|^2}{3}\mathbb{I}_{3\times 3}, \quad B (v)=v\Big(\frac{|v|^2}{2}-\frac{5}{2}\Big),$$
$\widehat{A} (v)$ is such that $\mathcal{L}\big( \widehat{A}(v) \mu^{1/2}\big)= A(v) \mu^{1/2}$ with $\widehat{A}(v) \mu^{1/2} \in \mathcal{N}^{\bot}(\mathcal{L})$,
$\widehat{B} (v)$ is such that $\mathcal{L}\big( \widehat{B}(v) \mu^{1/2}\big)$ $= B(v) \mu^{1/2}$ with $\widehat{B}(v) \mu^{1/2} \in \mathcal{N}^{\bot}(\mathcal{L})$ and $\mathcal{L}$ is the linearized collision operator for the one-species Boltzmann equation, cf. \cite{GS2011}.

By standard convergent method, there are $f$, $E$, $B$, $\rho$, $u$, $\theta$, $n$, $j$, $\omega$ such that
\begin{align*}
f^\varepsilon & \rightarrow f  \quad \quad \qquad \quad \;\;\; \qquad\;\!\text{~~weakly}\!-\!* ~\text{in}~ L^\infty(\mathbb{R}^+; H^N_x L^2_v), \\
(E^\varepsilon,B^\varepsilon, \rho^\e, u^\e, \theta^\e, n^\e )& \rightarrow (E,B, \rho, u, \theta, n )  \quad\quad \!\text{~~weakly}\!-\!* ~\text{in}~ L^\infty(\mathbb{R}^+; H^N_x ), \\
(j^\varepsilon,\omega^\varepsilon) & \rightarrow (j,\omega)  \quad \quad\quad  \qquad\; \;\;\; \;\text{~~weakly} {\text{~in}~L^2(\mathbb{R}^+;H^N_{x})}
\end{align*}
as $\varepsilon \rightarrow 0$, where $f={f}(t, x, v)$ has the form
\begin{align*}%\label{limit:3}
f=\;&\big(\rho+\frac{1}{2}n\big)\frac{q_1+q_2}{2}\mu^{1/2} + \big(\rho-\frac{1}{2}n\big)\frac{q_2-q_1}{2}\mu^{1/2} + u \cdot v q_2 \mu^{1/2} + \theta\big(\frac{|v|^2}{2}-\frac{3}{2}\big)q_2 \mu^{1/2},
\end{align*}
with
\begin{align*}%\label{limit:4}
\rho =\frac{1}{2}\Big\langle f, q_2\mu^{1/2} \Big\rangle,\;
u= \frac{1}{2} \Big \langle f, q_2 v\mu^{1/2} \Big \rangle,\;
\theta = \frac{1}{2} \Big\langle f, q_2 \big(\frac{|v|^2}{3}-1 \big)\mu^{1/2}\Big\rangle, \;
n = \Big\langle f, q_1\mu^{1/2} \Big\rangle.
\end{align*}
In the sense of distributions, by using \eqref{thm1 estimate 1}, Aubin--Lions--Simon Theorem and the similar argument as in \cite{JL2019}, we can obtain that
$$
(u, \theta, n, E, B) \in C (\mathbb{R}^{+} ; H_x^{N-1} ) \cap L^{\infty} (\mathbb{R}^{+} ; H_x^N )
$$
satisfy the following two-fluid incompressible NSFM system with Ohm's law
\begin{equation*}
\left\{\begin{array}{lr}
\displaystyle \partial_t u+u \cdot \nabla_x u-\nu \Delta_x u+\nabla_x p=\frac{1}{2}(n E+j \times B), & \nabla_x \cdot u=0, \\[1mm]
\displaystyle  \partial_t \theta+u \cdot \nabla_x \theta-\kappa \Delta_x \theta=0, & \rho+\theta=0, \\[1mm]
\displaystyle \partial_t E-\nabla_x \times B=-j, & \nabla_x \cdot E=n, \\[1mm]
\displaystyle \partial_t B+\nabla_x \times E=0, & \nabla_x \cdot B=0, \\
\displaystyle j-n u=\sigma\big(-\frac{1}{2} \nabla_x n+E+u \times B\big), & \omega= n \theta,
\end{array}\right.
\end{equation*}
with initial data
$$
u(0, x)=\mathcal{P} u_0(x),\; \theta(0, x)=\frac{3}{5} \theta_0(x)-\frac{2}{5} \rho_0(x),\; E(0, x)=E_0(x),\; B(0, x)=B_0(x).
$$
 We omit the detailed proof for brevity. This completes the proof of Theorem \ref{mainth2}.
\end{proof}
\bigskip

\noindent\textbf{Declarations of competing interest}

The authors declare that they have no conflicts of interest.
\medskip

\noindent\textbf{Data availability}

No data was used for the research described in the article.
\medskip

\noindent\textbf{Acknowledgements}

The research is supported by NSFC under the grant number 12271179,
Basic and Applied Basic Research Project of Guangdong under the grant number 2022A1515012097,
Basic and Applied Basic Research Project of Guangzhou under the grant number SL2022A04J01496.

\end{document}